\documentclass[12pt]{amsart}

\usepackage{amsmath}
\usepackage{amsfonts}
\usepackage{amssymb}
\usepackage{graphicx}
\usepackage{mathrsfs}
\usepackage{pb-diagram}
\usepackage{epstopdf}
\usepackage{amscd}
\usepackage{color}
\usepackage{tensor}
\usepackage[all]{xy}
\usepackage{subcaption}

\newtheorem{theorem}{Theorem}[section]
\newtheorem{lemma}[theorem]{Lemma}
\newtheorem{corollary}[theorem]{Corollary}

\newtheorem{prop}[theorem]{Proposition}
\newtheorem{defn}[theorem]{Definition}

\newtheorem{remark}[theorem]{Remark}

\numberwithin{equation}{section}
\newcommand{\Pabc}{\mathbb{P}^1_{a,b,c}}
\newcommand{\N}{\mathbb{N}}
\newcommand{\Z}{\mathbb{Z}}

\newcommand{\C}{\mathbb{C}}
\newcommand{\Q}{\mathbb{Q}}
\newcommand{\bS}{\mathbb{S}}

\newcommand{\CM}{\mathcal{M}}

\newcommand{\cP}{\mathcal{P}}
\newcommand{\CF}{\mathcal{F}}

\newcommand{\bL}{\mathbf{L}}
\newcommand{\BL}{\mathbb{L}}

\newcommand{\I}{\mathrm{I}}
\newcommand{\II}{\mathrm{II}}

\newcommand{\Jac}{\mathrm{Jac}}

\newcommand{\dd}[1]{\frac{\partial}{\partial #1}}

\newcommand{\WT}[1]{\widetilde{#1}}

\newcommand{\quat}{\mathbb{H}}

\newcommand{\one}{\mathbf{1}}

\newcommand{\bP}{\mathbb{P}}

\newcommand{\pt}{\mathrm{pt}}
\newcommand{\AI}{A_\infty}

\newcommand{\QH}{\mathrm{QH}}

\newcommand{\arc}[1]{\stackrel{\frown}{#1}}

\usepackage{mathtools}

\DeclarePairedDelimiter\floor{\lfloor}{\rfloor}

\newcommand{\KS}{\mathsf{KS}}

\newcommand{\m}{\mathfrak{m}}

\usepackage[colorlinks]{hyperref}

\begin{document}

\author[Amorim]{Lino Amorim}
\address{Department of Mathematics\\ Kansas State University\\ 138 Cardwell Hall, 1228 N. 17th Street\\
	Manhattan, KS 66506\\ USA}
\email{lamorim@ksu.edu}

\author[Cho]{Cheol-Hyun Cho}
\address{Department of Mathematical Sciences, Research institute of Mathematics\\ Seoul National University\\ San 56-1, 
Shinrimdong\\ Gwanakgu \\Seoul 47907\\ Korea}
\email{chocheol@snu.ac.kr}

\author[Hong]{Hansol Hong}
\address{Department of Mathematics \\ Yonsei University \\ 50 Yonsei-Ro \\ Seodaemun-Gu \\ Seoul 03722 \\ Korea} 
\email{hansolhong@yonsei.ac.kr}

\author[Lau]{Siu-Cheong Lau}
\address{Department of Mathematics and Statistics\\ Boston University\\ 111 Cummington Mall\\ Boston \\ MA 02215\\ USA}
\email{lau@math.bu.edu}

\title[Big Quantum cohomology of orbifold spheres]{Big Quantum cohomology of orbifold spheres}

\begin{abstract}
We construct a Kodaira-Spencer map from  the big quantum cohomology of a sphere with three orbifold points to
the Jacobian ring of the mirror Landau-Ginzburg potential function. This is constructed via the Lagrangian Floer theory of the Seidel Lagrangian and we show that Kodaira-Spencer map is a ring isomorphism.
\end{abstract}
\subjclass[2010]{53D45, 53D40}
\maketitle

\vspace{-.5cm}

{
\hypersetup{linkcolor=black, pdfauthor={Name}}
	\tableofcontents
}


\vspace{-1cm}

\section{Introduction}
Orbifold projective lines $\bP^1_{a,b,c}$ are two-dimensional spheres with three orbifold singular points as drawn in Figure \ref{fig:introabc}.  They provide a simple yet very interesting class of geometries. 
Despite low dimensionality, their orbifold Gromov-Witten theory is surprisingly rich. 
Satake-Takahashi \cite{ST} computed the Gromov-Witten invariants and Frobenius structures for elliptic $\bP^1_{a,b,c}$ (where $1/a + 1/b + 1/c = 1$), which involves many interesting number theoretic power series.  Rossi \cite{Rossi} obtained analogous results for spherical $\bP^1_{a,b,c}$ (where $1/a + 1/b + 1/c > 1$).  
More recently, Ishibashi-Takahashi-Shiraishi \cite{IST2} proved that the Frobenius structure from the Gromov-Witten invariants of hyperbolic $\bP^1_{a,b,c}$ (where $1/a + 1/b + 1/c < 1$) is isomorphic to the one from their associated affine cusp polynomials.
In this paper, we provide a geometric approach to study closed-string mirror symmetry for $X=\bP^1_{a,b,c}$ in all three cases, with help of Lagrangian Floer theory. Namely, we will construct a Kodaira-Spencer map from orbifold quantum cohomology of $X$
with bulk deformations
to the Jacobian ring of the mirror potential function and show that it is an isomorphism.

Lagrangian Floer theory has provided a purely mathematical approach to construct and prove mirror symmetry.  A typical example is a compact toric manifold, whose mirror can be nicely constructed from Lagrangian Floer theory.
In the Fano case, the second-named author and Yong-Geun Oh \cite{CO} classified the holomorphic discs bounded by toric fibers and showed that the LG mirror $W$ can be formulated as the count of these discs.  Later Fukaya-Oh-Ohta-Ono \cite{FOOOT,FOOO2,FOOO_MS} used Lagrangian deformation theory to construct the LG mirrors in general.  They also constructed the Kodaira-Spencer map (or closed-open map) which produces closed string mirror symmetry for all compact toric manifolds.  This provides a mirror construction from the first principle, which has the advantage that Kontsevich's homological mirror symmetry conjecture \cite{Kontsevich94} can be canonically derived (See \cite{CHL-toric} for Fano cases).

For $X=\bP^1_{a,b,c}$, the Landau-Ginzburg (LG for short) mirrors $W$ were uniformly constructed in \cite{CHL17} based on Lagrangian Floer theory of a certain immersed Lagrangian $\mathbb{L}$, which was first used by Seidel \cite{Se2}.  Moreover, homological mirror symmetry for the elliptic and hyperbolic cases was derived by a family version of a Yoneda functor naturally coming with the construction.  In the hyperbolic case, the LG mirror is an infinite series in variables $x,y,z$.  \cite{CHKL14} found an inductive algorithm to compute the explicit expressions in all cases. 
 
 In this article, we consider the bulk-deformed version of the Lagrangian Floer theory of $\mathbb{L}$ in $X=\bP^1_{a,b,c}$. And we use it to describe the \emph{big quantum cohomology} of $X=\bP^1_{a,b,c}$. This requires a change in the technical setup, instead of the classical methods in \cite{CHL17}, we are forced to use Kuranishi structures and CF-perturbations \cite{FOOO_new} to handle the various moduli spaces of holomorphic curves. Nevertheless we show the bulk-deformed potentials have the same leading order terms as the ones in \cite{CHKL14}.

In this approach to constructing the mirror, it is crucial to find a {\it large} space of solutions to the weak Maurer-Cartan equation. For the immersed Lagrangian $\mathbb{L}$ we show (see Proposition \ref{prop:weakunobsbulk}) that any linear combination of the odd-degree immersed points gives a solution of the Maurer-Cartan equation. This extends the result in \cite{CHL17} to the case of bulk deformations by orbi-sectors.
   
In order to relate the Gromov-Witten invariants of $\bP^1_{a,b,c}$ with  the Jacobian ring of the bulk-deformed mirror potential,
we use the method of Kodaira-Spencer map invented by \cite{FOOO_MS}, which gives a homomorphism from the quantum cohomology of $X$ to the Jacobian ring of the mirror $W_\tau$.  The following is our main theorem.

\begin{theorem} \label{thm:main}
	Let $X=\bP^1_{a,b,c}$ and $W_\tau$ be its bulk-deformed disc potential by $\tau \in H^* (X, \Lambda_+)$.  Let $\Jac(W_\tau)$ be the completed Jacobian ring over the Novikov field $\Lambda$ in a certain choice of coordinates.  Denote the big quantum cohomology of $X$ over $\Lambda$ with quantum product $\bullet_\tau$ by $\QH^*_{orb}(X,\tau)$.  The Kodaira-Spencer map $\KS_\tau: \QH^*_{orb}(X,\tau) \to \Jac(W_\tau)$ is a ring isomorphism.
\end{theorem}

The full closed-string mirror symmetry conjecture predicts an isomorphism of Frobenius manifolds. Therefore one expects that the map $\KS_\tau$ identifies the Euler vector field on the big quantum cohomology with the Euler vector field on $\Jac(W_\tau)$, which is the class $[W_\tau]$. We show that this does hold in our case, see Theorem \ref{thm:KSofdiv}. Moreover the $\KS_\tau$ map intertwines the Poincar\'e pairing with the residue pairing on $\Jac(W_\tau)$ determined by some volume form. An algebraic (B-side) description of this volume form is however not yet available.

Theorem \ref{thm:main} should also be helpful in proving the Homological Mirror Symmetry conjecture in this setting. More concretely it should imply, assuming a generalization of Abouzaid's generation criterion to the orbifold setting, that the Lagrangian $\mathbb{L}$ equipped with its bounding cochains split-generate the Fukaya category of $\bP^1_{a,b,c}$.

The construction of Kodaira-Spencer map  \cite{FOOO_MS} crucially depends on the existence of $T^n$-action, hence
the definition is still missing in general cases. The above theorem provides the first class of examples of Kodaira-Spencer map beyond toric manifolds.

In fact, there is a crucial difference between our case of $\bP^1_{a,b,c}$ and that of toric manifolds. Namely, we need to enlarge the domain of LG potential to make the above theorem hold true. Maurer-Cartan formalism of Lagrangian Floer theory provides a natural set of coordinates $\tilde{x},\tilde{y},\tilde{z} \in \Lambda_0$. Namely, they are the coordinates of the Maurer-Cartan space which are dual to the immersed sectors of $\BL$. Given the bulk deformed mirror potential $W_\tau(\tilde{x},\tilde{y},\tilde{z})$, one can define
the Jacobian ring as in Definition \ref{def:jac} as the completed power series ring $\Lambda \ll \tilde{x},\tilde{y},\tilde{z} \gg$
modulo Jacobian ideal of $W_\tau(\tilde{x},\tilde{y},\tilde{z})$. With this Jacobian ring, $\KS_\tau$ is not an isomorphism in general
hence the above theorem fails. In Section \ref{subsec:valcritpt1}, we give an explicit counter-example.

In this paper, we will make the change of variables
\begin{align}\label{eqn:chvar1}
\left\{\begin{array}{lcl}
x&=&T^3 \tilde{x},\\
y&=&T^3 \tilde{y},\\
z&=&T^3 \tilde{z}.
\end{array}\right.
\end{align}
and consider $x,y,z \in \Lambda_0$. In terms of old variables, this is equivalent to allowing
$$val(\tilde{x}),val(\tilde{y}), val(\tilde{z}) \geq -3.$$

In terms of non-archimedean norm $e^{-val}$, $\tilde{x},\tilde{y},\tilde{z}$ are functions on a disc $D(1)$ of radius $1 =e^0$,
and $x,y,z$ are functions on a disc $D(e^3)$ which contains $D(1)$. In the above counter example, critical points of the potential   $W_\tau(\tilde{x},\tilde{y},\tilde{z})$ lie on $D(e^3) \setminus D(1)$ as shown in Proposition \ref{prop:critescape22r}. Thus we need the bigger disc $D(e^3)$ to match the number of critical points with the rank of the quantum cohomology ring. See \ref{subsec:valcritpt1} for related discussions. 

However, this necessary enlargement of domain is the main source of complication almost in every steps of the proof of the main theorem.
Namely, Lagrangian Floer theory for bounding cochains of negative valuation does not work in general and we need to take care of convergence issues in each step of the proof.

We give another perspective of the above coordinate change.  
For readers convenience, we first recall the case of toric manifolds briefly. For a compact toric n-fold, which can be understood as a compactification of $\C^n$, $W$ takes the form
$$ z_1 + \ldots + z_n + \sum_{i=1}^\infty T^{A_i} Z_i + h.o.t. $$
where $A_i > 0$ and $Z_i$ are monomials in $z_1,\ldots,z_n$, and $h.o.t.$ consists of higher-order terms in $T$.
Under the Kodaira-Spencer map, the images of the toric divisors $D_1,\ldots,D_n$, which are compactifications of the coordinate hyperplanes of $\C^n$, are sent to $z_1,\ldots,z_n$, which generate (a suitable completion of ) $\Lambda[z_1,\ldots,z_n]$ and hence the Jacobian ring.  Thus surjectivity of the Kodaira-Spencer map is automatic in this case.

On the other hand for $\bP^1_{a,b,c}$, the potential $W$ (with $\tau=0$) takes the form
$$W(\tilde{x},\tilde{y},\tilde{z}) = -T \tilde{x}\tilde{y}\tilde{z} + T^{3a} \tilde{x}^a + T^{3b} \tilde{y}^b + T^{3c} \tilde{z}^c + h.o.t. $$
whereas, in new coordinates, the leading terms of the above become
$$W_{lead}:=-T^{-8} xyz + x^a +y^b + z^c.$$ 
The images of the orbifold points $[1/a],[1/b],[1/c]$ are
$T^3 \tilde{x}, T^3 \tilde{y}, T^3 \tilde{z}$ respectively, but in new coordinates these orbifold points map to $x,y,z$ which generates $\Lambda \langle\langle x,y,z\rangle\rangle$. This is one of key ingredient in proving surjectivity of  the KS map. Therefore the coordinate change is also quite natural in this perspective as well.

Once we establish surjectivity of $\KS_\tau$, we match the dimension of the Jacobian ring of the bulk-deformed potential with that of $\QH^*_{orb}(X,\tau)$ to show that $\KS_\tau$ is injective, where the former is given as $a+b+c-1$. For this, we argue with the deformation invariance of the dimension, as it is relatively easy to analyze the leading order terms. In fact, the rank of the Jacobian ring for $-T^{-8} xyz + x^a +y^b + z^c$ is already quite nontrivial, as one needs to additionally take into account the convergence issue when working over $\Lambda$. For this reason, the computation for leading order terms is somewhat lengthy which we will see in Appendix \ref{sec:thmasdefpfpf}. Then we prove that the leading order terms and the actual potential can be interpolated by a flat deformation. This involves a delicate induction step together with some nontrivial algebraic facts.


%


While the necessity of the coordinate change is now clear, it results in the analytic difficulty that we need to insure convergence throughout the construction under this coordinate change, which a priori is not at all obvious.
Even though the construction in Floer theory has automatic  $T$-adic convergence for bounding cochains in $\Lambda_+$, this coordinate change has an effect that our bounding cochains lie in $\Lambda_{\geq -3}$.  Hence we need a better control in areas to have convergence.  First we will show that in the coordinates $x,y,z$, every term of $W_\tau$ has non-negative valuation (Lemma \ref{lem:geq0}).  Then we use an orbifold version of Gauss-Bonnet theorem (Theorem \ref{thm:GB}) to show that $W_\tau$ actually converges in $T$-adic topology.

\begin{theorem}[Theorem \ref{thm:bdconv}]
	The bulk-deformed potential $W_\tau$ is a convergent series in new variables $x,y,z$ as in \eqref{eqn:chvar1}, that is, it is an element of $\Lambda\langle\langle x,y,z\rangle\rangle$.
\end{theorem}

In the orbifold setting, the twisted sectors have fractional degrees.  For $X=\bP^1_{a,b,c}$, $H^{<2}_{orb}(X)$ is spanned by the fundamental class $\one_X$ and the twisted sectors $[i/a],[j/b],[k/c]$ for $0<i<a$, $0<j<b$, $0<k<c$.
The compatibility of $\KS_\tau: \QH^*_{orb}(X,\tau) \to \Jac(W_\tau)$ with ring structures follows from the standard cobordism argument as in \cite{FOOO_MS}, but it still requires a careful analysis on the associated virtual perturbation scheme in our context. The details will be provide in \ref{subsec:ringhomkura}.

The main theorem is particularly interesting in the hyperbolic case, which belongs to the class of general-type manifolds whose mirror symmetry is mostly conjectural.  Theorem \ref{thm:main} together with the result in \cite{CHKL14} provides the first class of manifolds in general-type whose small quantum cohomology has a presentation which can be explicitly computed.  Even in the toric case, $W$ is a highly non-trivial series due to obstructed non-constant sphere bubbling with negative Chern number.  There is no general algorithm to compute $W$ for toric manifolds of general type.  On the other hand, for hyperbolic $\bP^1_{a,b,c}$ with no bulk deformation (that is $\tau=0$), there is an algorithm to compute the series $W_\tau$ by \cite{CHKL14}, which in turn gives an explicit presentation of the small quantum cohomology $\QH^*_{orb}(X,0)$.  

Finally in the last section, we exhibit several interesting properties of the bulk-deformed potential as well as a few explicit calculations for $\KS_\tau$. Most importantly, we show that the bulk-deformation of the Floer theory of $\mathbb{L}$ produces a \emph{versal} deformation of the mirror potential. More specifically,
\begin{theorem}[Theorem \ref{thm:versality}]
	Consider $P \in \Lambda\langle\langle x,y,z\rangle\rangle$ with $val(P-W_{lead})>0$ where $W_{lead}=-T^{-8} xyz + x^a +y^b + z^c$.
	Then there exist $\tau'\in H^*_{orb}(\mathbb{P}^1_{a,b,c}, \Lambda_0)$ and a coordinate change $(x',y',z')$ such that 
	$$P(x',y',z')=W_{\tau'}.$$
\end{theorem}
\noindent Note that this is analogous to the versality statement in toric case proven in \cite[Theorem 2.8.1]{FOOO_MS}. The proof is based on the induction argument on energy, which is similar to the one used to establish surjectivity of $\KS_\tau$.

%

The organization of the paper is as follows. In Section \ref{sec:bdfloerofseidellag}, we review Floer theory of the Lagrangian $\mathbb{L}$ in $\bP^1_{a,b,c}$ and its bulk-deformation including orbi-sectors. In Section \ref{sec:MC}, we prove the weakly unobstructedness of $\mathbb{L}$ after the bulk-deformation, and in Section \ref{sec:bdpotchangevar}, we study the resulting bulk-deformed potential and its convergence after coordinate change. In Section \ref{sec:ptdependence}, we prove that the bulk-deformed potential changes by an explicit coordinate change for different choices of cohomology representatives, and hence its well-definedness follows. Throughout Section \ref{sec:KSringhom} and \ref{sec:KSisomsurjinj}, we show that $\KS_\tau$ is a ring homomorphism that is surjective and injective, which proves our main theorem. Finally, we provide some concrete calculations of $\KS_\tau$, and prove the versality theorem in Section \ref{sec:applicationcal}.

\subsection*{Acknowledgments}
The authors express their gratitude to Kenji Fukaya and Yong-Geun Oh for useful discussions on virtual perturbation schemes. 
C.-H. Cho was supported by the NRF grant funded by the Korea government(MSIT) (No.
2017R1A22B4009488). H. Hong and C.-H. Cho are supported by the NRF grant funded by the Korea government(MSIT) (No.2020R1A5A1016126).
S.-C. Lau is partially funded by Travel Support for Mathematicians of Simons Foundation.

\section{\texorpdfstring{Bulk deformed Floer theory of Seidel Lagrangian in $\bP^1_{a,b,c}$}{Bulk deformed Floer theory of Seidel Lagrangian in P1,a,b,c}}\label{sec:bdfloerofseidellag}
In this section, we recall orbifold quantum cohomology and immersed Lagrangian Floer theory mainly to set the notations.
In short, we will consider orbifold quantum cohomology by Chen-Ruan \cite{CR} and  a de Rham version of immersed Lagrangian Floer theory (defined by Akaho-Joyce \cite{AJ} and Fukaya\cite{F17}). One can enhance the latter by including orbi-discs following the work of the second author and Poddar \cite{CP}. This gives bulk deformations by twisted sectors.


%
\subsection{\texorpdfstring{ $\bP^1_{a,b,c}$ and its orbifold quantum cohomology}{ P1,a,b,c and its orbifold quantum cohomology}}
Let $\bP^1_{a,b,c}$ be an orbifold sphere with three orbifold points with isotropy groups $\Z/a$, $\Z/b$, $\Z/c$, where $a,b,c\geq 2$. We take the K\"ahler form $\omega$ descended from the universal cover of $\bP^1_{a,b,c}$ with constant curvature.  For later convenience we scale it such that the total area of $\bP^1_{a,b,c}$ is $8$.
The orbifold Euler characteristic is given by
$$\chi\left(\bP^1_{a,b,c}\right)=\frac{1}{a}+\frac{1}{b}+\frac{1}{c}-1.$$
	
Depending on $\chi$ being positive, zero or negative, the universal cover of $\bP^1_{a,b,c}$ is $\bP^1$, $\C$ or the hyperbolic upper half plane $\quat^2$ (with the standard complex structure). We refer to these as the spherical, elliptic and hyperbolic cases, respectively. In all cases, $\bP^1_{a,b,c}$ can be constructed as a global quotient of a Riemann surface $\Sigma$ by a finite group. In the spherical case $\Sigma$ is a sphere, in the elliptic case $\Sigma$ is an elliptic curve and in the hyperbolic case $\Sigma$ is a surface of genus $\geq 2$.

Recall that the Chen-Ruan orbifold cohomology of an orbifold $X$, as a vector space, is given by
the direct sum of singular cohomology groups of the inertia orbifolds. Namely, let 
$IX = \{(p,(g)_{G_p}): p\in X, g\in G_p\}$, where $G_p \subset G$ is the isotropy subgroup of $p$ and $(g)_{G_p}$ is the conjugacy class of $g \in G_p$.  Its connected components $X_{(g)}$ are called inertia orbifolds, and the indices $(g)$ are called the twisted sectors.  Then for $d\in \Q$,
$$H^{d}_{orb}(X) := \bigoplus_{(g)}  H^{d- 2\iota_{(g)}}(X_{(g)}).$$  
The degree-shifting $\iota_{(g)} \in \Q$ is called the age of the twisted sector in literature.  See \cite{CR04} for more detail.

For $H^*_{orb}(\Pabc)$, we have the cohomology classes $\one_{X}, [\pt] \in H^2(\Pabc,\mathbb{R})$, as
well as the twisted sectors 
\begin{equation}\label{eq:twisted}
\floor*{\frac{1}{a}}, \ldots, \floor*{\frac{a-1}{a}}, \floor*{\frac{1}{b}}, \ldots, \floor*{\frac{b-1}{b}}, \floor*{\frac{1}{c}}, \ldots, \floor*{\frac{c-1}{c}}
\end{equation}
where $\floor*{\frac{k}{a}}$ has degree $\frac{2k}{a}$. Let us denote by $H^{tw}(X)$ the span of the twisted sectors.

By local computations, the classical part of Chen-Ruan product of $\floor*{\frac{j}{a}}$ and $\floor*{\frac{k}{a}}$
is $\floor*{\frac{j+k}{a}}$ if $j+k < a$, and is $\frac{1}{a}[\pt]$ if $j+k=a$ and zero otherwise.
These are the products from constant orbi-spheres.
There are non-trivial contributions from
non-constant orbi-spheres as well for the quantum cohomology.  They can be written as follows via the orbifold Poincar\'e pairing:
$$\langle \one_{X},[\pt] \rangle_{PD_X} = 1,  \ \ \ \langle \floor*{\frac{j}{a}}, \floor*{\frac{a-j}{a}} \rangle_{PD_X} = \frac{1}{a}.$$

Fix $\tau \in H^*_{orb}(X, \Lambda_+)$ and for each $A,B \in H(X,\Lambda_0)$
 the bulk deformed quantum product $A \bullet_\tau B$ is defined by 
$$ \langle A \bullet_\tau B , C \rangle_{PD_X} = \sum_{l=0}^\infty \frac{1}{l!}GW_{l+3}(A,B,C,\tau,\cdots,\tau).$$
where $GW_{l+3}$ is the orbifold Gromov-Witten invariant with $l+3$ inputs (\cite{CR}). The above sum converges over $\Lambda$ by our choice of $\tau$. We denote the resulting Frobenius algebra by $\QH^*_{orb}(X,\tau)$, the quantum cohomology of $X$.
Readers are referred to \cite{IST2, Rossi, ST} for the computation on the quantum cohomology of $\mathbb{P}^1_{a,b,c}$.

\subsection{Immersed Lagrangian Floer theory}
Immersed Lagrangian Floer theory was introduced by Akaho-Joyce \cite{AJ} using singular chains, extending the embedded case
of Fukaya, Oh, Ohta Ono \cite{FOOO}.
A different version using Morse function (and pearl complex) was given by Seidel \cite{Se} and Sheridan \cite{Sheridan11,Sh}.
In our previous work \cite{CHL17}, we used the definition by Seidel to prove homological mirror symmetry.
In this paper, we work with de Rham version of  immersed Lagrangian Floer theory (by  Fukaya \cite{F17}) since
we use Kuranishi structures to deal with orbifold quantum cohomology.
We refer the readers to the above references for general definitions.

Seidel \cite{Se} constructed an immersed circle $\BL: \bS^1\looparrowright \bP^1_{a,b,c}$ with three transversal (double) self-intersections (see Figure \ref{fig:introabc}), and we refer to it as the Seidel Lagrangian.  
We assume that the image of $\BL$ is invariant under reflection with respect to the equator (which passes through the three orbifold points), which is crucial for weakly unobstructedness in the next section.  The image of $\BL$ and the equator divide the sphere into eight regions: two triangles and six bigons. We take $\BL$ such that each of these regions have area $1$. The two triangles provide the leading order terms for our calculation of the $A_\infty$-structure later, and will be called \emph{minimal triangles} from now on for the obvious reason.
The Lagrangian $\BL$ is equipped with a non-trivial spin structure (this is needed for weakly unobstructedness in Section \ref{sec:MC}).  This is given by fixing a point in $\BL$ (which is not an immersed point) and any holomorphic disc contribution through this point gets a (-1) sign for each $\AI$-operation.

\begin{figure}[h]
\begin{center}
\includegraphics[height=2in]{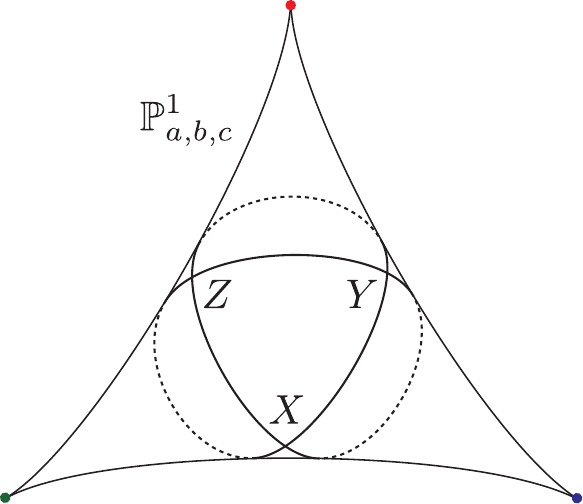}
\caption{}\label{fig:introabc}
\end{center}
\end{figure}

One associates to the Seidel Lagrangian its Fukaya algebra $\CF(\BL)$, which is a filtered $A_{\infty}$-algebra with the underlying $\mathbb{Z}_2$-graded vector space
\begin{equation}\label{eq:CF_def}
	\CF^*(\BL):=\left(\Omega^*(\bS^1)\oplus\bigoplus_{X,Y,Z}\Lambda_0^{\oplus 2}\right)\hat{\otimes}\Lambda_0.
\end{equation}
Here $\Omega^*(\bS^1)$ is the classical de Rham cochain algebra of $\bS^1$ (the domain of $\BL$) with coefficients in $\C$. Each of the intersection points gives rise to two generators in $\CF(\BL)$, one even and one odd. We denote by $X$, $Y$ and $Z$ the odd ones and by $\bar{X}$, $\bar{Y}$ and $\bar{Z}$ the  even generators. 

The $\AI$-operations are defined using the moduli space of pseudo-holomorphic polygons as in Fukaya \cite{F17}, to which we refer readers for details.

In \cite{CHL17} bulk deformations by twisted sectors were not considered. In that setting the $\AI$-operation can be defined using the signed count of rigid immersed polygons in $\bP^1_{a,b,c}$ with prescribed (convex) corners.
In this case, by automatic regularity of (holomorphic) polygons in Riemann surface (see  \cite[Part II, Section 13]{S08}), they are already transversal, and hence it is legitimate to use them for counting. Moreover there exists a compact (smooth) Riemann surface and a covering map to $\bP^1_{a,b,c}$ (which is ramified at exactly the three orbifold points with the corresponding orders), such that the lifts of $\BL$ in the cover are embedded Lagrangians.  Thus one can count polygons in the cover to define Lagrangian Floer theory. 

In this paper however, we consider bulk deformations by twisted sectors and therefore this approach is not available and we follow Fukaya-Oh-Ono-Ohta to define $\AI$-operations using pull-back and push-forward of differential forms over the moduli spaces. In order to do this one needs the technique of  \emph{continuous family of multi-sections}, also called \emph{CF-perturbations}. We give a brief overview of this in the next subsection and refer the reader to \cite{FOOO_new} for a detailed account of this technique.

\subsection{Orbi-discs and Lagrangian Floer theory for orbifolds}\label{subsec:generalorbmod}
We recall how to incorporate orbi-discs into the story. In our case, the Seidel Lagrangian stays away from the orbifold points of $\bP^1_{a,b,c}$, which can be  handled as in the case of toric orbifolds \cite{CP}.



Let us first recall the definition of an orbi-disc, adapted for an immersed Lagrangian boundary condition.
Let $T$ be the index set of inertia components of $X$, where $0 \in T$ corresponds to the underlying topological space of $X$.  Let $R$ be the index set of the immersed sectors of $\BL$, where $+ \in R$ corresponds to the underlying immersed Lagrangian.

\begin{defn} \label{orbi-disc moduli}
	Let $\beta \in H_2(X,\BL)$ be a disc class, $\gamma:\{0,\ldots,k\} \to R$ a specification of immersed sectors of $\BL$ and $\nu:\{1,\ldots,l\} \to T$ a specification of twisted sectors of $X$.  The moduli space $\CM^{main}_{k+1,l} (\beta;\nu;\gamma)$ consists of elements of the form $(\Sigma, \vec{z}^+, \vec{m}, \vec{z}, u)$ such that
	\begin{itemize}
		\item $(\Sigma, \vec{z}^+, \vec{m})$ is a (prestable) bordered orbifold Riemann surface with genus zero, where $\vec{z}^+ = (z^+_1,\ldots,z^+_l) \in (\Sigma - \partial \Sigma)^l$ is a sequence of interior orbifold marked points which are not (orbi-)nodes, and $\vec{m} = (m_1,\ldots,m_l) \in \N^l$ specifies the multiplicities of the uniformizing chart at these orbifold points.
		
		\item $u:(\Sigma, \partial \Sigma) \to (X,\BL)$ is a holomorphic map on each component, that is, $u$ is a continuous map which is holomorphic in the interior of $\Sigma$ away from the orbifold points, and around each orbifold point $z^+_i$, $u$ can be locally lifted to be a holomorphic map from the uniformizing chart at $z^+_i$ to a uniformizing chart of $X$ at $f(z^+_i)$.  Moreover, $z^+_i$ is mapped to the twisted sector $X_{\nu(i)}$ for $i = 1,\ldots,l$.
		
		\item $u$ is good and representable as an orbifold morphism.
		
		\item $\vec{z} = (z_0,\ldots,z_{k}) \in (\partial \Sigma)^{k+1}$ is a sequence of boundary marked points obeying the cyclic ordering of $\partial \Sigma$.  Moreover, $z_i$ is mapped to the immersed sector labeled by $\gamma(i)$ for $i=0,\ldots,k$.
	\end{itemize}
\end{defn}

In the case of toric orbifolds, such an orbi-disc with boundary on a Lagrangian torus fiber was studied and classified.  The orbi-disc potential for a toric Calabi-Yau orbifold or a compact semi-Fano toric orbifold was computed using the mirror map in \cite{CCLT13,CCLT12}.

To state the dimension formula for the moduli spaces, we use two related notions of the Maslov index. 
The first one is the desingularized Maslov index $\mu^{de}$ (following \cite{CR}). Given an orbi-disc with a Lagrangian boundary condition, the pull-back bundle data is an orbi-bundle over the domain of the orbi-disc together with a Lagrangian sub-bundle over the boundary of the disc.  This bundle cannot be trivialized due to the non-trivial orbifold structure.  On the other hand there is an associated smooth bundle,
called the desingularized bundle, which has the same set of local holomorphic sections. The latter property enables us to compute the virtual dimension.
The other one is the Chern-Weil Maslov index  $\mu_{\mathrm{CW}} $.
It was shown in Proposition 6.10 of \cite{CS} that 
\begin{equation}\label{decw}
\mu^{de} = \mu_{\mathrm{CW}} - 2 \sum_i \iota(\nu(i))
\end{equation}
where $\iota(\nu(i))$ is the degree shifting number associated to the twisted sector labeled by $\nu(i)$.

Let us also explain how to  handle $J$-holomorphic polygons with transversally intersecting Lagrangian boundary conditions.
Given two Lagrangian subspaces $L_i, L_{i+1} = J \cdot L_i$ in a symplectic vector space with a compatible linear complex structure $J$, we have the path of Lagrangian subspaces from $L_i$ to $L_{i+1}$
given by $e^{\pi Jt/2}L_i$ for $t \in [0,1]$ (which is called a positive path). Given a holomorphic polygon, we get a loop of Lagrangians along the boundary by concatenating with these positive paths at each corner. The resulting Maslov index of the Lagrangian loop is called the topological Maslov index. If there is in addition an orbifold
point in the interior, we can first desingularize it as above, and concatenate with positive paths to define the topological Maslov index, which is also denoted by $\mu^{de}$. 
\begin{remark}\label{rem:cw}
For the definition of the Chern-Weil index, we choose a unitary connection, which asymptotically sends $L_i$ to $JL_i$ at the puncture along the positive path.
Then the relation \eqref{decw} also holds for polygons. We remark that there is an error in \cite{CS} Proposition 5.6. Namely, the formula (23) in \cite{CS} holds true for connections which are trivial near the puncture, but it does not hold for general connections. Rather we have \eqref{decw}  with asymptotic conditions
given by positive paths.
\end{remark}

It is well-known that the Fredholm index of the $\overline{\partial}$ operator on discs with smooth Lagrangian boundary condition equals
$n + \mu^{de}$. For transversely intersecting Lagrangians, we can glue orientation operators of positive paths (where each path has index $n$) at the punctures to obtain a formula 
$$ Ind(\overline{\partial})  + (k+1) n =  n + \mu^{de}$$
(If we had used negative paths instead of positive paths to define the topological Maslov index, the term $(k+1)n$ will disappear.
In this sense, it would be more convenient to use negative paths.  We follow the usual convention to use positive paths.)

Hence the dimension of the moduli space of $J$-holomorphic polygons are given by (adding the effects of $l$ interior and $k+1$ boundary marked points and equivalences)
$$ Ind(\overline{\partial})  + 2l + (k+1) -3 =  n+ \mu^{de} -(k+1)n + 2l+ k- 2$$
In our case of $n=1$ the moduli space $\CM_{l,k+1} (\beta;\nu;\gamma)$
has virtual dimension $$\mu^{de} + 2l-2= \mu_{\mathrm{CW}} + 2 \sum_j (1 -  \iota(j))-2$$
For $l=0$,  it is simply given by $\mu^{de} -2$. 


\begin{figure}
	\begin{center}
		\includegraphics[height=2in]{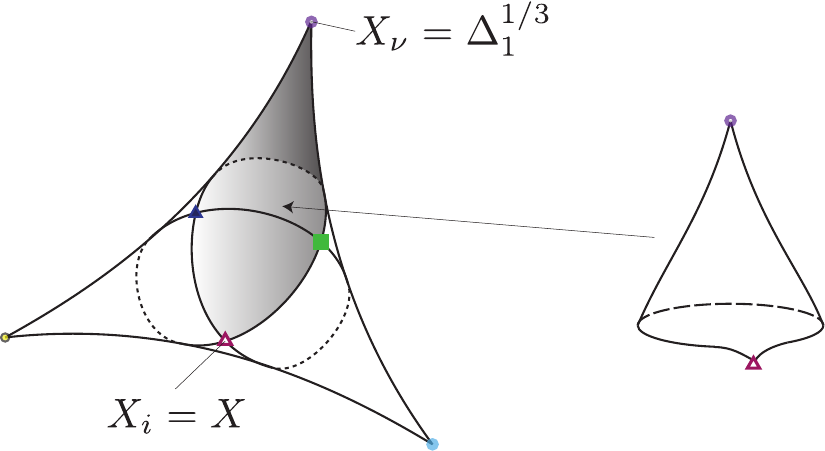}
		\caption{Image of $[1/3]$ orbi-discs in the quotient space $\mathbb{P}^1_{3,3,3}$}\label{fig:orbdisk_config}
	\end{center}
\end{figure}



\subsection{Bulk deformed Fukaya algebra}\label{subsec:bdeffukalg}
The Fukaya algebra with bulk deformations by twisted sectors can be defined as follows.
First, we see how to adapt the definition of $\mathfrak{q}$ operator to our orbifold setting. Let ${\bf T}_1,\ldots, {\bf T}_m$ denote the twisted sectors in (\ref{eq:twisted}). For each multi-index $I=(i_1,\ldots,i_l)$ let $\nu_I$ be defined by $\nu_I(z^+_j)={\bf T}_{i_j}$ which specifies the twisted sector at each interior point $z^+_j$. We denote the corresponding moduli space $\mathcal{M}^{main}_{k+1,l}(\beta, \nu_I,\gamma)$ as $\mathcal{M}^{main}_{k+1,l}(\beta, \mathbf{T}_I,\gamma)$.
This space has a Kuranishi structure, and we can take a system of \emph{CF-perturbations} which are transversal to the zero-section, compatible with the other perturbations given at the boundaries as in \cite{FOOO_MS}, \cite{F17} or \cite{FOOO_new}.
There is an evaluation map 
$$ev^I: \mathcal{M}^{main}_{k+1,l}(\beta, \mathbf{T}_I,\gamma)  \to \prod_{i=1}^k L(\gamma(i))$$
as well as $ev^I_0$ where $L(\gamma(i))$ is the corresponding immersed sector for $i \neq +$, and $L(+) = \BL$. The system of \emph{CF-perturbations} can also be chosen such that $ev^I_0$ is \emph{strongly submersive} \cite{FOOO_new}.
Using these perturbations we can define 
\begin{equation}\label{eq:qmap}
\mathfrak{q}_{l,k,\beta}(\mathbf{T}_I;h_1 \otimes \cdots \otimes h_k) = (ev^I_0)_*(ev^I)^*(h_1 \times \cdots \times h_k)
\end{equation}
$$\mathfrak{q}_{l,k}^\rho (\mathbf{T}_I;h_1 \otimes \ldots \otimes h_k) = \sum_{\beta} T^{\beta \cap \omega /2\pi} \rho(\partial \beta) \mathfrak{q}_{l,k,\beta}(\mathbf{T}_I; h_1 \otimes \cdots \otimes h_k). $$
Please note that the push-forward operation above depends on the choice of CF-perturbations on the  moduli spaces $\mathcal{M}^{main}_{k+1,l}(\beta, \nu_I,\gamma)$ - see \cite[Appendix C]{FOOOT} or \cite{FOOO_new} for a more detailed account. We will omit this dependence from our notation.
The way to handle the unitary line bundle $\rho$ (on $\BL$) is very standard, and we will omit the superscript $\rho$ from now on.

Given a cohomology class $[\tau] \in H^*_{orb}(X;\Lambda_0)$ we pick a representative $\tau=\tau^0 \one_{X}+ \tau^2 \mathbf{p}+ \tau_{tw}$, where $\mathbf{p}$ is a $\Z_2$-invariant cycle (away from $\BL$ and the orbi-points) representing $[pt] \in H^2 (X,\Lambda_0)$ and $\tau_{tw}=\sum_k \tau_k \mathbf{T}_k$.

The Seidel Lagrangian is invariant under $\iota$ the (anti-symplectic) reflection on the equator that contains the orbifold points as well as the self-intersections. By the above choice the bulk term $\tau$ is also invariant under this reflection. Therefore there is a automorphism of the various moduli spaces 
		\[ \iota_*: \mathcal{M}^{main}_{k+1,l}(\beta, \mathbf{T}_I,\gamma) \to \mathcal{M}^{main}_{k+1,l}(\beta, \mathbf{T}_I,\gamma). 
		\]
	On a holomorphic disk the map $\iota_*$ is given by pre-composing with conjugation (in $\mathbb{C}$) and post-composing with the reflection on the equator. 
	This construction extends to the compactified moduli spaces and is, in fact, a automorphism of spaces with Kuranishi structure \cite{FOOOanti}. Thus it defines a $\Z/2$ action on the moduli spaces. We will further require the the CF-perturbations are $\Z/2$-equivariant. As with any finite group action this is always possible, see \cite[Section 24.2]{FOOO_new} for example. This last compatibility is essential to show unobstructedness of the Seidel Lagrangian.

We define 
\begin{equation}\label{eq:mktau}
\m_k^{\tau}(h_1,\cdots,h_k) = \sum_\beta \exp( \tau^2 \mathbf{p} \cap \beta) \sum_{l=0}^\infty \frac{T^{\beta\cap \omega}}{l!}\mathfrak{q}_{l,k,\beta}(\tau_{tw}^l;h_1,\cdots,h_k)
\end{equation}
for $k>0$, and $\m_0^\tau = \tau_0 \cdot 1_{\BL} + \sum_\beta \exp( \tau^2 \mathbf{p} \cap \beta) \sum_{l=0}^\infty \frac{T^{\beta\cap \omega}}{l!}\mathfrak{q}_{l,0,\beta}(\tau_{tw}^l)$.

\begin{remark}
	Here we are slightly abusing notation. The expression $\mathfrak{q}_{l,k,\beta}(\tau_{tw}^l;h_1,\cdots,h_k)$ actually stands for
	$$\sum_{I=(i_1,\ldots,i_l)} \tau_{i_1}\cdots \tau_{i_l}\mathfrak{q}_{l,k,\beta}(\mathbf{T}_I;h_1,\cdots,h_k)$$
\end{remark}

The collection of operations $\{\m_k^\tau: k\in \Z_{\geq 0}\}$, for each $\tau$, an unital filtered $\AI$-algebra which we denote by $\CF(\BL, \tau)$. Like before, this $A_\infty$-algebra depends on the choices of CF-perturbations, but its homotopy type is independent of these choices \cite[Section 3.2]{FOOO_MS}.
Given an odd Maurer-Cartan element $b \in \CF^{\textrm{odd}}(\BL, \tau)$ we denote the deformed $\AI$ operations by $\m_k^{\tau,b}$.  

Recall that the Seidel Lagrangian together with the equator divides $X$ into eight regions with equal area (say $1$).  For our convenience, we may take $\mathbf{p}$ to be $\lambda$ times the sum of eight points, one in each region (for $\lambda \in \Q$).  We have $\tau^2\mathbf{p}\cap \beta = \lambda \cdot (\omega \cap \beta)$ since the area is given by the number of regions.   Then $ \exp( \tau^2\mathbf{p} \cap \beta) = t^{\omega \cap \beta}$ where $t := e^{\lambda}$ and $\omega \cap \beta \in \Z_{\geq 0}$. We will see in Proposition \ref{prop:tau2tau2'cc} that the choice of a representative of $\mathbf{p}$ does not affect our calculation significantly.




\section{Weakly unobstructedness for bulk-deformed Fukaya algebra} \label{sec:MC}
In this section, we show weakly unobstructedness of the Seidel Lagrangian in bulk deformed Floer theory.
The main geometric idea behind this result is the anti-symplectic involution of the orbifold sphere $\bP^1_{a,b,c}$. 
Let $\iota$ be the anti-symplectic involution on the orbifold sphere. The Seidel Lagrangian $i: \bS^1\mapsto \bP^1_{a,b,c}$ is chosen
so that the immersion $i$ is equivariant (with the involution on the domain $\bS^1$ by $\pi$-rotation). Note that $\iota$ preserves the orientation and spin structure of the Seidel Lagrangian.  The bulk inputs that we will consider are orbifold cohomology representatives of  $\bP^1_{a,b,c}$.  The twisted sectors and
the fundamental cycle are invariant under the involution.  A representative of the point class $[\pt] \in H^*_{orb}(X, \Lambda)$ will be chosen to be invariant under $\iota$.

Recall that the Seidel Lagrangian is shown to be weakly unobstructed in \cite{CHL17}. We extend it to the case of bulk deformations.
\begin{prop}\label{prop:weakunobsbulk}
Fix $\tau \in H^*_{orb}(X, \Lambda_+)$, let $\tilde{x},\tilde{y},\tilde{z} \in\Lambda_+$ and define $b=\tilde{x}X+\tilde{y}Y+\tilde{z}Z\in \CF^*(\BL,\tau)$. Any such $b$ is a weak Maurer-Cartan element (that is, a weak bounding cochain). In other words we have
$$\m^{\tau, b}_0 = \sum_{k\geq 0} \m^\tau_k(b,\ldots,b)=\cP(\tau,b) \one_\BL,$$
where $\one_\BL$ is the unit in $\CF^*(\BL,\tau)$ and $\cP(\tau,b)$ is some element in $\Lambda$.
\end{prop}
\begin{remark}
Here $\tilde{x},\tilde{y},\tilde{z}$ are
regarded as a scalar. In later sections they will be regarded as variables.  In particular, we will investigate convergence problems.
\end{remark}
\begin{proof}
From $\Z/2$-grading of Lagrangian Floer theory, $\m_0^{\tau,b}$ is given by a linear combination of even-degree immersed generators and a zero-form, i.e., a function on $\bS^1$ (the normalization of $\BL$).  In order to prove the proposition, that is, to verify that the weak Maurer-Cartan equation is satisfied, we need to: i) show that the immersed outputs are all zero;  ii) show the function on $\bS^1$ is constant.

The first step is similar to the argument in \cite{CHL17}, namely, any immersed output of $\m^\tau_k(b,\ldots,b)$ vanishes due to the cancellation from the involution $\iota_*$. An immersed output of $\m_k^{\tau} (b,\ldots,b)$, comes from considering a moduli space $\mathcal{M}^{main}_{k+1,l}(\beta, \mathbf{T}_I,\gamma)$, where the zero-th marked point is specified (by $\gamma$) to an immersed sector. In this case the automorphism
	\[ \iota_*: \mathcal{M}^{main}_{k+1,l}(\beta, \mathbf{T}_I,\gamma) \to \mathcal{M}^{main}_{k+1,l}(\beta, \mathbf{T}_I,\gamma)
\]
commutes with the evaluation map $ev_0^I$, as well as the other evaluation maps. Therefore, if we can show that $\iota_*$ reverses the orientation then the push-forward $(ev_0^I)_*$ (in the notation of \ref{eq:qmap}) vanishes. This is because the CF-perturbation is equivariant, which implies that $\iota_*$ restricts to an orientation-reversing diffeomorphism of the zero-set of the perturbed continuous-family of multi-sections (on each contributing Kuranishi chart) \cite[Section 7.3]{FOOO_new}. Then the vanishing follows from the same classical result on manifolds. Hence the immersed output of $\m^\tau_k(b,\ldots,b)$ will vanish.

Hence, for the first step in the proof, we are left with showing that $\iota_*$ reverses orientations. First recall that $\BL$ should is equipped with a non-trivial spin structure which brings the exact cancellation of signs (It is not weakly unobstructed with  the trivial spin structure). As explained in \cite[Section 3.3]{FOOOanti}, when studying the orientations it is enough to study the regular part of the moduli spaces, so we don't need to consider (orbi)-sphere bubbles. We will first consider the combinatorial sign rule of Seidel \cite{Se} and later argue that we may use them for our computation.

Let us consider a orbi-polygon $P$ that produces an immersed output of $\m_k^{\tau} (b,\cdots,b)$ (in particular, such a $P$ should have $k+1$ edges). By applying the reflection about the equator of $\mathbb{P}^1_{a,b,c}$ to $P$, we get another polygon $P^{op}$.
The $\AI$-operations for $P$ and $P^{op}$ give the same output (in $\Z/2$-graded theory) but  if $P$ contributes to $\m_k^{\tau} (X_1, \cdots, X_k)$ then $P^{op}$ contributes to $\m_k^{\tau} (X_k, \cdots, X_1)$. We claim that these two contributions have the opposite signs to each other.

Without loss of generality, let us assume that the boundary orientation of $P$ is coherent with that of $\mathbb{L}$. Then the boundary operation of $P^{op}$ is opposite (for each edges of $P^{op}$) to that of $\mathbb{L}$ since the reflection preserves the orientation of $\mathbb{L}$ whereas it reverses boundary orientations of holomorphic polygons.  From the sign rule of \cite{Se}, there is a sign difference of $(-1)^{k}$ between $P$ and $P^{op}$.  Another source of sign difference is how many times $P$ and $P^{op}$ pass through the point at which the spin structure is non-trivial. Let $s_1$ and $s_2$ denote these numbers, respectively.  We now show that $s_1 - s_2$, or equivalently, $s_1 + s_2$ has the same parity as $k+1$. We first claim that $\partial P \cup \partial P^{op}$ evenly covers $\mathbb{L}$. To see this, let us divide $\mathbb{L}$ into 6 minimal arcs, which are edges joining one corner with another without passing through other corners. We will denote these arcs by $\arc{XY}_{\pm},\arc{YZ}_{\pm},\arc{ZX}_{\pm}$ as in the left of Figure \ref{fig:evenlycover}. 

\begin{figure}[h]
\begin{center}
\includegraphics[height=1.5in]{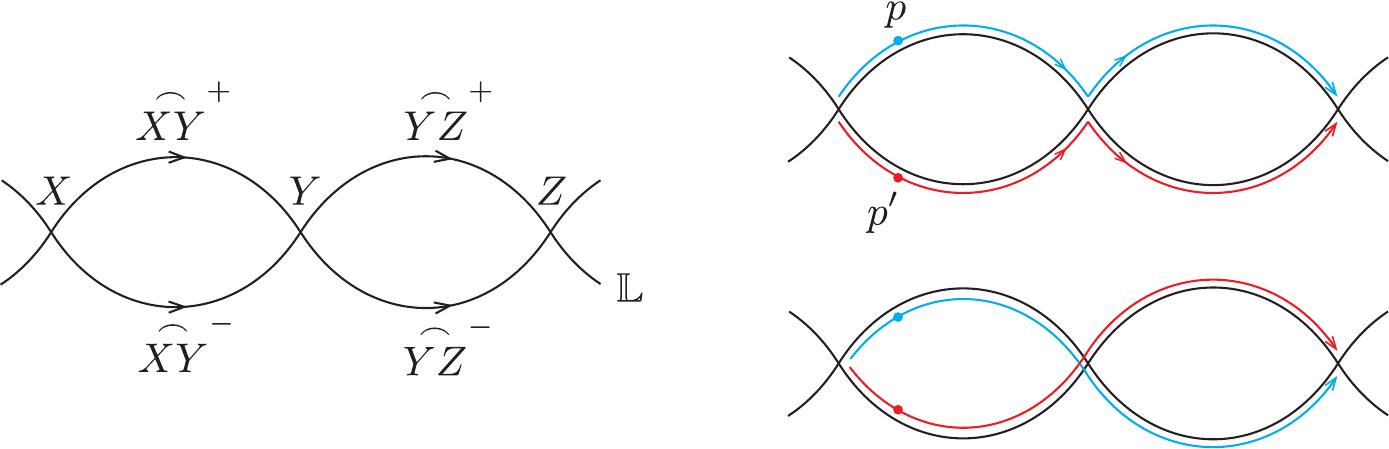}
\caption{}\label{fig:evenlycover}
\end{center}
\end{figure}

Suppose $p$ is a point on the boundary of $P$ that lies in $\arc{XY}_{\pm}$, then its reflection image, say $p'$, is located on $\arc{XY}_{\mp}$. When $p$ travels along $\partial P$, the pair $(p,p')$ first  covers both of $\arc{XY}_{\pm}$.  And then the pair starts covering both of $\arc{YZ}_{\pm}$ afterward, regardless of having corners at $Y$ or not (see the right of Figure \ref{fig:evenlycover}). Since $p$ starts at and comes back to the same point when we go along $\partial P$ once, we see that $(p,p')$ covers $\mathbb{L}$ evenly. 

Having this, let $\partial P \cup \partial P^{op} = s [\mathbb{L}]$, which implies $s_1 + s_2 = s$. If $(k+1)$ edges of $\partial P$ consists of $a_1, a_2, \cdots, a_{k+1}$ minimal arcs, we have
$$6s = 2 (a_1 + a_2 + \cdots + a_{k+1}) \Rightarrow 3s = a_1 + a_2 + \cdots + a_{k+1}$$
since $\partial P \cup \partial P^{op} $ has $6s$ minimal arcs. On the other hand, it is easy to see that each edge of $\partial P$ (and $\partial P^{op}$) consists of an odd number of minimal arcs (as otherwise, the polygon would have a non-convex corner) i.e., $a_i$ are all odd, and hence we conclude that the parity of $s$ is the same as $k+1$, which completes the proof of the claim.

Now, let us argue that the combinatorial sign of Seidel is compatible with the de Rham model we use in this paper.
In \cite[Part II section 13]{S08}, it is shown that the sign of an $\AI$-operation defined using Floer theory (orientation operators) and that defined by combinatorial convention can be identified. Seidel showed that in this surface case, the sign in Floer theory is local, and hence depends only on absolute indices of intersection points. On the other hand, there is a combinatorial way of assigning signs in this case. Seidel constructs a linear isomorphism $\gamma(k)$ for each Floer group $CF^k$ which makes these two signs compatible. We use the existence of this isomorphism to show cancellations.
Note that the combinatorial sign only depends on the parity of the intersection points. 
 
On the other hand, one can also show that the Floer sign also depends only on the parity of the absolute indices of corners.  Note that a choice of path of Lagrangian subspaces from $T_pL_1$ to $T_pL_2$ for $p \in L_1\cap L_2$ defines an orientation operator, which can be used to define its absolute index as well as associated orientation space (determinant of the orientation operator). Absolute indices from different choice of paths may differ by even integer, and one can fix a canonical isomorphism between two different choices using gluing of discs with Lagrangian loop of the difference of paths. It can be shown that
this gluing provides a canonical way to relate orientation spaces corresponding to different paths, which gives rise to the same
sign for associated polygons. In this way, one can observe that the Floer theoretic sign only depends on the parity of the absolute indices at the intersection points based on the above isomorphism of orientation spaces.

Thus we have completed the first step and shown that $\m_k^\tau(b,\cdots,b)$ does not involve even-degree immersed generators and hence is an element in $\Omega^0(\bS^1)$.  It remains to show that it is a constant function in $\bS^1$ and therefore a multiple of the unit. We will prove this in the next lemma. Note that the argument below 
is simply a consequence of the $\AI$-equations. Therefore it could be of interest in other situations.

\begin{lemma}\label{lem:multunit}
The expression $\m_k^\tau(b,\cdots,b)$  is a constant function in $\bS^1$, for $k \geq 0$.
\end{lemma}
\begin{proof}
 Observe that the energy zero component $\m^\tau_{1,\beta=0}$ of $\m^\tau_1$ on a function $f$ is given by $df$ by definition. Hence, in order to show that the $\m_k^\tau(b,\cdots,b)$ is constant, it is enough to prove that 
$\m^\tau_{1,\beta=0}(\m^\tau_{k}(b,\ldots,b))=0$. 

We start with the case $k=0$. The first $\AI$-equation gives 
\begin{equation}\label{eq:ainf}
	\sum_{\beta=\beta_1+\beta_2}\m^{\tau}_{1, \beta_1}(\m^\tau_{0,\beta_2})=0,
\end{equation}
for each energy level $\beta$. For the minimal (non-zero) energy, the equation immediately gives the desired equation $\m^\tau_{1,\beta=0}(\m^\tau_{0, \beta})=0$. Then we proceed by induction on the energy level - recall that these levels form a countable set by Gromov-compactness. Fix $\bar{\beta}$ and assume we have proven the result for any $\beta<\bar{\beta}$. Consider Equation (\ref{eq:ainf}) for $\bar{\beta}$. For each $\beta_2<\bar{\beta}$, by induction hypothesis $\m^\tau_{0,\beta_2}$ is a multiple of the unit and by unitality $\m^{\tau}_{1, \beta_1}(\m^\tau_{0,\beta_2})=0$. Therefore Equation (\ref{eq:ainf}) simplifies to $\m^\tau_{1,\beta=0}(\m^\tau_{0, \bar{\beta}})=0$. This proves the $k=0$ case.

We then proceed by induction on $k$.
Let us assume the statement for $k\leq i$ and prove the case of $i+1$. Using the $\AI$-equation, the induction hypothesis and unitality gives $\m^\tau_{1}(\m^\tau_{i+1}(b,\ldots,b))=0$. As in the $k=0$ case using an induction on energy we see that this last equality is equivalent to $\m^\tau_{1,0}(\m^\tau_{i+1, \beta}(b,\ldots,b))=0$ for each $\beta$. This proves the result. 

\end{proof} 

This proves the weakly unobstructedness of the Seidel Lagrangian $(\BL, b)$. 
\end{proof}

It follows from Lemma \ref{lem:multunit} that each $b$ determines a deformation of $\CF(\BL)$ with central curvature $\m_0^{\tau,b}$. This means that $\m_1^{\tau,b}$ is a differential and its cohomology $HF^*(\BL,\tau, b)$ is an algebra with product $\m_2^{\tau,b}$.
We will describe this algebra in Section \ref{sec:FA}.

\section{Bulk-deformed potential function and change of variables}\label{sec:bdpotchangevar}
The previous section asserts that the Lagrangian Floer potential function $\cP(b)$ is a formal power series in $\tilde{x},\tilde{y}$ and $\tilde{z}$ with coefficients in the Novikov ring $\Lambda_0$, where $b=\tilde{x} X + \tilde{y} Y + \tilde{z} Z$. 
As explained in the introduction, it is essential for the purpose of studying the Kodaira-Spencer map that we work with the following change of variables.
\begin{align}\label{eqn:chvar}
\left\{\begin{array}{lcl}
x&=&T^3 \tilde{x},\\
y&=&T^3 \tilde{y},\\
z&=&T^3 \tilde{z}.
\end{array}\right.
\end{align}
with $x,y,z \in \Lambda_0$.
From now on we denote 
$$W_\tau(x,y,z)=\cP(\tau,b)$$ and call this the \emph{potential function}. Notice that $b$ on the right hand side is now given by $b= T^{-3} x X +T^{-3} y Y + T^{-3} z Z$.

This coordinate change will be essential in our study of Kodaira-Spencer map, and at the same time it is the main source of complication.

After the coordinate change the term of minimal valuation in the potential $W_\tau$ is $T\tilde{x}\tilde{y}\tilde{z} = T^{-8} xyz$.  This negative energy term should be handled in a delicate way as we will see in our proof of the Kodaira-Spencer map being an isomorphism. For this reason, we will need a better control on the energy of the terms appearing in the related Floer operations
and algebraic manipulations.

We first examine the potential function and its convergence in new variables.
When there is no bulk deformation (i.e. $\tau=0$), \cite{CHKL14} gives closed formulas for $W$ in the spherical and elliptic cases - in these cases, $W$ is simply a polynomial on $x,y,z$. In the hyperbolic case (again when $\tau=0$), an algorithm that computes $W$ is given in \cite{CHKL14}. 

\subsection{Gauss--Bonnet theorem and convergence}
Recall that $b=\tilde{x}X+\tilde{y}Y+\tilde{z}Z=T^{-3}(xX+yY+zZ)$ where $\tilde{x},\tilde{y},\tilde{z}$ are the dual variables to the immersed generators $X,Y,Z$.  Gromov compactness ensures the boundary deformed $A_\infty$ algebra is convergent when $\mathrm{val}\,{\tilde{x}}$, $\mathrm{val}\,{\tilde{y}}$, $\mathrm{val}\,{\tilde{z}} >0$.  We will show that it is still convergent for $\mathrm{val}\,{\tilde{x}},\mathrm{val}\,{\tilde{y}},\mathrm{val}\,{\tilde{z}} \geq -3$ (that is $\mathrm{val}\,{x},\mathrm{val}\,{y},\mathrm{val}\,{z} \geq 0$).

\begin{defn}\label{def:cpr}
	A \emph{convergent power series} in $x,y,z$ is a series of the form
	$$\sum_{i,j,k\in\Z_{\geq 0}}c_{i,j,k}x^iy^jz^k,$$
	with $c_{i,j,k}\in \Lambda$ and $\lim_{i+j+k\to\infty}val(c_{i,j,k})=+\infty$, where $val$ is the usual valuation in $\Lambda$. We denote by $\Lambda \langle\langle x,y,z\rangle\rangle$ the ring of convergent power series.
\end{defn}

Recall that the valuation $val$ in $\Lambda$ determines a non-archimedean metric by the formula $|\xi|=e^{-val(\xi)}$. The condition above then states that the coefficients of the series converge to zero in this norm. Therefore the ring just defined is a special case of the \emph{Tate algebra}, see \cite{BGR}.

Note that any element $P\in\Lambda \langle\langle x,y,z\rangle\rangle$ is indeed convergent (in the unit disc), in the sense that it determines a map $P:\Lambda_0^3\longrightarrow\Lambda$. 

In the remainder of this section we will show that $W_\tau(x,y,z)$ is a convergent power series for each $\tau$ with $\mathrm{val}\,(\tau)>0$, and hence, is an element in $\Lambda \langle\langle x,y,z\rangle\rangle$. 
We begin by establishing a complete classification of the non-positive energy terms of $W_\tau(x,y,z)$. First of all, we have $T^{-8} xyz$ from  the minimal triangle. In addition, we have orbi-discs which give $x^i, y^j, z^l$ with $1 \leq i \leq a-1$, $1 \leq j \leq b-1$, $1 \leq l \leq c-1$  with energy zero coefficients (which are all $1$), which are obtained as fractions of the minimal smooth discs corresponding to $x^a, y^b, z^c$. These discs have exactly one interior orbi-insertion, and can be viewed as the first-order contribution of orbi-sectors in $\QH^*_{orb}(X,\tau)$ to the potential. 
We give a precise description on such discs by the lifting argument below, which is valid for general orbi-discs although their liftings are the maps defined on higher genus (bordered) Riemann surfaces in most of cases.

\begin{lemma}
Suppose $\mathcal{D}$ is an orbifold disc with interior orbifold marked points $p_1,\cdots,p_k$ where 
$p_i$ is $\Z/k_i$ cone point.
Consider a map $\pi: U \to \mathcal{D}$ between Riemann surfaces with boundary (mapping boundaries to boundaries)
and suppose that $p_i$ is a branch point of $\pi$ of multiplicity $k_i$ for each $i$.
For any orbifold holomorphic disc $u: (\mathcal{D}, \partial \mathcal{D}) \to (X,L)$,
the composition $u \circ \pi : (U,\partial U) \to   (X,L)$ is a holomorphic disc.
\end{lemma}
\begin{proof}
Note that $u \circ \pi$ is a holomorphic disc away from $\pi^{-1}(p_i)$ by definition.
Near each $\pi^{-1}(p_i)$,  $u \circ \pi$ is nothing but the lift to a uniformizing cover, hence it is holomorphic.
\end{proof}

Applying the lemma to orbi-discs with a single orbi-insertion, we obtain the following.  

\begin{corollary}\label{coro:conelifts}
A holomorphic orbi-disc $u$ with one orbifold marked point has a holomorphic lift $\widetilde{u}: U \to X$, where $U$ is a disc.
\end{corollary}

Therefore we see that all such orbi-discs are slices of the discs that contribute to non-bulk-deformed potential. In particular, the slice of the discs for $x^a,y^b,z^c$ will be called the \emph{basic orbi-discs} from now on.

The following lemma gives a complete description of the low energy orbi-discs contributing to the potential. It is a kind of energy quantization at the corners of the discs. A similar result in dimension greater or equal than two appears in \cite[Lemma 4.2]{woodward}.

\begin{lemma} \label{lem:geq0}
	Except for the single term corresponding to the minimal triangle, every term of the bulk-deformed orbifold potential in $x,y,z$ variable 
	\eqref{eqn:chvar} 
	has non-negative $T$-exponent.
	Moreover, it has $T$-exponent being $0$ exactly for basic orbi-discs,  and the $T$-exponents are positive for the rest (except the minimal triangle).
\end{lemma}

\begin{proof}
Let $u:(\mathcal{D},\partial \mathcal{D}) \to (\Pabc, \BL)$ be a non-constant holomorphic orbi-disc which is not the minimal triangle. We may assume that the counter-clockwise orientation of $\partial D^2$ agrees with the orientation of $\BL$ under $u$. (The other case can be handled similarly).  $\Pabc$ is decomposed into 8 pieces by $\BL$ and the equator.  We denote by $M_u, M_l$  the upper and lower minimal triangles and denote by $A_u, B_u, C_u$ (resp. $A_l,B_l,C_l$) the triangles with one of their corners at $a,b,c$ orbifold points respectively and lies in the upper (resp. lower) hemisphere.  We may decompose the domain of the orbi-disc $\mathcal{D}$ according to the above decomposition under the map $u$.
	Suppose $u$ has an immersed corner mapping to $X,Y$ or $Z$ contributing to the monomial of
	the potential. By our choice of orientation, the map $u$ covers the piece $M_u$.
	(One can check that we cannot turn corners at $M_l$ in this case). 
	We consider the part of $u$ which maps to $M_u$ as in Figure \ref{fig:posorbpot}. 
	We argue that in the neighborhood of each such corner, we have additional regions (in $\mathcal{D}$) of area 2 which are distinct for
	each corner.
	Note that the piece $M_u$ should be attached to exactly one of  $A_u, B_u, C_u$, say the piece $A_u$, since it involves a corner of the disc.
	In this case, this $A_u$ cannot be attached to any other preimages of $M_u$ or $M_l$ in $\mathcal{D}$. Now, $A_u$ is attached to
	two of $A_l$ pieces $A_l^1, A_l^2$. Let us further cut these $A_l$ pieces into halves.
	
	\begin{figure}[h]
		\begin{center}
			\includegraphics[height=1.2in]{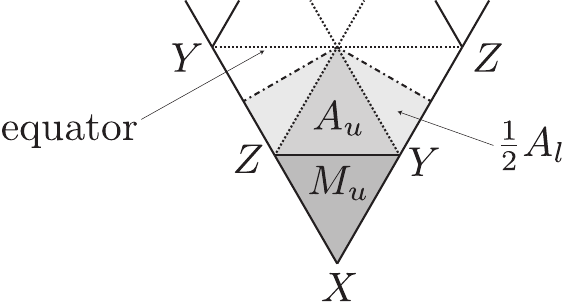}
			\caption{}\label{fig:posorbpot}
		\end{center}
	\end{figure}
	
	Hence each corner of $\mathcal{D}$ at least covers $M_u, A_u, \frac{1}{2} A_l^1, \frac{1}{2} A_l^2$, which gives the area $3$ (or $T^3$).
	Given two different corners, the above local pieces do not overlap since $A_u$ can be attached to the only one corner piece $M_u$.
	This proves the first part of the proposition. Suppose that after the coordinate change, it has no $T$-component.
	This means that the holomorphic orbi-disc consists of these $T^3$ pieces only.  It is elementary to see that the only orbifold holomorphic discs that  we can make in this way are the basic ones. This proves the lemma.
\end{proof}

Next, we will use a version of the Gauss--Bonnet theorem to compute the valuations of the monomials appearing in $W_\tau$.

Recall that the Seidel Lagrangian $\BL$  is taken to be symmetric about the equator, and it subdivides each of the upper and lower hemispheres into four triangles with equal area $A$ (which is set to be $1$), which are $M_u,A_u,B_u,C_u$ and $M_l,A_l,B_l,C_l$ respectively in the proof of Lemma \ref{lem:geq0}.  Let $K$ be the constant curvature of the orbi-sphere.  The equator is taken to be a union of three geodesics connecting the three orbi-points, and the reflection about the equator is an isometry.  Denote by $k$ the geodesic curvature of $\BL$.

We arrange $\BL$ in such a way that the exterior angles of the minimal triangles at $X,Y,Z$ are $2\pi(1/a - \epsilon), 2\pi(1/b - \epsilon), 2\pi(1/c - \epsilon)$ respectively, where $\epsilon$ is taken such that the angles are in $(0,\pi)$.  $\epsilon$ can be set to zero in case $a,b,c\geq 3$.

The Gauss-Bonnet formula for an orbi-polygon is given as follows. 

\begin{theorem}[Gauss-Bonnet formula for an orbi-polygon] \label{thm:GB}
	For an (embedded) orbi-polygon $P\subset X$ with exterior angles $\angle_i$, boundary edges $\gamma_j$, and ages of interior orbi-points being $\iota_k$, 
	$$ \int_P K dA + \sum_i \angle_i + \sum_j \int_{\gamma_j}kds +  2\pi \sum_k (1-\iota_k) = 2\pi. $$
\end{theorem}

\begin{proof}
	Given an orbi-polygon, we make a cut for each orbi-point (along a simple path connecting the orbi-point to a boundary point).  This gives a new polygon where the cutting path appears as part of the boundary.  Then we apply the ordinary Gauss-Bonnet formula to this usual polygon (with the restricted metric).  The exterior angle at each orbi-point equals to $2\pi (1-\iota_k)$.  Moreover, the integration of geodesic curvature along the pair of paths going back and forth on each cut cancel with each other.  This gives the formula as stated.
\end{proof}

Closely related to this, the Maslov-index formula for an orbi-polygon class $(\beta,\alpha)$ in terms of curvature (where $\alpha$ is the collection of immersed generators that the corners hit) is given as follows (see Remark \ref{rem:cw}, \cite{CS} or \cite{Pacini}). (Here $\angle X$ denotes the exterior angles).
\begin{equation} \label{eq:mu}
\mu_{\mathrm{CW}}(\beta,\alpha)=\frac{1}{\pi}\left(\int_\beta K dA + \int_{\partial \beta} k ds + \sum_{X\in \alpha} \angle X\right).
\end{equation}

In our case the Maslov index is given as follows.
\begin{prop} \label{prop:GB}
	For an orbi-disc bounded by $\BL$ with corners being only $X,Y$ or $Z$ (thus contributing to the potential), denote its area by $mA$
	the numbers of $X,Y,Z$ corners by  $n_1,n_2,n_3$ respectively.  Its Maslov index $\mu_{\mathrm{CW}}$ equals to
	$$ 
	2\left(\frac{n_1}{a}+\frac{n_2}{b}+\frac{n_3}{c} +  (m-3(n_1+n_2+n_3)) \cdot \frac{\chi}{8}\right).
	$$
	In particular the Chern number of an orbi-sphere equals to $ m\chi/8.$
\end{prop}

Recall that $\chi = -1 + \frac{1}{a} + \frac{1}{b} + \frac{1}{c}$ is the (orbifold) Euler characteristic of $\mathbb{P}^1_{a,b,c}$.

\begin{proof}
	First of all we find the geodesic curvatures of the edges of the minimal triangles.
	By Gauss-Bonnet formula applied to the upper hemisphere (which is a triangle bounded by three geodesics segments forming the equator), we have
	$$ 4AK +  \left(\pi-\frac{\pi}{a}\right) + \left(\pi-\frac{\pi}{b}\right)+\left(\pi-\frac{\pi}{c}\right)= 2\pi  $$
	(where $A=1$ is the area of the minimal triangle.)
	Thus
	\begin{equation} \label{eq:chi}
	\chi = \left(\frac{1}{a}+\frac{1}{b}+\frac{1}{c}\right)-1 = \frac{8AK}{2\pi}.
	\end{equation}
	
	Let $2\pi\,k_{12}$ be the total geodesic curvature along the edge connecting the $X$ and $Y$ corners of the minimal triangle contained in the upper hemisphere. Here, the orientation of the Seidel Lagrangian is fixed such that the orientations of the edges agree with the induced ones from the minimal triangle contained in the upper hemisphere.  $k_{23}$ and $k_{31}$ are similarly defined.
	We apply the Gauss-Bonnet formula to the triangle $C_u$ which is the triangle in the upper hemisphere having corners at the point $X,Y$ and the orbi-point $[1/c]$. This gives
	$$ \left(\pi-\frac{\pi}{a}+\pi\epsilon \right)+\left(\pi-\frac{\pi}{b}+\pi\epsilon\right)+\left(\pi-\frac{\pi}{c}\right) - 2\pi k_{12} + KA = 2\pi.$$
	Combining with \eqref{eq:chi}, we have $k_{12}=\kappa+\epsilon$ where $\kappa=-3\chi/8$.  Applying the same argument for the other two triangles contained in the upper hemisphere, we obtain
	$$ k_{12}=k_{23}=k_{31}=-\frac{3\chi}{8} + \epsilon=-\frac{3KA}{2\pi}+ \epsilon=\kappa+\epsilon.$$
	Then the total geodesic curvatures of the edges of the minimal triangle in the lower hemisphere (in the fixed orientation of the Seidel Lagrangian) are ($2\pi$ times)
	$$ k_{12}'=k_{23}'=k_{31}'=-k_{31}=-\kappa-\epsilon.$$
	
	Consider a holomorphic orbi-disc bounded by $\BL$ with the numbers of $X,Y,Z$ corners being $n_1,n_2,n_3$ respectively.  The area is a multiple $mA$ of the area of the minimal triangle for $m\in\Z_{>0}$.  Thus
	$$ \frac{1}{\pi} \int_\beta K dA = \frac{mKA}{\pi}=\frac{m\chi}{4}.$$
	
	The edges of the orbi-disc are unions of the edge segments of the two minimal triangles in the upper and lower hemispheres.
	By the property of holomorphic orbi-disc, each side between two corners must consist of an odd number of edge segments.
	The total geodesic curvatures of the edge segments cancel with each other, except for one edge segment for each side.  Such a segment in each side lie in the same hemisphere, and in the boundary orientation of the holomorphic disc its geodesic curvature is $2\pi k_{12}=2\pi (\kappa+\epsilon)$.  Then 
	$$ \frac{1}{\pi} \int_{\partial \beta} k ds = 2 (n_1+n_2+n_3) k_{12} = \frac{-3\chi \cdot (n_1+n_2+n_3)}{4} + 2 (n_1+n_2+n_3) \epsilon.$$
	Since the exterior angles of $X,Y,Z$ are $2\pi(1/a-\epsilon), 2\pi(1/b-\epsilon),2\pi(1/c-\epsilon)$ respectively, we have
	$$\frac{1}{\pi}\sum_{X\in \alpha} \angle X = 2\left(\frac{n_1}{a}+\frac{n_2}{b}+\frac{n_3}{c}\right) - 2 (n_1+n_2+n_3) \epsilon.$$
	Hence the error term $2 (n_1+n_2+n_3) \epsilon$ cancels in the sum.  Combining the above equations, result follows.
\end{proof}

\begin{corollary}
	Consider a stable orbi-disc bounded by $\BL$
	which contributes to the disc potential.  Suppose it has area $mA$, the interior orbi-insertions have ages $\iota_j$, and the numbers of $X,Y,Z$ corners are $n_1,n_2,n_3$ respectively.  Then it satisfies
	\begin{equation}\label{eqn:chiabc}
	(m-3(n_1+n_2+n_3)) \cdot \frac{-\chi}{8}=  \frac{n_1}{a}+\frac{n_2}{b}+\frac{n_3}{c}+ \sum_j (1-\iota_j)-1.
	\end{equation}
\end{corollary}

\begin{proof}
Recall that an orbi-disc contributes to the potential if it is rigid, or equivalently its Maslov index satisfies 
	$$\mu_{\mathrm{CW}} + 2 \cdot \sum_{j} (1 - \iota_j) -2 =0.$$
	A stable orbi-disc consist of disc and sphere components.
	Since Maslov index is additive, Proposition \ref{prop:GB} applied to each component gives the result.  (It can also be seen by taking an orbi-smooth disc representative of the class.)
\end{proof}

From the corollary, the area of the orbi-disc is given in terms of the numbers of corners and ages by
\begin{equation}\label{eqn:orbareaformula}
mA = \left(3(n_1 + n_2 + n_3) + 8 \dfrac{\frac{n_1}{a} + \frac{n_2}{b} + \frac{n_3}{c} + \sum_{j} (1-\iota_j) -1}{1- \left(\frac{1}{a} + \frac{1}{b} + \frac{1}{c} \right)} \right)A
\end{equation}
when $\chi \neq 0$, which matches the one given in \cite{CHKL14} when there is no orbi-insertions. If $\chi = 0$ (i.e. elliptic case), we have $\frac{n_1}{a}+\frac{n_2}{b}+\frac{n_3}{c}+ \sum_j (1-\iota_j)=1$.

Suppose that $T^{a} \tau^{\vec{k}} x^{n_1}y^{n_2}z^{n_3} \in \Lambda[[\tau,x,y,z]]$ is one of monomials contained in $W_\tau$. Our discussion so far tells us that the exponent $a$ satisfies $a=m-3(n_1+n_2+n_3) \geq 0$, where $m$ is given by \eqref{eqn:orbareaformula} and $\vec{k}$ records the ages $\iota_j$ in the formula. Notice that the coordinate change \ref{eqn:chvar} is responsible for the term $-3(n_1+n_2+n_3)$ in $a$. 
By Lemma \ref{lem:geq0}, $m-3(n_1+n_2+n_3) \geq 0$ for all holomorphic orbi-discs other than the minimal triangle.

\begin{theorem}\label{thm:bdconv}
	The bulk-deformed potential $W_\tau$ is a convergent series, that is, it is an element of $\Lambda\langle\langle x,y,z\rangle\rangle$.
\end{theorem}
\begin{proof}
	By Gromov compactness, it suffices to show that the area $mA$ is bounded above once the exponent $m-3(n_1+n_2+n_3)$ of $T$ is bounded.
	In the elliptic case, \eqref{eqn:chiabc} gives 
	$$\frac{n_1}{a}+\frac{n_2}{b}+\frac{n_3}{c}+ \sum_j (1-\iota_j)=1.$$
	There are only finitely many possibilities of $n_i$ and orbi-insertions satisfying this.  (And hence $W_\tau$ is a polynomial in $x,y,z,\tau$.)
	In particular $n_i$ are bounded.  Once $m-3(n_1+n_2+n_3)$ is bounded, $m$ is bounded.
	
	For the spherical case, the $T$-valuation $m- 3(n_1+n_2+n_3) \geq 0$ (other than the minimal triangle) implies that 
	$$ \frac{n_1}{a}+\frac{n_2}{b}+\frac{n_3}{c}+ \sum_j (1-\iota_j) \leq 1.$$
	There are only finitely many possibilities of $n_i$ and orbi-insertions satisfying this.  Thus there are only finitely many possibilities of $m- 3(n_1+n_2+n_3)$, and hence $m$.  $W_\tau$ just consists of finitely many terms.
	
	In the hyperbolic case, if $m- 3(n_1+n_2+n_3)$ is bounded above, then $\frac{n_1}{a}+\frac{n_2}{b}+\frac{n_3}{c}+ \sum_j (1-\iota_j)$ is also bounded above.  This gives finitely many possibilities of $n_i$ and $\iota_j$.  Hence $n_1+n_2+n_3$ is bounded above, and so is $m$.
\end{proof}

\begin{remark}
	Recall that $\tau=\tau^0 \one_{X}+ \tau^2 \mathbf{p}+ \tau_{tw}$, and $\exp (\tau^2\mathbf{p}\cap \beta)=t^{\omega \cap \beta}$ (see the paragraph before Section \ref{sec:MC}).  The above shows that $W(\tau_0,t, \tau_1,\ldots, \tau_m,x,y,z)$ is an element of $\Lambda\langle\langle \tau_0,t, \tau_1,\ldots, \tau_m,x,y,z\rangle\rangle$.
\end{remark}

Similarly, we can show that every $\m_k^{\tau,b}(\alpha_1,\ldots,\alpha_k)$ is a convergent power series. This will allow us to, instead of evaluating the $\AI$ operations at a specific value of $b$, to define $\CF(\BL, \tau, \Lambda\langle\langle x,y,z\rangle\rangle)$ a Fukaya algebra with coefficients on  the ring $\Lambda\langle\langle x,y,z\rangle\rangle$. More concretely, elements in $\CF(\BL, \tau, \Lambda\langle\langle x,y,z\rangle\rangle)$ are power series of the form
$\sum_{i,j,k\in\Z_{\geq 0}}c_{i,j,k}x^iy^jz^k,$ with $c_{i,j,k}\in \CF^*(\BL)\otimes_{\Lambda_0} \Lambda$ and $\lim_{i+j+k\to\infty}val(c_{i,j,k})=+\infty$, where $val$ is the valuation in $\CF^*(\BL)\otimes_{\Lambda_0} \Lambda$ inherited from $\Lambda$. Then we can extend the operations $\m_k^{\tau,b}$ to $\CF(\BL, \tau, \Lambda\langle\langle x,y,z\rangle\rangle)$ - this is the content of the next proposition. 

\begin{prop}\label{prop:fukalgconv}
$\CF(\BL, \tau, \Lambda\langle\langle x,y,z\rangle\rangle)$ is an $A_\infty$-algebra linear over $\Lambda\langle\langle x,y,z\rangle\rangle$.
\end{prop}
\begin{proof}
	First observe that, once we show that $\m_k^{\tau,b}(\alpha_1,\ldots,\alpha_k)$ is convergent, that is it belongs to $\CF(\BL, \tau, \Lambda\langle\langle x,y,z\rangle\rangle)$ for $\alpha_1,\ldots,\alpha_k \in \CF^*(\BL)$, we can simply extend $\m_k^{\tau,b}$ linearly (and continuously) over $\Lambda\langle\langle x,y,z\rangle\rangle$. This will give the proposition. In the remainder of the proof we show this convergence.
	
Consider an orbi-polygon $P$ with $k_1,k_2,k_3$ numbers of $X,Y,Z$ corners (which have odd degree), and $k_1^-,k_2^-,k_3^-$ numbers of $\bar{X},\bar{Y},\bar{Z}$ corners (which have even degree) respectively.  Without loss of generality, we assume that in a neighborhood of one of the odd corners, $P$ is contained in the upper hemisphere.  (If there is no odd corner, then we assume that in a neighborhood of one of the even corners, $P$ is contained in the lower hemisphere.)

Note that for a side of $P$ between adjacent odd and even corners, the number of minimal edge segments is even.  Similarly for an odd-odd side or even-even side, the number of minimal edge segments is odd.  Thus two corners adjacent to an odd-odd or even-even side remain in the same hemisphere; on the other hand, an odd-even side connects a corner in the upper hemisphere to a corner in the lower hemisphere.  It implies all odd corners of $P$ are contained in the upper hemisphere, and all even corners are contained in the lower hemisphere. 

Following the notations in Proposition \ref{prop:GB}, each minimal edge segment lying in the upper (or lower) hemisphere has total geodesic curvature $k_{12}=\kappa + \epsilon$ (or $-k_{12}$ resp.) where $\kappa = -3\chi / 8$.  
Thus, for an odd-even side, the geodesic curvature of the edge segments cancel among each other. For an odd-odd edge (resp. even-even edge), the geodesic curvature of all but one edge segment cancel, and that edge segment lies in the upper (resp. lower) hemisphere.  It follows that the total geodesic curvature of an odd-odd edge, odd-even edge, and even-even edge is $2\pi$ times $k_{12}=\kappa + \epsilon, 0, -k_{12}=-\kappa - \epsilon$ respectively.

Consider Formula \eqref{eq:mu}.
We claim that the error term, namely the term which is a multiple of $\epsilon$, is zero in $\mu_{\mathrm{CW}}(\beta,\alpha)$.  Recall that the exterior angles of the odd vertices $X,Y,Z$ are $2\pi(1/a-\epsilon), 2\pi(1/b-\epsilon),2\pi(1/c-\epsilon)$ respectively.  As in the proof of Proposition \ref{prop:GB}, each odd vertex contributes $-2\epsilon$ to the error term in the Maslov index.  For the even vertices $\bar{X},\bar{Y},\bar{Z}$, the exterior angles are $\pi - 2\pi(1/a-\epsilon), \pi - 2\pi(1/b-\epsilon),\pi - 2\pi(1/c-\epsilon)$ respectively.  Each even vertex contributes $2\epsilon$ to the error term in the Maslov index.  

Let $l_{oo}, l_{oe}, l_{ee}$ be the numbers of odd-odd, odd-even, even-even sides respectively.  Then the numbers of odd and even vertices equal $(l_{oe}+2l_{oo})/2$ and $(l_{oe}+2l_{ee})/2$ respectively.  The total error contribution from the angles of the vertices is
$$ - 2\pi\epsilon (l_{oe}+2l_{oo})/2 + 2\pi\epsilon (l_{oe}+2l_{ee})/2 = -2\pi\epsilon(l_{oo} - l_{ee}).$$

The geodesic curvature of an odd-odd (even-even resp.) edge contributes $2\pi\epsilon$ ($-2\pi\epsilon$ resp.) to the error; the geodesic curvature of an odd-even edge has no error term.  Thus the total error contribution from the geodesic curvatures of the sides is
$$ 2\pi\epsilon (l_{oo} - l_{ee}). $$
We see that the above two error contributions cancel among each other and hence there is no $\epsilon$-term in the Maslov index.

Thus we can throw away the $\epsilon$ terms.  The total geodesic curvature (mod $\epsilon$) of the sides equal to $2\pi \kappa (l_{oo} - l_{ee})$, and $l_{oo} - l_{ee}$ equals to the number of odd vertices minus the number of even vertices, that is $k_1+k_2+k_3 - k_1^- - k_2^- - k_3^-$.  Thus we obtain the following which is a generalization of the Maslov index formula in Proposition \ref{prop:GB}:
\begin{align*}
\mu_{\mathrm{CW}}(P)  = 
2\bigg(\frac{k_1-k_1^-}{a}+\frac{k_2-k_2^-}{b}+\frac{k_3-k_3^-}{c}& +  \frac{k_1^-+k_2^-+k_3^-}{2}  \\
& + (m-3(k_1+k_2+k_3-k_1^- - k_2^- - k_3^-)) \cdot \frac{\chi}{8}\bigg).
\end{align*}

Now consider a stable orbi-disc contributing to a term of $\m_k^{\tau,b}(\alpha_1,\ldots,\alpha_k)$ with the monomial $T^{m-3(n_1+n_2+n_3)}x^{n_1}y^{n_2}z^{n_3}$.  The corresponding dimension formula is
$$\mu_{\mathrm{CW}} + \delta + 2 \cdot \sum_{j} (1 - \iota_j) =2$$
where $\delta=0$ if the output is a zero-form or an immersed generator, $\delta=1$ if the output is a one-form.  
Combining, we have
$$ \left(m-3\left(n_1+n_2+n_3 + \sum_{i=0}^k s_i\right)\right) \cdot \frac{-\chi}{8} = \frac{n_1}{a}+\frac{n_2}{b}+\frac{n_3}{c}+ \sum_{i=0}^k \theta_i + \frac{\delta}{2} -1 + \sum_j (1-\iota_j). $$
Here, for $i=1,\ldots,k$, $s_i=1,-1,0$ depending on the input $\alpha_i$ being odd generators, even generators, or point classes respectively.  For $i=0$, $s_0=1,-1,0$ depending on the output being even generators, odd generators, or $\one_\BL,\pt_\BL$ respectively.   $\theta_i = 1/a,1/b,1/c$ if ($\alpha_i = X,Y,Z, i>0$) or ($\alpha_i=\bar{X},\bar{Y},\bar{Z},i=0$); $\theta_i = 1/2-1/a,1/2-1/b,1/2-1/c$ if ($\alpha_i = X,Y,Z, i=0$) or ($\alpha_i=\bar{X},\bar{Y},\bar{Z},i>0$).  $\theta_i=0$ in all other cases (that is, $\alpha_i$ is either $\one_\BL$ or $\pt_\bL$).

Then the argument goes in a similar way as the proof of Theorem \ref{thm:bdconv}.  We shall show that there are only finitely many possibilities of $n_i$ satisfying the above equality.  Then the boundedness of the valuation $m-3(n_1+n_2+n_3)$ will imply that of $m$.  This means that the area is bounded above.  It follows from Gromov compactness that there are just finitely many terms satisfying the area bound.

In the elliptic case  $-\chi=0$: the LHS $=0$.  Note that on the RHS, $n-i, \theta_i, \delta, 1-\iota_j$ are all non-negative.  Hence there are just finitely many possibilities of $n_i$ that makes it vanish.

In the spherical case $-\chi<0$: note that the proof of Lemma \ref{lem:geq0} still works for polygons with even-degree corners.  Thus $m-3(n_1+n_2+n_3) \geq 0$.  Then $\left(m-3\left(n_1+n_2+n_3 + \sum_{i=0}^k s_i\right)\right)$ is bounded below, and hence the LHS is bounded above.  Thus there are just finitely many possibilities of $n_i$ satisfying the above equality.

In the hyperbolic case, $-\chi>0$.  The valuation $m-3(n_1+n_2+n_3)$ being bounded above implies that the LHS is bounded above.  Hence there are just finitely many possibilities of $n_1,n_2,n_3$ satisfying the above equality.  
\end{proof}


We conclude by summarizing some of the results and providing an explicit  formula for the low energy terms of the potential $W_\tau$.

\begin{prop}\label{coro:w-8}
	$W_\tau$ is a convergent series. Moreover, if $val(\tau)>0$, we have
	\begin{equation}\label{eq:Wlow}
	W_\tau= -T^{-8} xyz + x^a+y^b+z^c+ W_{high}, 
	\end{equation}
	where $val(W_{high})\geq \lambda >0$ for some $\lambda$ and the representative for $[\pt]$ is chosen not to intersect the minimal triangles. 
\end{prop}
\begin{proof}
	Convergence of $W_\tau$ is the content of Theorem \ref{thm:bdconv}. Lemma \ref{lem:geq0}, combined with $val(\tau)>0$, shows that the only possible monomials in $W_\tau$ with $val\leq 0$ are exactly the ones in (\ref{eq:Wlow}). We then only need to show why the coefficients of these monomials, which in our setting are computed after CF-perturbations, agree with the count in \cite{CHL17} using another model. 
	
	Consider the case of the monomial $xyz$, whose contribution comes from the two minimal triangles. The relevant moduli spaces are (topologically) six smooth closed intervals, which parameterize the position of the zero-th marked point on the six arcs of the Seidel Lagrangian (determined by the self-intersections). When the marked point is away from the self-intersections, the evaluation map $ev_0$ is weakly submersive and in fact a local diffeomorphism. Therefore, following the construction of the Kuranishi structure and CF-perturbations \cite[Section 5]{F10} (or alternatively \cite[Section 12]{FOOO_new}), there is no CF-pertubation (or rather, it is the trivial one on some sub-intervals) away from intersection points. Hence on these sub-intervals the output of \eqref{eq:mktau} is simply a constant function which agrees with the naive count. Since by Proposition \ref{prop:weakunobsbulk}, $\m_0^{\tau,b}$ is constant, we can conclude the contribution of the minimal triangles to $W_\tau$ is exactly $-T^8$ (after the change of variables). The other monomials can be argued in the same way.
\end{proof}

We point out that the signs of the coefficients depend on the choice of the location of a generic point (representing the nontrivial spin structure). Nevertheless, one can always make the potential into the form above by substituting variables with their negatives. 

For convenience, we choose the representative of the point class away from the minimal triangles bounded by the Seidel Lagrangian for the rest of the paper.
Note however that if we choose a representative of point class in one of the minimal triangle, then we may get bulk deformed contribution of the minimal triangle, which
have a negative valuation. Nevertheless, the coefficient $\tau_2$ of $[\pt]$ has a non-negative valuation, and hence does not affect the valuation of the coefficient of $xyz$. More specifically, the potential still admits an expression
\begin{equation}\label{eqn:leadingterms-xixyz}
	W_\tau=  -\xi xyz + x^a+y^b+z^c+ W_{high}, 
	\end{equation}
and we still have $val(\xi) = -8$ in this case. 
We refer to Section 5 for a more detailed study of the dependence of $W_\tau$ on the choice of a representative of the point class.

\subsection{\texorpdfstring{Fukaya algebra of $\BL$}{Fukaya algebra of L}}\label{sec:FA}

In this section we will give a partial description of the Fukaya algebra $\CF(\BL,\tau,\Lambda\langle\langle x,y,z\rangle\rangle)$ described in Proposition \ref{prop:fukalgconv}.
In fact it is more convenient to work on the canonical model of $\CF(\BL,\tau,\Lambda\langle\langle x,y,z\rangle\rangle)$. 
That is, we use the homological perturbation lemma to transfer the $\AI$-algebra structure to $H^*(\BL)$.  
We denote by $HF^*(\BL, \tau, \Lambda\langle\langle x,y,z\rangle\rangle)$ the canonical model and denote the $\AI$ operations by $\m_{k,can}^{\tau,b}$.  

\begin{lemma}\label{lem:m2rel}
	Let $p$ be the odd degree generator of $H^*(\bS^1)$. We have the following identities:
	$$\m_{2,can}^{\tau,b}(X,Y)=\bar{Z}T+ T^3\xi_1 + T^{-2}d_1 \one_\BL ,$$
	$$\m_{2,can}^{\tau,b}(Y, Z)=\bar{X}T+ T^3\xi_2 + T^{-2}d_2 \one_\BL,$$
	$$\m_{2,can}^{\tau,b}(Z, X)=\bar{Y}T+ T^3\xi_3 + T^{-2}d_3 \one_\BL,$$
	$$\m_{2,can}^{\tau,b}(X,\bar{X})=(1+c_1)p+\eta_1,$$
	$$\m_{2,can}^{\tau,b}(Y,\bar{Y})=(1+c_2)p+\eta_2,$$
	$$\m_{2,can}^{\tau,b}(Z,\bar{Z})=(1+c_3)p+\eta_3,$$
	where each $\xi_i$ is a linear combination of $\bar{X},\bar{Y},\bar{Z}$ with $val(\xi_i)\geq0$, $val(d_i)\geq0$; each $\eta_i$ is a linear combination of $X,Y,Z$ and $c_i$ is an element of $\Lambda_+$.
\end{lemma} 
\begin{proof}
First recall that the operations on the canonical model are given by summing over trees decorated by the original operations $\m_{k}$. Therefore the minimal area terms contributing to $\m_{2,can}^{\tau,b}$ will be the same as $\m_{2}^{\tau,b}$. Hence the minimal triangle with $X,Y,Z$-corners in counter-clockwise order gives the first term in formula above for $\m_{2,can}^{\tau,b}(X,Y)$. Also, the last term is essentially $\partial^2 / \partial x \partial y$ applied to the bulk deformed potential evaluated at $(T^3 x, T^3 y, z)$. This is because $X = \tilde{x} X|_{\tilde{x} =1}$ should be interpreted as  $x X|_{x = T^3} = T^3 X$ after change of variables (in particular, when computing in $x,y,z$-variables) and similar for $Y$. Therefore $T^{-2}$ again comes from the minimal triangle, and that is the smallest valuation among the terms in the coefficient of $\one_\BL$ by Proposition \ref{coro:w-8}. Finally, $T^3 \xi_1$ comes from the polygons apart from the minimal triangle, where one of their corners are used as outputs. Since the output (one of $\bar{X},\bar{Y},\bar{Z}$) in this case has valuation zero unlike variables $x,y,z$, we have additional $T^3$ (which we would lose if the corresponding corner was not an output).
The valuations in $\m_{2,can}^{\tau,b}(Y, Z)$ and $\m_{2,can}^{\tau,b}(Z, X)$ can be estimated in a similar way.

For $\m_{2,can}^{\tau,b}(X,\bar{X})$, constant triangle contributes to $p$ which gives the first term. In general, $\m_{2,can}^{\tau,b}(X,\bar{X})$ should be of odd degree. In fact, there does not exist a polygon
whose corners are $X,Y,Z$'s except one of $\bar{X},\bar{Y},\bar{Z}$ corner due to orientation of Lagrangian.
Thus any non-trivial polygon contributing to $\m_{2,can}^{\tau,b}(X,\bar{X})$ should have an output in $X,Y,Z$.
%
\end{proof}

\begin{lemma}
	Let $R(X,Y,Z)$ be the subring of $HF^*(\BL, \Lambda\langle\langle x,y,z\rangle\rangle)$ generated by $X,Y,Z$. There exists $r, s, t$ in the closure of $R(X,Y,Z)$ and $q_1,q_2,q_3 \in \Lambda\langle\langle x,y,z\rangle\rangle$ such that
	$$\bar{X} = r + q_1 \one_\BL,$$
	$$\bar{Y} = s + q_2 \one_\BL,$$
	$$\bar{Z} = t + q_3 \one_\BL.$$
\end{lemma}
\proof
 The proof of the three statements is identical, we prove the last one. First we prove by induction, that for each integer $k\geq1$ there are $t_k\in R(X,Y,Z)$ and $c_k \in \Lambda\langle\langle x,y,z\rangle\rangle$ such that $E_k=\bar{Z}- t_k-c_k 1_L$ is a linear combination of $\bar{X},\bar{Y},\bar{Z}$ and
 $$val(\bar{Z}- t_k-c_k 1_L)\geq 2k, \ val(t_{k}-t_{k-1})\geq 2k-3, \ val(c_{k}-c_{k-1})\geq 2k-3.$$
 
 The first equation in Lemma \ref{lem:m2rel}, gives the case $k=1$ with $t_1=T^{-1}\m_{2,can}^{\tau,b}(X,Y)$  and $c_1=T^{-3}d_1$. Assuming the statement for $k$, we have 
 $$E_k=T^{2k}\left(\alpha\bar{X}+\beta\bar{Y}+ \gamma\bar{Z}  \right),$$
 where $val(\alpha,\beta,\gamma)\geq0$. Now using the formulas for $\bar{X},\bar{Y},\bar{Z}$ given by the first three equations in Lemma \ref{lem:m2rel} we can write
 $$E_k=T^{2k}\left(T^{-3}R+ d 1_L + T^{2}J  \right),$$
 where $R\in R(X,Y,Z) $, $J$ is a linear combination of $\bar{X},\bar{Y},\bar{Z}$ and $val(R)\geq0$, $val(d)\geq-3$, $val(J)\geq0$.
 Therefore we have by induction
 $$\bar{Z}=t_k+T^{2k-3}R + (c_k+dT^{2k})1_L + T^{2k+2}J.$$
 Hence we can take $t_{k+1}=t_k+T^{2k-3}R$ and $c_{k+1}=c_k+dT^{2k}$, satisfying the conditions required.
 
 Finally we simply take the limits $t=\lim_{k} t_k$ and $c=\lim_{k} c_k$.
\qed

We are now ready to prove the main proposition in this subsection.

\begin{prop}\label{prop:Imdiff}
The image of $\m_{1,can}^{\tau,b}$ is contained in the Jacobian ideal. That is, we have the following
	$$\textrm{Im}\left(\m^{\tau,b}_{1,can}\right) \subset  \ \langle \partial_x W_\tau, \partial_y W_\tau, \partial_z W_\tau\rangle \cdot HF^*(\BL, \Lambda\langle\langle x,y,z\rangle\rangle).$$	
\end{prop}
\begin{proof}
In order to prove this proposition one first differentiates the Maurer-Cartan equation
\begin{align*}
\frac{\partial W}{\partial \tilde{x}}(\tau,b)\cdot \one_\BL=\sum_{k_1,k_2\geq 0}\m^\tau_{k_1+k_2+1}(\overbrace{b,\ldots,b}^{k_1},X,\overbrace{b,\ldots,b}^{k_2})=\m_1^{\tau,b}(X).
\end{align*}
We have analogous identities for $Y$ and $Z$. Taking the change of variables into account we have
$$\m_{1,can}^{\tau,b}(X)=T^3\frac{\partial W_\tau}{\partial x}(b)\one_\BL, \ \m_{1,can}^{\tau,b}(Y)=T^3\frac{\partial W_\tau}{\partial y}(b)\one_\BL \ \textrm{and} \ \m_{1,can}^{\tau,b}(Z)=T^3\frac{\partial W_\tau}{\partial z}(b)\one_\BL.$$

By the previous lemma,

$$\m_{1,can}^{\tau,b}(\bar{X})= \m_{1,can}^{\tau,b}(r) + q_1\m_{1,can}^{\tau,b}(1_L)= \m_{1,can}^{\tau,b}(r) ,$$
 
 Now, given the Leibniz rule for $\m_{1,can}^{\tau,b}$ and $\m_{2,can}^{\tau,b}$,  we have that $\m_{1,can}^{\tau,b}(R(X,Y,Z))$ is contained in the Jacobian ideal. Recall from \cite[Section 5.2.7]{BGR}, that the Jacobian ideal, like any ideal in the Tate algebra is closed. Therefore $\m_{1,can}^{\tau,b}$ of the closure of $R(X,Y,Z)$ is also in the Jacobian ideal. Hence $\m_{1,can}^{\tau,b}(\bar{X})$ is in the Jacobian ideal. The same is true for $\bar{Y}$ and $\bar{Z}$.
 	
Finally, from the fourth equation in Lemma \ref{lem:m2rel} we have
$$\m_{1,can}^{\tau,b}(p)=(1+c_1)^{-1}\left(\m_{1,can}^{\tau,b}(\m_{2,can}^{\tau,b}(X,\bar{X}))+\m_{1,can}^{\tau,b}(\eta_1)\right),$$
since $1+c_1$ is invertible.
Again, it follows from the Leibniz rule that the first term on the right is in the Jacobian ideal. Recall, by construction $\eta_1$ is a linear combination of  $X,Y,Z$ and therefore we conclude that $\m_{1,can}^{\tau,b}(\eta_1)$ is in the Jacobian ideal, which completes the proof.
\end{proof}

\begin{prop}\label{floerhom}
	The cohomology $HF^*(\BL,\tau, b)$ is nonzero if and only if $(x,y,z)$ (corresponding to $b$) is a critical point of $W_\tau$. 
	In this case, $HF^*(\BL, \tau, b)$ is isomorphic to 
	$$H^*(\BL):=\left(H^*(\bS^1)\oplus\bigoplus_{X,Y,Z}\Lambda_0^{\oplus 2}\right)\otimes\Lambda,$$ 
	as a vector space.
\end{prop}
\begin{proof}
From the previous proposition we have
$$\m_{1,can}^{\tau,b}(X)=T^3\frac{\partial W_\tau}{\partial x}(b)\one_\BL, \ \m_{1,can}^{\tau,b}(Y)=T^3\frac{\partial W_\tau}{\partial y}(b)\one_\BL \ \textrm{and} \ \m_{1,can}^{\tau,b}(Z)=T^3\frac{\partial W_\tau}{\partial z}(b)\one_\BL.$$

Since $\one_\BL$ is the identity in $HF^*(\BL, \tau, b)$, this immediately implies that $HF^*(\BL, \tau, b)$ is zero if $(x,y,z)$ is not a critical point of $W_\tau$. 

For the converse, note that Proposition \ref{prop:Imdiff} implies that $\textrm{Im}(\m_{1,can}^{\tau,b})=0$, when $(x,y,x)$ is a critical point. Hence the differential $\m_{1,can}^{\tau,b}=0$, which implies that $HF^*(\BL, \tau, b)$ is isomorphic to $H^*(\BL)$.
\end{proof}

Moreover, one can show that when $b$ is a critical point, $HF^*(\BL,\tau, b)$ is isomorphic, as a ring, to the Clifford algebra associated to the Hessian of $W_\tau$ at the point $b$. But we will not make use of this fact.

\section{Dependence of the potential on chain level representatives of bulk}\label{sec:ptdependence}
In this section, we prove that if we change the $\Z/2$-representative $\tau^2\mathbf{p}$ for the $H^2(\Pabc, \Lambda_0)$ component of the bulk deformation, the associated potentials are
related by a coordinate change. We start with the definition of this notion.

\begin{defn}\label{defn:coord}
	A coordinate change is a map $\varphi: \Lambda\langle\langle x',y',z'\rangle\rangle\rightarrow\Lambda\langle\langle x,y,z\rangle\rangle$ of the form
	$$x'\to c_1 x + u_1, \ y'\to c_2 y + u_2, \ z'\to c_3 z + u_3,$$
	where $c_i\in \mathbb{C}^*$ and $u_i \in \Lambda\langle\langle x,y,z\rangle\rangle$ satisfy $val(u_i)>0$, for $i=1,2,3$.
\end{defn}

\begin{prop}\label{prop:tau2tau2'cc}
Let $\tau_2$ and $\tau_2'$ be two $\Z/2$-invariant representatives of the same class in $H^2(\Pabc, \Lambda_0)$.  Their associated potentials $W_{\tau_2}$ and $W_{\tau_2'}$ are related by a coordinate change
$$ x'= \exp (k_X) x, \quad y'= \exp (k_Y) y, \quad  z'= \exp (k_Z) z,$$
with $k_X, k_Y, k_Z \in \Lambda_0$, i.e., $W_{\tau_2} (x',y',z') = W_{\tau_2'} (x,y,z)$. The coefficients $k_X, k_Y, k_Z$ are given in Lemma \ref{lemma:kkkxyzrefinv}.
\end{prop}

The main ingredient in the proof of this result is the following topological lemma.

\begin{lemma}\label{lemma:kkkxyzrefinv}
	Suppose $\tau_2$ and $\tau_2'$ are two reflection-invariant, cohomologous cycles, and let $Q=\tau_2-\tau_2'$. Given an (orbi-)polygon $\beta$ contributing to the potential $W_\tau$, denote by $\beta(X), \beta(Y)$ and $\beta(Z)$ the numbers of $X$, $Y$ and $Z$ corners, respectively. Then
	there exist $k_X$, $k_Y$ and $k_Z \in \Lambda_0$ such that
	$$- Q \cap \beta  =  k_X \beta(X) + k_Y \beta(Y) + k_Z \beta(Z)$$
	for any (orbi-)polygon $\beta$ contributing to the potential $W_\tau$. 
\end{lemma}

\begin{proof}
Since $Q=\tau_2 - \tau_2'$ is cohomologous to zero we can choose a smooth 1-(co)cycle $R$ such that $\partial R = Q$. Here, we are abusing notations for cycles and cocycles via Poincar\'e duality (see Figure \ref{fig:pbetakxkykz} (a)). Moreover we can choose $R$ so that it is reflection invariant and avoids $X, Y,Z$. Then we have $Q \cap \beta = -R \cap \partial \beta$ for any (orbi-)polygon class $\beta$ which contributes to the potential.

Let $\arc{XY}_{+}$ denote the minimal segment of $\mathbb{L}$ between $X$ and $Y$ lying on the upper hemisphere and $\arc{XY}_{-}$ denote its reflection image,  see Figure \ref{fig:pbetakxkykz} (a). We analogously define $\arc{YZ}_{\pm}$ and $\arc{ZX}_{\pm}$. Define $i,j,k$ by $i:= R \cap \arc{XY}_{+}$, $j:= R \cap \arc{YZ}_{+}$ and $k:= R \cap \arc{ZX}_{+}$. Since $R$ is reflection invariant $R \cap \arc{XY}_{-}= - R \cap \arc{XY}_{+}$, therefore $R\cap \partial \beta$ equals
\begin{equation}\label{eq:Rbeta}
\pm\left( i\left(\beta \left(\arc{XY}_{+} \right)- \beta \left(\arc{XY}_{-} \right)\right)+j \left(\beta \left(\arc{YZ}_{+} \right)- \beta \left(\arc{YZ}_{-} \right)\right)+ k\left(\beta \left(\arc{ZX}_{+} \right)- \beta \left(\arc{ZX}_{-} \right)\right)\right)
\end{equation}
where $\beta \left(\arc{XY}_{\pm}\right)$ is the number of $\arc{XY}_{\pm}$-segments in $\partial \beta$ (and analogously for $YZ$ and $ZX$). The plus or minus depends on $\beta$, having its boundary orientation match with that of $\mathbb{L}$ or its opposite. The former was called a positive polygon in \cite{CHKL14}, and the latter a negative polygon for this reason.

\begin{figure}[h]
\begin{center}
\includegraphics[height=2.5in]{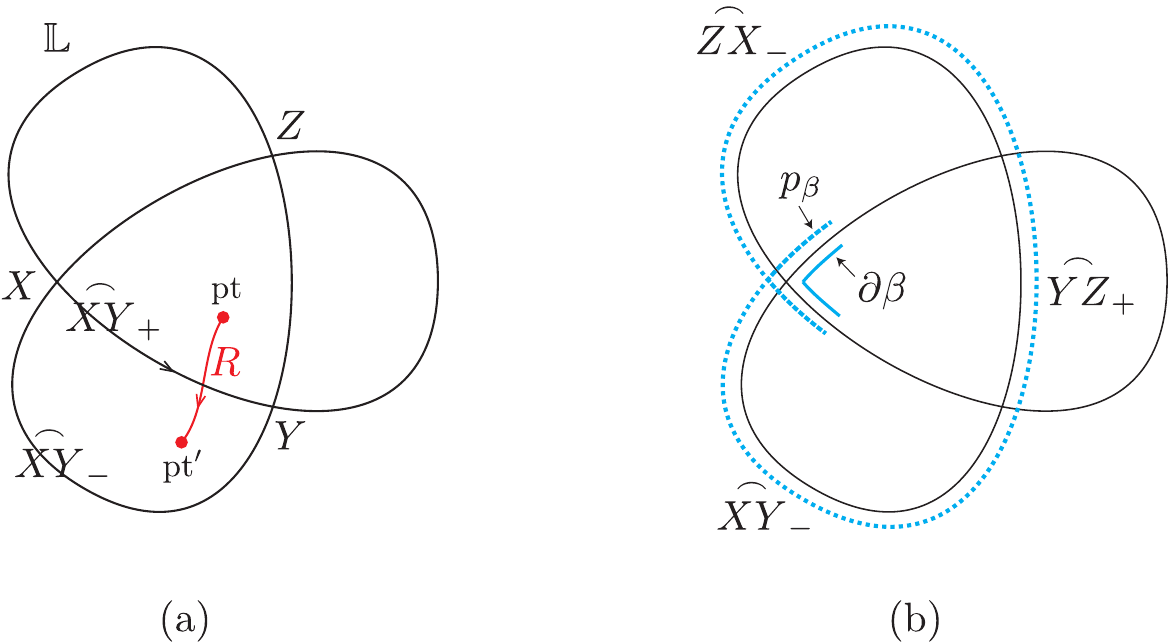}
\caption{}\label{fig:pbetakxkykz}
\end{center}
\end{figure}

Let us assume that $\beta$ is positive, we claim that 
\begin{equation}\label{eq:betaarc}
\beta \left(\arc{XY}_{+} \right)- \beta \left(\arc{XY}_{-} \right) = \beta(X) + \beta(Y) - \beta(Z).
\end{equation}
We consider the loop $p_\beta$ in $\pi_1 (\mathbb{L})$ obtained by attaching three consecutive minimal segments to each corner of $\partial \beta$. Namely, $p_\beta$ near a corner of $\beta$ is parameterized in such a way that we walk past the corner without turning and coming back to the same vertex, rather than jumping to another branch at the corner. See Figure \ref{fig:pbetakxkykz} (b).

Since $p_\beta$ is an integer-multiple of $[\mathbb{L}]$, we have
\begin{equation}\label{eqn:pbetapmxy}
p_\beta \left(\arc{XY}_{+} \right) - p_\beta \left(\arc{XY}_{-} \right) =0.
\end{equation}
Thus $\beta \left(\arc{XY}_{+} \right) - \beta \left(\arc{XY}_{-} \right)$ can be computed from the extra minimal segments that are attached to $\beta$ to obtain $p_\beta$. For instance, the extra segments attached to $X$-corner of $\beta$ consists of $\arc{XY}_-$, $\arc{YZ}_+$ and $\arc{ZX}_-$, and hence removing these from $p_\beta$ increases \eqref{eqn:pbetapmxy} by 1. Likewise, for each $Y$-corner (resp. $Z$-corner) of $\beta$, removing extra segments from $p_\beta$ increases (resp. decreases) \eqref{eqn:pbetapmxy} by $1$. Therefore we conclude the claim.

When $\beta$ is negative, formula (\ref{eq:betaarc}) still holds but now with a minus sign on the right-hand side. There are analogous formulas for the arcs $YZ$ and $ZX$. Combining all of these with \eqref{eq:Rbeta} we obtain
$$R \cap \partial \beta = (i-j+k) \beta(X) + (i+j-k) \beta(Y) + (-i+j+k) \beta(Z),$$
which gives the desired result.
\end{proof}

\begin{proof}[Proof of Proposition \ref{prop:tau2tau2'cc}]

Let us expand our bulk-deformed potential as
$$W_{\tau_2} (x,y,z) = \sum_\beta \exp (\tau_2  \cap \beta) c_{\beta, \tau_{tw}} x^{\beta(X)} y^{\beta(Y)} z^{\beta(Z)} T^{\omega(\beta)} $$
where the sum is taken over the set of all (orbi-)polygon classes. Here, the potential also depends on other bulk parameters as the coefficient $c_{\beta,\tau_{tw}}$ shows, but we wrote $W_{\tau_2}$ to highlight its dependence on $[\pt]$ which is what we want to analyze. $W_{\tau_2'}$ can be written analogously.

Using Lemma \ref{lemma:kkkxyzrefinv}, one can compute $W_{\tau_2} (x',y',z')$ as follows 
\begin{align*}
W_{\tau_2} (x',y',z') &= \sum_\beta \exp (\tau_2  \cap \beta ) c_{\beta, \tau_{tw}} x'^{\beta(X)} y'^{\beta(Y)} z'^{\beta(Z)} T^{\omega(\beta)} \\
&=  \sum_\beta \exp (\tau_2  \cap \beta ) c_{\beta, \tau_{tw}} \exp ( k_X \beta(X) + k_Y \beta(Y) + k_Z \beta(Z)) x^{\beta(X)} y^{\beta(Y)} z^{\beta(Z)} T^{\omega(\beta)} \\
&= \sum_\beta \exp (\tau_2  \cap \beta + R \cap \partial \beta ) c_{\beta, \tau_{tw}} x^{\beta(X)} y^{\beta(Y)} z^{\beta(Z)} T^{\omega(\beta)} \\
&= \sum_\beta \exp (\tau_2  \cap \beta - Q \cap   \beta ) c_{\beta, \tau_{tw}} x^{\beta(X)} y^{\beta(Y)} z^{\beta(Z)} T^{\omega(\beta)} \\
&= \sum_\beta \exp (\tau_2'  \cap \beta ) c_{\beta, \tau_{tw}} x^{\beta(X)} y^{\beta(Y)} z^{\beta(Z)} T^{\omega(\beta)} \\
&= W_{\tau_2'} (x,y,z),
\end{align*}
which is the desired result.
\end{proof}

\section{The Kodaira--Spencer map}\label{sec:KSringhom}

In this section we will define the Kodaira-Spencer map in our setting. This is a map
$ \KS_\tau : \QH^*_{orb}(X,\tau)\longrightarrow \Jac(W_\tau)$ from the quantum cohomology of $\bP^1_{a,b,c}$ to the Jacobian ring of $W_\tau$, constructed geometrically using $J$-holomorphic discs. 
The only previously known construction of $\KS_\tau$ map is the case of toric manifolds by Fukaya-Oh-Ohta-Ono \cite{FOOO_MS}. We will follow
the line of their construction. Their construction heavily uses the $T^n$-action on the moduli space of holomorphic discs.
In our construction,  the $\Z/2$-action (reflection on the equator) will play an analogous role.  We show that the $\KS_\tau$ map is well-defined (independent of the choice of cohomology representative) and it is a ring homomorphism.

\subsection{\texorpdfstring{Definition of $\KS_\tau$ and well-definedness}{Definition of KS and well-definedness}}

We start by defining the Jacobian ring of $W_\tau$.
Recall that we use the convergent power series ring $\Lambda \langle\langle x,y,z\rangle\rangle$ defined in Definition \ref{def:cpr}.
\begin{defn}\label{def:jac}
	Consider $P\in\Lambda \langle\langle x,y,z\rangle\rangle$. We define the \emph{Jacobian ring of $P$} as the ring
	$$\Jac(P)=\frac{\Lambda \langle\langle x,y,z\rangle\rangle}{<\partial_x P,\partial_y P,\partial_z P>}.$$
\end{defn} 

We would like to point out that there is no need to take closure of the ideal, since in $\Lambda \langle\langle x,y,z\rangle\rangle$ (as a Tate algebra) all ideals are closed, see \cite[Section 5.2.7]{BGR}.

\begin{remark}
	We saw in Section 5 that $W_\tau$ is well defined up to a change of variables. Since a change of variables induces a ring isomorphism on the corresponding Jacobian rings, we see that the Jacobian ring $\Jac(W_\tau) $ is well defined up to isomorphism.
\end{remark}

Let $w_0,\cdots,w_B$ be coordinates of $\tau$ with respect to the basis $\{f_i\}_{i=0}^B$ (i.e. $\tau = \sum_i w_i f_i$).
The potential function $W_\tau$  can be regarded as a function $W_\tau(w_0,\cdots,w_B, x,y,z)$ with $w_i \in \Lambda_+$.
Regarding $\tau$ as an element of $H^*_{orb}(X,\Lambda)$, we identify the tangent space $T_\tau H^*_{orb}(X,\Lambda)$ at $\tau$ with $\QH^*_{orb}(X,\tau)$.

\begin{defn}\label{def:KS}
	We define the \emph{Kodaira-Spencer map}
	$$\KS_\tau: \QH^*_{orb}(X,\tau)\longrightarrow \Jac(W_\tau)$$ by the formula 
	$$ \KS_\tau( \frac{\partial}{\partial w_i}) =  \frac{\partial W_\tau}{\partial w_i} $$
\end{defn}


There is an ambiguity of the choice of representatives $f_i$ in $\QH^*_{orb}(X,\tau)$.
In our case, the twisted sectors as well as the fundamental cycle have canonical representatives.  Hence we only need to consider the choice
of $\tau_2$.
\begin{lemma}
	The map $\KS_\tau$ is well-defined. In other words, if $\tau_2=\partial R$ for some 1-cocycle $R$, then $\KS_\tau( \frac{\partial}{\partial \tau_2})=0$ in $\Jac(W_\tau)$ .
\end{lemma}
\begin{proof}
	As in the proof of Proposition \ref{prop:tau2tau2'cc} we write 
	$$W_{\tau} (x,y,z) = \sum_\beta \exp (\tau_2  \cap \beta) c_{\beta, \tau_{tw}} x^{\beta(X)} y^{\beta(Y)} z^{\beta(Z)} T^{\omega(\beta)}. $$
	Then, by definition we have
	$$\KS_\tau( \frac{\partial}{\partial \tau_2}) = \sum_\beta (\tau_2  \cap \beta )\exp (\tau_2  \cap \beta )  c_{\beta, \tau_{tw}} x^{\beta(X)} y^{\beta(Y)} z^{\beta(Z)} T^{\omega(\beta)}.$$
	By assumption, $\tau_2=\partial R$, then by Lemma \ref{lemma:kkkxyzrefinv}, there are $k_X, k_Y, k_Z$ such that $-Q \cap \beta  =  k_X \beta(X) + k_Y \beta(Y) + k_Z \beta(Z)$ for all $\beta$. Therefore
	
	\begin{eqnarray*}
		\KS_\tau( \frac{\partial}{\partial \tau_2}) &=& \sum_\beta (-k_X \beta(X) - k_Y \beta(Y) - k_Z \beta(Z) )\exp (\tau_2  \cap \beta )  c_{\beta, \tau_{tw}} x^{\beta(X)} y^{\beta(Y)} z^{\beta(Z)} T^{\omega(\beta)} \\
		&=&  -k_X x \frac{\partial W_\tau}{\partial x} - k_Y y \frac{\partial W_\tau}{\partial y} - k_Z z \frac{\partial W_\tau}{\partial z}.
	\end{eqnarray*}
	Therefore $\KS_\tau( \frac{\partial}{\partial \tau_2})=0$ in $\Jac(W_\tau)$
\end{proof}

Here is an alternative description of $\KS_\tau$, for indices $i$ such that $f_i$ is a twisted sector. 
The derivative $\frac{\partial}{\partial w_i}$ has the effect of removing the variable $w_i$ in one of the $\tau= \sum_i w_i f_i$ insertions on the disc.
Therefore we have the following expression
$$ \KS_\tau(f_i)= \sum_{\beta, k} \exp( \tau^2 \mathbf{p} \cap \beta) \sum_{l=0}^\infty \frac{T^{\beta\cap \omega}}{l!}\mathfrak{q}_{l+1,k,\beta}(f_i, \tau_{tw}^l; b,\ldots, b).$$
We would like to have an analogous description for the cases of the fundamental and point classes in $\Pabc$. More concretely, let $\mathbf{Q}$ be a $\Z/2$-invariant cycle in $\Pabc$, and define the moduli spaces
$$ \mathcal{M}^{main}_{k+1,l+1}(\beta, \mathbf{Q}, \tau_{tw},\gamma)=  \mathcal{M}^{main}_{k+1,l+1}(\beta, \tau_{tw},\gamma)\times_X \mathbf{Q}.$$
Using these spaces and their evaluation maps, analogously to (\ref{eq:qmap}) we can define maps $\mathfrak{q}_{l+1,k,\beta}(\mathbf{Q}, \tau_{tw}^l, - )$. Then we have the following statement.

\begin{prop}\label{prop:divisor}
	Let $\mathbf{Q}$ be a cycle representing the fundamental cycle or the point class in $\Pabc$.  Then
	$$\KS_\tau(\mathbf{Q})= \sum_{\beta, k} \exp( \tau^2 \mathbf{p} \cap \beta) \sum_{l=0}^\infty \frac{T^{\beta\cap \omega}}{l!}\mathfrak{q}_{l+1,k,\beta}(\mathbf{Q}, \tau_{tw}^l; b,\ldots, b),$$
	in $\Jac(W_\tau)$.
\end{prop}

This proposition essentially asserts that the $\mathfrak{q}$ maps are unital and satisfy a version of the divisor axiom in Gromov-Witten theory. Both properties are related to the compatibility of the Kuranishi structures (and perturbations) on the moduli spaces of discs with forgetting interior marked points. It turns out that ensuring this compatibility for all moduli spaces seems a rather complicated problem. We will avoid tackling that problem by taking a homotopy between the usual Kuranishi structures on $\mathcal{M}^{main}_{k+1,l+1}(\beta, \mathbf{Q}, \tau_{tw},\gamma)$ and one constructed specifically to ensure this compatibility. Therefore the equality in the statement holds only in the Jacobian ring, but not necessarily at chain-level. We will postpone this proof to Appendix \ref{sec:divisor}.

\subsection{Ring homomorphism}\label{subsec:ringhomkura}

In this subsection we will prove the following

\begin{theorem}\label{KSring}
	The map $\KS_\tau:\QH^*_{orb}(X,\tau)\longrightarrow\Jac(W_\tau)$ is a ring homomorphism.
\end{theorem}
This map is rather surprising in that it identifies the complicated quantum multiplication with the standard multiplication of polynomials in Jacobian ring.
The geometric idea behind this map is rather well-known. Namely, the closed-open maps in topological conformal field theory are ring homomorphisms from the closed theory to Hochschild cohomology of the open theory.  They are explored in Seidel \cite{Seidel-I}, Biran-Cornea \cite{BC}, Fukaya-Oh-Ohta-Ono \cite{FOOO}.
A benefit of this construction is that Hochschild cohomology of the Fukaya category is very heavy object to handle, whereas the construction of $\KS_\tau$ map is rather direct and simple.

\proof
As before, we will follow the line of proof of \cite{FOOO_MS} Theorem 2.6.1 and we will use their notation freely to shorten our exposition.
The proof is based on a cobordism argument. Consider two cohomology representatives $A, B$ in $\QH^*_{orb}(X,\tau)$.
Consider the forgetful map for the moduli space introduced in Section \ref{subsec:bdeffukalg}
$$\mathfrak{forget}: \mathcal{M}^{main}_{k+1,l+2}(\beta, A \otimes B \otimes \tau_{tw}^{\otimes l},\gamma)  \to \mathcal{M}^{main}_{1,2}$$
which forgets maps and shrinks resulting unstable components if any, followed by the forgetful map $\mathcal{M}^{main}_{k+1,l+2} \to
\mathcal{M}^{main}_{1,2}$ forgetting the boundary marked points, except the first one and forgetting the interior marked points except the first two.  
In Lemma 2.6.3 \cite{FOOO_MS}, $\mathcal{M}^{main}_{1,2}$ is shown to be topologically a disc with some stratification, so that the above $\mathfrak{forget}$ is a continuous and stratified smooth submersion.

 The idea of the proof is to consider a curve in  $\mathcal{M}^{main}_{1,2}$ which connects two point strata of $D^2$. $\Sigma_0$ is a stratum where  two interior marked points lie on a sphere bubble, and
$\Sigma_{12}$ is a component where there are two disc bubbles each of which contains one of the interior marked points.
We will see that {\it integration} over $\mathfrak{forget}^{-1}(\Sigma_0)$ and  $\mathfrak{forget}^{-1}(\Sigma_{12})$ correspond to $\KS_\tau(A \bullet_\tau B)$ and $\KS_\tau(A)\KS_\tau(B)$ respectively and that the pre-image of the curve defines the desired cobordism relation between them.
There is a technical issue in that the map $\mathfrak{forget}$ is only a stratified submersion. We will explain below how to handle this issue following \cite{FOOO_MS}.

We first have to construct Kuranishi structures and continuous family of multi-sections on (neighborhoods of) the spaces $\mathfrak{forget}^{-1}(\Sigma_0)$ and  $\mathfrak{forget}^{-1}(\Sigma_{12})$. 
To describe the neighborhood near $\Sigma_0$, we consider the following moduli space.
For $\alpha \in H_2(X,\Z)$, let $\mathcal{M}_l(\alpha)$ the moduli space of stable maps from genus zero closed orbifold Riemann surface with $l$-marked points and of homology class $\alpha$.  Take the fiber product
$$\mathcal{M}_{l_1+3}(\alpha, A \otimes B \otimes \tau_{tw}^{l_1})
= \mathcal{M}_{l_1+3}(\alpha)_{(ev_1,\cdots,ev_{l_1+2})} \times_{IX^{l_1+2}} ( A \otimes B \otimes \tau_{tw}^{l_1})$$
(where $IX$ denotes the inertia orbifold of $X$).
Then $ev_{l_1+3}$ defines an evaluation map from the above moduli space to $X$.
Define the moduli space $\mathcal{M}_{k+1, l_1,l_2}(\alpha, \beta; A,B,K)$ to be the fiber product
$$\big(\mathcal{M}_{l_1+3}(\alpha, A \otimes B \otimes \tau_+^{l_1})\times \mathcal{M}^{main}_{k+1,l_2+1}(\beta,\tau_+^{l_2}) \big) \times_{(IX \times IX)} K $$
for a chain $K$ in $IX \times IX$.

Let us consider the case that $K = \Delta'$ defined as
\begin{equation}\label{del}
	\Delta' = \{ (x,g), (x, g^{-1})\} \subset IX \times IX
\end{equation}
for the inertia orbifold $IX$.
The following is an analogue of Lemma 2.6.9 \cite{FOOO_MS}, to which we refer readers for the proof.
\begin{lemma}
	There exist a surjective map 
	\begin{equation}\label{eqglue}
		\mathfrak{Glue}: \bigcup_{\alpha \sharp \beta' = \beta} \bigcup_{l_1+l_2=l} \mathcal{M}_{k+1, l_1,l_2}(\alpha, A,B,\Delta')
		\to \mathfrak{forget}^{-1}(\Sigma_0)
	\end{equation}
	which defines an isomorphism of spaces with Kuranishi structures, away from codimension 2 strata. Here we are using the stratification of the moduli spaces given by the combinatorial type of the curves.
\end{lemma}
The $\mathfrak{Glue}$ map gives a way to describe an element of  $\mathfrak{forget}^{-1}(\Sigma_0)$ as a fiber product of sphere and disc moduli space.
In this way, it corresponds to first taking the quantum multiplication and then taking the Kodaira-Spencer map.
In the case that there are several sphere bubbles attached to a disc component (codimension higher than 2), there may be several ways of
such description. Namely, $\mathfrak{Glue}$ map image may overlap in codimension two strata. 

The more important issue is the compatibility of Kuranishi perturbations to be chosen, where there are differences between toric and our cases.
Recall that the proof of \cite{FOOO} uses $T^n$-action on moduli space of holomorphic discs in an essential way. Because finite group symmetry is much easier to handle than $T^n$-symmetry, many of the arguments simplify in our $\Z/2$-symmetry case.

First note that finite group symmetry can be easily achieved for Kuranishi structures - see \cite[Appendix A1.3]{FOOO}.  Hence, one may consider 
$\Z/2$-equivariant Kuranishi structures on moduli space of $J$-holomorphic discs or spheres.

Therefore, we may consider the following Kuranishi structure on $\mathfrak{forget}^{-1}(\Sigma_0)$. We choose a component-wise $\Z/2$-equivariant Kuranishi structure and CF-perturbations on the moduli spaces $\mathcal{M}^{main}_{k+1,l+2}(\beta,A,B,\tau_{tw}^{l})$, following \cite{F10}. Here component-wise means that the Kuranishi structure is compatible with the fiber product description of each of the strata of the disc-sphere stratification (formed by disc or sphere bubbling) - see Definition 4.2.2 in \cite{FOOO_MS} for complete details.

\begin{lemma}\label{lem:cfp}
	There exist component-wise $\Z/2$-equivariant Kuranishi structures and CF-perturbations on the moduli spaces $\mathcal{M}^{main}_{k+1,l+2}(\beta,A,B,\tau_{tw}^{l})$. This Kuranishi structure and perturbations induce Kuranishi structures and perturbations on $\mathfrak{forget}^{-1}(\Sigma_0)$. 
	
	Moreover these Kuranishi structures and perturbations on $\mathfrak{forget}^{-1}(\Sigma_0)$ coincide with the ones induced by the Glue map \eqref{eqglue}, which agree with each other on the overlapped part. 
\end{lemma}
\begin{remark}
	The $T^n$-equivariant analogue of this lemma near $\Sigma_{12}$ is given in Lemma 2.6.23 \cite{FOOO_MS} in the toric case. But this statement near $\Sigma_0$ does not hold in the toric case because the moduli space of $J$-holomorphic spheres cannot be made $T^n$-equivariant. Therefore, the construction
	for $\Sigma_0$ is much more involved than that of $\Sigma_{12}$ in \cite{FOOO_MS}. But since we can impose $\Z/2$-symmetry even for sphere moduli spaces, we can treat both cases in the same way.
\end{remark}
\begin{proof}
	We can choose $\Z/2$-equivariant CF perturbation, following \cite{F10}, on each sphere or disc component. On the overlapped part, two Kuranishi structures can be shown to be isomorphic (using the associativity of fiber products of Kuranishi structures).
	Hence perturbations can be chosen to inductively component-wise to have this compatibility because evaluation map $ev_0$ can be made submersive (see Lemma 3.1 \cite{F10}).
\end{proof}

The multi-section in the neighborhood of $\mathfrak{forget}^{-1}(\Sigma_{12})$ can be constructed in a similar way. 
There exist a surjective map 
$$\mathfrak{Glue}: \bigcup_{\beta_{(0)}+\beta_{(1)}+\beta_{(2)}=\beta} \big( (\mathcal{M}_{k_1+1, l_1+1}(\beta_{(1)}, A \otimes \tau_+^{l_1}) \times
\mathcal{M}_{k_2+1, l_2+1}(\beta_{(2)}, B \otimes \tau_+^{l_2}) ) $$
$$ \;_{(ev_0,ev_0)} \times_{(ev_i,ev_j)} \mathcal{M}_{k_3+3, l_3}(\beta_{(0)},  \tau_+^{l_3})\big) \to \mathfrak{forget}^{-1}(\Sigma_{12})$$
The relationship of Kuranishi structures under the $\mathfrak{Glue}$ map is the same as that of Lemma 2.6.22 \cite{FOOO_MS}
(we consider $\Z/2$-equivariance instead of $T^n$-equivariance), and
we can choose $\Z/2$-equivariant CF-perturbations as in Lemma \ref{lem:cfp} in this case also.

Now that we have Kuranishi structures and CF-perturbations, we can define maps using these moduli spaces. 
For $\Sigma \in \mathcal{M}^{main}_{1,2}$, consider $\mathfrak{forget}^{-1}(\Sigma)$.
Following \eqref{eq:qmap}, we use the evaluation map $ev: \mathfrak{forget}^{-1}(\Sigma) \to \prod_{i=1}^k L(\alpha(i))$
to define 
$$\WT{Z}_{b,\tau}^\Sigma = \sum_{k,\beta, l} \frac{T^{\omega \cap \beta}}{l!} (ev_0)_* 
(ev_1^*b \wedge \cdots \wedge ev_k^*b).$$

\begin{lemma}
	We have
	$$\WT{Z}_{b,\tau}^\Sigma = Z_{b,\tau}^\Sigma \one_\BL,$$
	for some $Z_{b,\tau}^\Sigma \in \Lambda_0$
\end{lemma}
\begin{proof}
	The proof is similar to Proposition \ref{prop:weakunobsbulk}.
	Namely, the output given by immersed sectors vanishes by the reflection argument. Hence, the output is in $\Omega^0(\BL)$, i.e. a function on $\BL$.
	One can consider the boundary configuration of $\mathfrak{forget}^{-1}(\Sigma)$ to conclude that the output is $\m_1^{\tau,b}$-closed.
	Then by the same argument as Lemma \ref{lem:multunit} we conclude it is a constant function.
\end{proof}

Now, pick $\Sigma_{3,i}$ (resp. $\Sigma_{4,i}$) very close to $\Sigma_{12}$ (resp. $\Sigma_0$) in  $\mathcal{M}^{main}_{1,2}$
which converges to $\Sigma_{12}$ (resp. $\Sigma_0$) as $i \to \infty$. 
We have the following analogue of Lemma 2.6.27 of \cite{FOOO_MS}.
\begin{prop}\label{prop:gluing}
$$\lim_{i\to \infty} (ev_0)_* ( \mathfrak{forget}^{-1}(\Sigma_{3,i})) = (ev_0)_* ( \mathfrak{forget}^{-1}(\Sigma_{12})) $$
$$\lim_{i\to \infty} (ev_0)_* ( \mathfrak{forget}^{-1}(\Sigma_{4,i})) = (ev_0)_* ( \mathfrak{forget}^{-1}(\Sigma_{0})) $$
\end{prop}
\begin{proof}
The proof in our case is easier than that of \cite{FOOO_MS}, because of Lemma \ref{lem:cfp}.
Namely, in our case, Glue map is compatible with $\Z/2$-equivariant Kuranishi structures for both $\Sigma_0, \Sigma_{12}$. Therefore, we can just  apply Lemma 4.6.5 \cite{FOOO_MS} which claims the $C^1$-convergence of perturbations as $i \to \infty$.  We remark that this type of convergence was extensively studied in \cite{FOOOexp}.
\end{proof}

We will now show that $Z_{b,\tau}^{\Sigma_{3,i}}$ and $Z_{b,\tau}^{\Sigma_{4,i}}$ are equal in the Jacobian ring. For this purpose we introduce an additional moduli space: choose a smooth curve $\psi$ on the open stratum of $\mathcal{M}_{1,2}$ connecting $\Sigma_{3,i}$ to $\Sigma_{4,i}$ and define $\mathcal{N}_{k+1, l+2}(\beta)= \mathfrak{forget}^{-1}(\psi) \subset \mathcal{M}^{main}_{k+1,l+2}(\beta, A \otimes B \otimes \tau_{tw}^{\otimes l},\alpha)$. Since $\mathfrak{forget}$ is a weakly smooth submersion when restricted to the open stratum, the Kuranishi structure defined in Lemma \ref{lem:cfp} induces a Kuranishi structure on $\mathcal{N}_{k+1, l+2}(\beta)$.  Using the evaluation map  $ev: \mathcal{N}_{k+1, l+2}(\beta) \to \prod_{i=1}^k L(\alpha(i))$, as before, we define 
$$\WT{Y}_{b,\tau}= \sum_{k,\beta, l} \frac{T^{\omega \cap \beta}}{l!} (ev_0)_* 
(ev_1^*b \wedge \cdots \wedge ev_k^*b).$$

By the homological perturbation lemma there is an $A_\infty$-quasi-isomorphism from $\CF(\BL)$ to its canonical model $H^*_{can}(\BL)$. Denote by $\Pi^{b,\tau}$ the first (or linear) component of this quasi-isomorphism, by definition we have $\Pi^{b,\tau}\circ\m_1^{\tau,b}= (\m_1^{\tau,b})_{can}\circ\Pi^{b,\tau}$ and $\Pi^{b,\tau}(\one_\BL)=\one_\BL$. We define
$$Y_{b,\tau}:=\Pi^{b,\tau} (\WT{Y}_{b,\tau}).$$
The following lemma can be proved in the same way as Proposition \ref{prop:fukalgconv} for the maps of $\m_k^{\tau,b}$.
\begin{lemma}
	$\WT{Y}_{b,\tau}$ (resp.  $Y_{b,\tau}$) is a convergent series, more precisely it is as an element of $\CF(L,\Lambda\langle\langle x,y,z\rangle\rangle)$
	(resp. $H^*_{can}(\BL,\Lambda\langle\langle x,y,z\rangle\rangle)$.
\end{lemma}

\begin{prop}\label{prop:cobordism}
	We have the following relation in $H^*_{can}(\BL)$:
	$$(\m_1^{\tau,b})_{can}\left(Y_{b,\tau}\right)=\WT{Z}_{b,\tau}^{\Sigma_{4,i}}-\WT{Z}_{b,\tau}^{\Sigma_{3,i}}= \left(Z_{b,\tau}^{\Sigma_{4,i}} - Z_{b,\tau}^{\Sigma_{3,i}} \right) \one_\BL.$$
	Therefore $Z_{b,\tau}^{\Sigma_{3,i}} = Z_{b,\tau}^{\Sigma_{4,i}}$ in the Jacobian ring $\Jac(W_\tau)$.
\end{prop}
\begin{proof}
	First note that the second statement follows from the first together with the fact, proved in Proposition \ref{prop:Imdiff} that $\textrm{Im}\left(\m^{\tau,b}_{1,can}\right)$ is contained in the Jacobian ideal. Second note that the first statement is equivalent to $\m_1^{\tau,b}\left(\WT{Y}_{b,\tau}\right)=\WT{Z}_{b,\tau}^{\Sigma_{4,i}}-\WT{Z}_{b,\tau}^{\Sigma_{3,i}}$, by definition of $\Pi^{b,\tau}$. In order to prove this relation we describe the boundary of $\mathcal{N}_{k+1, l+2}(\beta)$. As a space with Kuranishi structure the boundary of $\mathcal{N}_{k+1, l+2}(\beta)$ is the union of $\mathfrak{forget}^{-1}(\Sigma_{3,i})$, $\mathfrak{forget}^{-1}(\Sigma_{4,i})$ and the fiber products
	\begin{equation}\label{eq:boundaryN}
         \mathcal{N}_{k_1+1, l_1}(\beta_1)\;_{ev_0}\times_{ev_i} \mathcal{M}_{k_2+2, l_2}(\beta_2, \tau_{tw}^{l_2}) \ \ \textrm{and} \ \ \mathcal{M}_{k_1+2, l_1}(\beta_1, \tau_{tw}^{l_1})\;_{ev_0}\times_{ev_i} \mathcal{N}_{k_1+1, l_1}(\beta_1),  
	\end{equation}
where $\beta_1+\beta_2=\beta$, $k_1+k_2=k, l_1+l_2=l$ and $i \in \{1,\ldots , k_2+1\}$.

Now, using Stokes theorem \cite[Lemma 12.13]{FOOO2} and summing over all $\beta, k,l$ (like in Lemma 2.6.36 \cite{FOOO_MS}), we see that the first product in (\ref{eq:boundaryN}) gives $\m_1^{\tau,b}\left(\WT{Y}_{b,\tau}\right)$. The second product in (\ref{eq:boundaryN}) contributes as zero since $\m_0^{b,\tau}$ is a multiple of the unit and the perturbation in $\mathcal{N}_{k, l}(\beta)$ is compatible with forgetting boundary marked points. Finally, by definition, $\mathfrak{forget}^{-1}(\Sigma_{3,i})$ and $\mathfrak{forget}^{-1}(\Sigma_{4,i})$ give $\WT{Z}_{b,\tau}^{\Sigma_{4,i}}$ and $\WT{Z}_{b,\tau}^{\Sigma_{3,i}}$ respectively. Now the desired relation follows from the Stokes theorem.
\end{proof}

Now we need to relate cohomological intersection product and geometric intersection.
Let $\{f_i\}_{i=1}^m$ be basis of $H_{orb}^*(X)$, $g_{ij}= \langle f_i, f_j \rangle_{PD}$ and $(g^{ij})$ be its inverse matrix.
On this basis, we write $A\bullet_\tau B = \sum_i c_i f_i$.
Let $R$ be a chain in $IX \times IX$ such that 
$$\partial R = \Delta' - \sum_{ij} g^{ij } f_i \times f_j,$$
 and consider the moduli space $\mathcal{M}_{k+1, l_1,l_2}(\alpha, \beta; A,B,R)$ defined above. Using the boundary evaluation maps on these moduli spaces we define
 $$\WT{\Xi}(A,B,K,b) = \sum_{\alpha,\beta,l_1,l_2,k} \frac{T^{\omega \cap (\alpha \sharp \beta)} }{(l_1+l_2)!}(ev_0)_*(ev_1^*b \wedge \cdots \wedge ev_k^*b).$$
 The following lemma can be proved exactly as Lemma 2.6.36 in \cite{FOOO_MS}.
 
\begin{lemma}
	$$\sum_i c_i\KS_\tau(f_i) \one_\BL - \WT{Z}_{b,\tau}^{\Sigma_0} = \m_{1}^{\tau,b}( \WT{\Xi}(A,B,R,b))$$
\end{lemma}

Please note that here we are using the description for $\KS_{\tau}$ provided by Proposition \ref{prop:divisor}.

\begin{prop}\label{prop:sigma0}
	$\KS(A  \bullet_\tau B)$ equals $Z_{b,\tau}^{\Sigma_0}$ modulo the Jacobian ideal.
\end{prop}
\begin{proof}
	As before we can show that  $\WT{\Xi}(A,B,R,b)$ is convergent. Then we apply $\Pi^{b,\tau}$ to the equation in the previous lemma to conclude that $\KS_\tau(A  \bullet_\tau B)\one_\BL$ and $Z_{b,\tau}^{\Sigma_0}\one_\BL$ differ by an element in the image of $(\m_1^{\tau,b})_{can}$. The result now follows from Proposition \ref{prop:Imdiff}.
\end{proof}

\begin{prop}[c.f. Lemma 2.6.29 \cite{FOOO_MS}]\label{prop:sigma12}
We have $$\KS_\tau(A) \cdot \KS_\tau(B) = Z_{b,\tau}^{\Sigma_{12}}.$$
\end{prop}

This proposition is completely analogous to Lemma 2.6.29 \cite{FOOO_MS}. Now combining Propositions \ref{prop:sigma12}, \ref{prop:sigma0}, \ref{prop:cobordism} and \ref{prop:gluing} we obtain the proof of Theorem \ref{KSring}. \qed

\section{\texorpdfstring{$\KS_\tau$ is an isomorphism}{KS is an isomorphism}}\label{sec:KSisomsurjinj}

\subsection{ Surjectivity}

In this subsection we show that $\KS_\tau$ is surjective. We start by computing the lower energy contributions to $\KS_\tau$.

\begin{lemma}\label{lem:imageKS}
	There is $\lambda>0$ (depending on $\tau$) such that:
	\begin{align}\label{eqn:ksabc1}
	\KS_\tau\left(\floor*{\frac{1}{a}}^{\bullet_\tau i}\right)&= x^i \mod T^\lambda, \nonumber\\
	\KS_\tau\left(\floor*{\frac{1}{b}}^{\bullet_\tau j}\right)&=y^j  \mod T^\lambda,  \\
	\KS_\tau\left(\floor*{\frac{1}{c}}^{\bullet_\tau k}\right)&=z^k  \mod T^\lambda, \nonumber\\
    \KS_\tau\left(8 [\pt]\right) &=  -T^{-8}xyz + 3ax^a +3bx^b + 3cz^c \mod T^\lambda,\nonumber
	\end{align}
	where $1\leq i <a$, $1\leq j <b$ and $1\leq k <c$.
\end{lemma}

\begin{proof}
The first order term follows from direct computation. For example, $\frac{1}{a}$-slice of the disc contributing to $x^a$ in $W_\tau$ produces $x$ in the first equation (see Corollary \ref{coro:conelifts} and the preceding discussion for the precise description of these orbi-discs).
Thus it suffices to show that all the higher order terms in the above equations have strictly positive powers in $T$. This directly follows from Lemma \ref{lem:geq0}.
\end{proof}

Lemma \ref{lem:imageKS} together with the fact that $\KS_\tau$ is a ring map, is enough to establish surjectivity.

\begin{prop}\label{prop:surj}
	The map $\KS_\tau$ is surjective.
\end{prop}
\begin{proof}

%
%

Since any element in $\Jac(W_\tau)$ can be written as $T^{-\epsilon} R$ such that $\epsilon >0$ and $R$ only has positive powers in $T$, it is enough to prove that any $R \in \Jac(W_\tau)$ with $\Lambda_0$-coefficients is in the image of $\KS_\tau$. Let $\lambda$ be the minimum of powers of $T$ appearing in the higher order terms in \eqref{eqn:ksabc1}. We claim that for any such $R$ there exists $\rho$ with 
$$R- \KS_\tau(\rho) = T^{\lambda} U$$
where $U$ is also an element in $\Jac (W_\tau)$ with $\Lambda_0$-coefficients only. To see this, write $R$ as
$$R = \sum_{l=1}^{N} a_l T^{\lambda_l} x^{i_l} y^{j_l} z^{k_l} + T^\lambda \tilde{U}$$
where $\tilde{U}$ has positive powers in $T$ (either of summands could be zero even for nonzero $R$) and $\lambda_l <\lambda$. We take $\rho$ to be as follows
$$\rho = \sum_{l=1}^N a_l T^{\lambda_l}  \floor*{\frac{1}{a}}^{\bullet_\tau i_l}\bullet_\tau \floor*{\frac{1}{b}}^{\bullet_\tau j_l} \bullet_\tau \floor*{\frac{1}{c}}^{ \bullet_\tau k_l}.$$
Using the fact that $\KS_\tau$ is a ring homomorphism, Lemma \ref{lem:imageKS} implies that the valuation of $R-\KS_\tau(\rho)$ is no less than $\lambda$. 

We next use this inductively to prove the surjectivity. For $R \in \Jac(W_\tau)$ only with $\Lambda_0$-coefficients, there exists $\rho_1$ such that
$$R - \KS_\tau( \rho_1) = T^{\lambda} R_1.$$
Applying the same to $R_1$, we get $\rho_2$ such that
$$R- \KS_\tau(\rho_1 + T^\lambda \rho_2) = T^\lambda \left( R_1 - \KS_\tau (\rho_2) \right) = T^{2 \lambda} R_3$$
Inductively, one sees that $\sum_{i} T^{(i-1) \lambda} \rho_i$ maps to $R$ under $\KS_\tau$.
\end{proof}

\subsection{Jacobian ring of the leading order potential}
From Proposition \ref{coro:w-8}, we can write
$$T^8 W_\tau=  \mathcal{W}_{lead} + W_+,$$
where 
\begin{equation}\label{eqn:leadingterms-8xyz}
\mathcal{W}_{lead} = -xyz + T^8(x^a+y^b+z^c)
\end{equation}
and $W_+ = T^8 W_{high}$. In particular, we have $val(W_+)=\lambda_0 >8$ for some $\lambda_0$. The coefficient of $xyz$ in $\mathcal{W}_{lead}$ depends on the choice of a representative of $[\pt]$, but we will only consider the case of \eqref{eqn:leadingterms-8xyz} in this section to make our exposition simpler. In general, one can have $\mathcal{W}_{lead} = - \tilde{\xi} xyz + T^8(x^a+y^b+z^c)$ for some $\tilde{\xi}$ with $val \left(\tilde{\xi}\right) =0$ (see \eqref{eqn:leadingterms-xixyz} where $\tilde{\xi}= T^8 \xi$), but the argument below will still apply for any $\tilde{\xi}$ without much change, since what essentially matters is its valuation. 
 We set the following notation
$$g_1=\partial_{x}\mathcal{W}_{lead}, \ g_2=\partial_{y}\mathcal{W}_{lead}, \ g_2=\partial_{z}\mathcal{W}_{lead}. $$
Moreover let $\gamma_1,\cdots,\gamma_N$ denote the following set of elements in $\Lambda\langle\langle x,y,z\rangle\rangle$:
\begin{equation}\label{eq:basis}
	1,x,x^2,\cdots,x^{a-1},y,\cdots,y^{b-1},z,\cdots,z^{c-1},xyz.
\end{equation}

\begin{theorem}
Let $\mathcal{A}_0$ be the Jacobian ring of $\mathcal{W}_{lead}$, that is $\mathcal{A}_0:=\Lambda\langle\langle x,y,z\rangle\rangle/\langle g_1,g_2,g_3 \rangle$. We have the following:

\begin{enumerate}\label{as:def}
\item $\{ \gamma_1,\ldots,\gamma_N \}$ forms a linear basis of $\mathcal{A}_0$ over $\Lambda$.\footnote{An analogous statement over $\mathbb{C}$ is well-known, but here, we additionally need a careful estimate on the valuation to prove this over $\Lambda$.}
\item Any $\rho \in \Lambda\langle\langle x,y,z\rangle\rangle$ with $val(\rho) \geq 0$ can be written as 
\begin{equation}\label{eqn:tjtjtj}
\rho = \sum_{i=1}^N c_i \gamma_i + \sum_{j=1}^3 t_j g_j
\end{equation}
where $c_i \in \Lambda$ with $val(c_i) \geq -8$ and $t_j \in \Lambda\langle\langle x,y,z\rangle\rangle$ with $val(t_j) \geq -8$.
\end{enumerate}
\end{theorem}

Note that the ideal $ \langle g_1,g_2,g_3\rangle $ is closed since $\Lambda\langle\langle x,y,z\rangle\rangle$ is a Tate algebra \cite[Section 5.2.7]{BGR}, as are other ideals appearing in earlier sections.

The proof of condition (2) in this theorem requires a clever usage of relations in the Jacobian ring and the argument varies for different types of $(a,b,c)$. We will provide a detailed proof in Appendix \ref{sec:thmasdefpfpf}.
It implies that the monomials in (\ref{eq:basis}) form a generating set for the $\Lambda$-vector space $\mathcal{A}_0$. Thus, in order to complete the proof of Theorem \ref{as:def}, it only remains to show that they are linearly independent.

\begin{prop}\label{prop:A0}
	The rank of $\mathcal{A}_0 $ is $a+b+c-1$.
\end{prop}
\proof

We need to show that 
	$$1,x,x^2,\ldots,x^{a-1},y,\ldots,y^{b-1},z,\ldots,z^{c-1},xyz.$$ 
are linearly independent in $\mathcal{A}_0 = \Lambda\langle\langle x,y,z\rangle\rangle / \langle g_1,g_2,g_3 \rangle $. As in \eqref{eq:basis}, we write them as $\gamma_1,\cdots,\gamma_N$.
Suppose we have the following equation in $\Lambda \langle\langle x,y,z \rangle\rangle$:
\begin{equation}\label{eq:ind1} \sum_{i=1}^N c_i \gamma_i + f_1 g_1 + f_2g_2 + f_3g_3 =0
\end{equation}
where  $c_i \in \Lambda$,  $f_j \in \Lambda\langle\langle x,y,z\rangle\rangle$.
It is enough to show that $c_1= \cdots=c_N =0$.
From the expression of $g_1,g_2,g_3$ it is easy to see that $f_ig_i$ cannot have terms like
$$1,x, \ldots, x^{a-2}, y,\ldots,y^{b-2},z,\ldots,z^{c-2}.$$
Thus we find that the coefficients on these monomials should vanish, and the equation \eqref{eq:ind1} can be written as
\begin{equation}\label{eq:ind2}
c_x x^{a-1} + c_y y^{b-1} + c_z z^{c-1} + c_N xyz +   f_1 g_1 + f_2g_2 + f_3g_3 =0.
\end{equation}
If $c_x \in \Lambda$ is non-zero,  then $f_1$ must have a nontrivial constant term $f_1^0 \in \Lambda$ in order to cancel $c_x x^{a-1}$ making use of $f_1^0 g_1 =f_1^0 (yz + T^8x^{a-1}) $. However, the monomial $f_1^0 yz$  cannot appear in other expressions
of \eqref{eq:ind2}. Thus $f_1^0=0$, and hence $c_x=0$. In the same way $c_y=c_z=0$, and the equation \eqref{eq:ind2} can be written as
\begin{equation}\label{eq:ind3}
c_N xyz +   f_1 g_1 + f_2g_2 + f_3g_3 =0.
\end{equation}
If $c_N \neq 0$, then one of $f_1,f_2,f_3$ should have a term of monomial $x,y,z$ respectively.
Suppose $f_1$ has a monomial $f_1^1x$. Then, $f_1^1 x^a$ cannot appear in other expressions of \eqref{eq:ind3} and
thus $f_1^1=0$. Similarly $f_2$ and $f_3$ cannot have monomials in $y$ and $z$ respectively, which implies $c_N=0$.
Therefore all the coefficients $c_1,\ldots, c_N$ must vanish, as desired.
%
%
%
\qed

This completes the proof of Theorem \ref{as:def}.

\subsection{\texorpdfstring{Deforming $\Jac(W_\tau)$}{Deforming Jac(W)}}\label{subsec:flatness}

In this subsection we will construct a flat family of rings interpolating between $\Jac(\mathcal{W}_{lead})$ and $\Jac(W_\tau)$.

Recall $T^8 W_\tau=  \mathcal{W}_{lead} + W_+$, for some $W_+$ with $val(W_+)=\lambda_0 > 8$. We define
$$W(s)= \mathcal{W}_{lead} + s W_+ \ \in \ \Lambda\langle\langle s,x,y,z\rangle\rangle, $$
and denote
$$ f_1=\partial_{x}W(s), \ f_2=\partial_{y}W(s), \ f_3=\partial_{z}W(s).$$

Note that $f_i=g_i + s h_i$, for some $h_i \in \Lambda\langle\langle x,y,z\rangle\rangle$ with $val(h_i)\geq\lambda_0>8$.

\begin{prop}\label{prop:genAs}
	Let $\mathcal{A}= \Lambda\langle\langle s,x,y,z\rangle\rangle/\langle f_1,f_2,f_3 \rangle$. Then $\mathcal{A}$ is a finitely generated $\Lambda\langle\langle s\rangle\rangle$-module.
\end{prop}
\proof
We will show that (\ref{eq:basis}) forms a generating set of $\mathcal{A}$. It is enough to show that any convergent series in $x,y,z$ with non-negative valuation is in the $\Lambda\langle\langle s\rangle\rangle$-span of (\ref{eq:basis}). Let $\rho$ be such series with $val(\rho) \geq 0$, by Theorem  \ref{as:def} we have 
$$\rho = \sum_{i=1}^N c_i \gamma_i + \sum_{j=1}^n t_j g_j,$$
with $c_i\in \Lambda$, $t_j\in\Lambda\langle\langle x,y,z\rangle\rangle$  and $val(c_i), val(t_j) \geq -8$. Rearranging we get

$$ \rho=  \sum_{i=1}^N c_i \gamma_i + \sum_{j=1}^n t_j (f_j -s h_j )$$
$$ = \sum_{i=1}^N c_i \gamma_i + \sum_{j=1}^n t_j f_j - T^{\lambda_0-8} s\sum_{j=1}^n t_j' h_j' $$
with $val(t_j' h_j' ) \geq 0$. By setting $\rho_1 =  \sum_{j=1}^n t_j' h_j' $, we can repeat the argument 
to prove that 
$$\rho_1 = \sum_{i=1}^N c_i^1 \gamma_i + \sum_{j=1}^n t_j^1 f_j - T^{\lambda_0-8} s \rho_{2}, $$
for some $\rho_2\in \Lambda\langle\langle x,y,z\rangle\rangle$ with non-negative valuation. Combining the two we obtain
$$\rho = \sum_{i=1}^N (c_i-T^{\lambda_0-8}s c_i^1) \gamma_i + \sum_{j=1}^n (t_j - T^{\lambda_0-8}t_j^1 s )f_j + T^{2(\lambda_0-8)} s\rho_2.$$
Note that this process increases the valuation of the error term (each time by $\lambda_0 -8 >0$), and hence using induction and taking the limit we obtain
$$\rho = \sum_{i=1}^N \tilde{c}_i \gamma_i + \sum_{j=1}^n \tilde{t}_j f_j$$
with $\tilde{c}_i\in \Lambda\langle\langle s\rangle\rangle$ and $ \tilde{t}_j \in \Lambda\langle\langle s,x,y,z\rangle\rangle$, which implies the result.
\qed

\begin{prop}
	$\mathcal{A}$ is a flat $\Lambda\langle\langle s\rangle\rangle$-module.
\end{prop}
\proof
Flatness of $\mathcal{A}$ is equivalent to flatness of the localizations $\Lambda\langle\langle s\rangle\rangle_{\mathfrak{n}}\rightarrow \mathcal{A}_{\mathfrak{m}}$ for all maximal ideals $\mathfrak{m}$ of $\mathcal{A}$, with $\mathfrak{n}=\mathfrak{m}\cap\mathcal{A}$, (see \cite[Section 3.J]{matsumura}). Note that, since $\Lambda\langle\langle s\rangle\rangle$ is a PID, $\mathfrak{n}=\langle s-s_0 \rangle$ for some $s_0 \in \Lambda$.

Now, since $\Lambda\langle\langle s\rangle\rangle_{\mathfrak{n}}$ is a regular local ring of dimension $1$, $\mathcal{A}_\mathfrak{m}$ is flat over it if and only if $s-s_0$ is not a zero-divisor in $\mathcal{A}_\mathfrak{m}$, by Lemma 10.127.2 in \cite{stacks}.

By the previous proposition $rk_\Lambda(\mathcal{A}/\mathfrak{n}) < \infty$, which implies that $\dim (\mathcal{A}/\mathfrak{n})=0$. Therefore
$$\dim \frac{\Lambda\langle\langle s,x,y,z\rangle\rangle_\mathfrak{m'}}{\langle f_1, f_2, f_3, s-s_0 \rangle}=0,$$
where $\mathfrak{m'}$ is the ideal of $\Lambda\langle\langle s,x,y,z\rangle\rangle$ corresponding to $\mathfrak{m}$. Which implies that $f_1,f_2,f_3,s-s_0$ is a system of parameters of $\Lambda\langle\langle s,x,y,z\rangle\rangle_\mathfrak{m'}$. Since this is a regular local ring, Theorem 31 in \cite{matsumura}, shows that $f_1,f_2,f_3,s-s_0$ is a regular sequence in $\Lambda\langle\langle s,x,y,z\rangle\rangle_\mathfrak{m'}$. This immediately implies that $s-s_0$ is not a zero-divisor in $\mathcal{A}_\mathfrak{m}$, which gives the desired result.
\qed

\begin{corollary}\label{cor:free}
	$\mathcal{A}$ is a free, finite dimensional $\Lambda\langle\langle s\rangle\rangle$-module.
\end{corollary}
\proof
This is an immediate consequence of the two previous propositions, since $\Lambda\langle\langle s\rangle\rangle$ is a PID.
\qed

\begin{remark}\label{rem:basisAs}
	It follows from our argument that (\ref{eq:basis}) forms a basis of $\mathcal{A}$, since this is the case for $s=0$. In addition to this, it follows from the proof of Proposition \ref{prop:genAs}, that any $\rho\in \mathcal{A}$ with non-negative valuation, can be written as	
	$$\rho = \sum_{i=1}^N c_i \gamma_i + \sum_{j=1}^3 t_j f_j$$
	with $c_i \in \Lambda$, $val(c_i) \geq -8$ and $t_j \in \Lambda\langle\langle s,x,y,z\rangle\rangle$, $val(t_j) \geq -8$.
\end{remark}

\subsection{Injectivity}

Now we are ready to prove injectivity of the Kodaira-Spencer map. 

\begin{prop}
	The Kodaira-Spencer map $\KS_\tau:\QH^*_{orb}(X,\tau) \to\Jac(W_\tau)$ is injective, and hence a ring isomorphism.
\end{prop}
\proof
We have already established that $\KS_\tau$ is a surjective ring homomorphism, so we need only to compare the ranks of the quantum cohomology and the Jacobian ring, to prove injectivity. It follows from the definition that $H^*_{orb}\left(\bP^1_{a,b,c}\right)$ has rank $a+b+c-1$.

As a consequence of Corollary \ref{cor:free}, we have
$$\dim_\Lambda \Jac(\mathcal{W}_{lead}) = \dim_\Lambda \mathcal{A}/\langle s \rangle =  \dim_\Lambda \mathcal{A}/ \langle s-1 \rangle= \dim_\Lambda \Jac(W_\tau). $$
We know, by Proposition \ref{prop:A0}, that $\Jac(\mathcal{W}_{lead})$ has rank $a+b+c-1$, therefore $\Jac(W_\tau)$ also has rank $a+b+c-1$, which implies the result.
\qed

\begin{remark}\label{rem:basis}
	In fact we have shown that (\ref{eq:basis}) forms a basis of $\Jac(W_\tau)$. Moreover, as explained in Remark \ref{rem:basisAs}, it follows that any $\rho\in \Jac(W_\tau)$ with non-negative valuation, can be written as	
	$$\rho = \sum_{i=1}^N c_i \gamma_i + t_1 \partial_x W_\tau+ t_2 \partial_y W_\tau+ t_3 \partial_z W_\tau.$$
	with $c_i \in \Lambda$, $val(c_i) \geq -8$ and $t_j \in \Lambda\langle\langle x,y,z\rangle\rangle$, $val(t_j) \geq 0$.
\end{remark}

\section{Calculations}\label{sec:applicationcal}

\subsection{Euler vector field}

Let $\chi$ be the Euler characteristic of $\mathbb{P}^1_{a,b,c}$.  We have $c_1(\mathbb{P}^1_{a,b,c}) = \chi [\pt]$.

\begin{theorem}\label{thm:KSofdiv}
	Under the Kodaira-Spencer map $\KS_\tau:\QH^*_{orb}(X,\tau) \to\Jac(W_\tau)$, 
	$$\KS_\tau\left(\chi [\pt]+\sum_k (1-\iota_k)\tau_k\bf{T}_k\right) = [W_\tau]$$
	where $\bf{T}_k$ form a basis of twisted sectors and $\tau_k$ are the corresponding coordinates.
\end{theorem}
\begin{proof}
	Recall that $\KS_\tau([\pt])$ is defined by counting discs with one interior point passing through $\pt$ and one boundary output point to the unit.  
	Since the total area of $\bP^1_{a,b,c}$ equals to $8A$, the image of $[\pt]$ equals to
	$$\frac{1}{8A}T \cdot \frac{\partial  \, W_\tau(\tilde{x},\tilde{y},\tilde{z})}{\partial T} $$
	written in terms of the geometric variables  $(\tilde{x},\tilde{y},\tilde{z})$ corresponding to the immersed generators.


	Let $T^{mA}\left(\prod_j \tau_j^{l_j} \right) \tilde{x}^{n_1}\tilde{y}^{n_2}\tilde{z}^{n_3}$ be a term in $W_\tau(\tilde{x},\tilde{y},\tilde{z})$.  Using Proposition \ref{prop:GB},
	\begin{align*}
	&\left(\chi \cdot \frac{1}{8A}\cdot T\frac{\partial}{\partial T} + \sum_j (1-\iota_j)\tau_j\frac{\partial}{\partial \tau_j}\right) \cdot T^{mA}\left(\prod_j \tau_j^{l_j}\right)\tilde{x}^{n_1}\tilde{y}^{n_2}\tilde{z}^{n_3} \\
	=&\left(\frac{m\chi}{8} + \sum_j l_j(1-\iota_j)\right) T^{mA}\left(\prod_j \tau_j^{l_j}\right)\tilde{x}^{n_1}\tilde{y}^{n_2}\tilde{z}^{n_3} \\
	=&\left(1-\frac{n_1}{a}-\frac{n_2}{b}-\frac{n_3}{c} + \frac{3\chi}{8}\cdot(n_1+n_2+n_3)   \right) T^{mA}\left(\prod_j \tau_j^{l_j}\right)\tilde{x}^{n_1}\tilde{y}^{n_2}\tilde{z}^{n_3}\\
	=&\left(\left(\frac{\chi}{8}-\frac{1}{a}\right)\frac{\partial}{\partial x}+\left(\frac{\chi}{8}-\frac{1}{b}\right)\frac{\partial}{\partial y}+\left(\frac{\chi}{8}-\frac{1}{c}\right)\frac{\partial}{\partial z}\right)\cdot T^{mA}\left(\prod_j \tau_j^{l_j}\right)\tilde{x}^{n_1}\tilde{y}^{n_2}\tilde{z}^{n_3}\\
	& +T^{mA}\left(\prod_j \tau_j^{l_j}\right)\tilde{x}^{n_1}\tilde{y}^{n_2}\tilde{z}^{n_3}
	\end{align*}
	Hence 
	\begin{align*}
	&\left(\chi \cdot \frac{1}{8A}\cdot T\frac{\partial}{\partial T} + \sum_j (1-\iota_j)\tau_j\frac{\partial}{\partial \tau_j}\right) \cdot W_\tau(\tilde{x},\tilde{y},\tilde{z}) \\
	&= W_\tau(\tilde{x},\tilde{y},\tilde{z}) + \left(\left(\frac{\chi}{8}-\frac{1}{a}\right)\frac{\partial}{\partial \tilde{x}}+\left(\frac{\chi}{8}-\frac{1}{b}\right)\frac{\partial}{\partial \tilde{y}}+\left(\frac{\chi}{8}-\frac{1}{c}\right)\frac{\partial}{\partial \tilde{z}}\right)\cdot W_\tau(\tilde{x},\tilde{y},\tilde{z}).
	\end{align*}
	Changing back to the variables $x=T^3 \tilde{x}, y=T^3 \tilde{y}, z=T^3 \tilde{z}$, the left hand side is 
	$$\KS_\tau\left(\chi [\pt]+\sum_k (1-\iota_k)\tau_k\bf{T}_k\right).$$ 
	The right hand side equals to $W_\tau(x,y,z) + \left(\left(\frac{\chi}{8}-\frac{1}{a}\right)\frac{\partial}{\partial x}+\left(\frac{\chi}{8}-\frac{1}{b}\right)\frac{\partial}{\partial y}+\left(\frac{\chi}{8}-\frac{1}{c}\right)\frac{\partial}{\partial z}\right) W_\tau(x,y,z)$ which is in the same class of $W_\tau(x,y,z)$ in the Jacobian ideal.
\end{proof}

\subsection{Versality of the potential}


The goal of this section is to prove the following statement, which describes the power series (up to a coordinate change) that can appear as the bulk deformed potential $W_\tau$.

\begin{theorem}\label{thm:versality}
	Let us denote $W_{lead}=-T^{-8}xyz+ x^a +y^b+z^c$ and consider $P \in \Lambda\langle\langle x,y,z\rangle\rangle$ with $val(P-W_{lead})>0$.
	Then there exist $\tau'\in H^*_{orb}(\mathbb{P}^1_{a,b,c}, \Lambda_0)$ and a coordinate change $(x',y',z')$ such that 
	$$P(x',y',z')=W_{\tau'}.$$
\end{theorem}

In order to prove this proposition we first need three lemmas.

\begin{lemma}[Refined surjectivity]\label{lem:refsurj}
	For any $P \in\Lambda\langle\langle x,y,z\rangle\rangle$ with $val(P)\geq0$, there is $\rho\in H^*_{orb}(\mathbb{P}^1_{a,b,c}, \Lambda_0)$  such that
	$$\KS_{\tau}(\rho)= P + t_1 \partial_{x}W_\tau + t_2 \partial_{y}W_\tau + t_3 \partial_{z}W_\tau,$$
for $t_1,t_2,t_3 \in \Lambda\langle\langle x,y,z\rangle\rangle$ with $val(t_1,t_2,t_3)\geq0$.
\end{lemma}
\begin{proof}
	Take $\lambda>0$ satisfying both Lemma \ref{lem:geq0} and $W_\tau= W_{lead} \mod T^\lambda$. We will show there is $\rho_0\in H^*_{orb}(\mathbb{P}^1_{a,b,c}, \Lambda_0)$ and $t_1,t_2,t_3 \in \Lambda\langle\langle x,y,z\rangle\rangle$ of non-negative valuation such that
	\begin{equation}\label{eq:refsur}
	\KS_{\tau}(\rho_0)- P - t_1 \partial_{x}W_\tau - t_2 \partial_{y}W_\tau - t_3 \partial_{z}W_\tau= T^\lambda Q,
	\end{equation}
	for some $Q\in \Lambda\langle\langle x,y,z\rangle\rangle$ of non-negative valuation. It is enough to consider the case $P=x^i y^j z^k$.
	
	If at least two of $(i,j,k)$ are non zero we take $\rho_0= \floor*{\frac{1}{a}}^{\bullet_\tau i}\bullet_\tau \floor*{\frac{1}{b}}^{\bullet_\tau j} \bullet_\tau \floor*{\frac{1}{c}}^{ \bullet_\tau k}$. Then $val( \KS_\tau(\rho_0))\geq 0$. In fact, it follows from Lemma \ref{lem:geq0}, that $\KS_\tau$ has non-negative valuation on the standard basis of $\QH^*_{orb}(X,\tau)$ except on the point class $\pt$, in which case it is $-8$. But only non-constant spheres contribute to $\rho_0$, hence $val(\rho_0)\geq8$ and $val( \KS_\tau(\rho_0))\geq 0$. Therefore $\KS_\tau(\rho_0)- \KS_\tau\left( \floor*{\frac{1}{a}}\right)^i \KS_\tau\left( \floor*{\frac{1}{b}}\right)^j \KS_\tau\left( \floor*{\frac{1}{a}}\right)^k $ has non-negative valuation and as $\KS_\tau$ is a ring map, it is in the Jacobian ideal. It then follows from Remark \ref{rem:basis} that there are $t_1,t_2,t_3 \in \Lambda\langle\langle x,y,z\rangle\rangle$ of non-negative valuation such that
	$$	\KS_{\tau}(\rho_0)-  \KS_\tau\left( \floor*{\frac{1}{a}}\right)^i \KS_\tau\left( \floor*{\frac{1}{b}}\right)^j \KS_\tau\left( \floor*{\frac{1}{a}}\right)^k= t_1 \partial_{x}W_\tau + t_2 \partial_{y}W_\tau +t_3 \partial_{z}W_\tau.$$
	It follows from  Lemma \ref{lem:geq0}, that $\KS_\tau\left( \floor*{\frac{1}{a}}\right)^i \KS_\tau\left( \floor*{\frac{1}{b}}\right)^j \KS_\tau\left( \floor*{\frac{1}{a}}\right)^k= x^iy^jz^k + T^\lambda Q$, which proves (\ref{eq:refsur}), in this case.
	
	If $i<a$, $j=k=0$ we take $\rho_0=\floor*{\frac{1}{a}}^{\bullet_\tau i}$. Using the fact that $\floor*{\frac{1}{a}}^{\bullet_\tau i}= \floor*{\frac{i}{a}} \mod T^8$ and the same argument as above we find $t_1,t_2,t_3 \in \Lambda\langle\langle x,y,z\rangle\rangle$ of non-negative valuation such that
	$$	\KS_{\tau}(\rho_0)-  \KS_\tau\left( \floor*{\frac{1}{a}}\right)^i = t_1 \partial_{x}W_\tau + t_2 \partial_{y}W_\tau +t_3 \partial_{z}W_\tau.$$
	Again it follows from  Lemma \ref{lem:geq0}, that $\KS_\tau\left( \floor*{\frac{1}{a}}\right)^i = x^i + T^\lambda Q$, which proves (\ref{eq:refsur}), in this case.
	
	If $i=a$, we take $\rho_0= [\pt] / a$. Recall that $y \partial_{y}W_\tau=-T^{-8}xyz + b y^b \mod T^\lambda$ and $z \partial_{z}W_\tau=-T^{-8}xyz + c z^c \mod T^\lambda$. Using Lemma \ref{lem:geq0}, we easily compute
	$$\KS_{\tau}(\rho_0) + \frac{5}{8a}x \partial_xW_\tau - \frac{3}{8a}y \partial_yW_\tau -\frac{3}{8a}z \partial_zW_\tau= x^a \mod T^\lambda,$$
	which is equivalent to (\ref{eq:refsur}). Here, we used the computation in the proof of Theorem \ref{thm:KSofdiv} which tells us that
	$$ \KS_{\tau}(\rho_0) = \frac{1}{a} \KS_{\tau} ([\pt]) \equiv \frac{3}{8} x^a + \frac{3b}{8a} y^b + \frac{3c}{8a} z^c - \frac{T^{-8}}{8a} xyz \mod T^\lambda.$$
	
	Finally, if $i=na +i_0$, $j=k=0$, with $n>0,$ $i_0<a$, we take $\rho_0= ([\pt]/a)^{\bullet_\tau n} \bullet_\tau \floor*{\frac{1}{a}}^{\bullet_\tau i_0} $. Only non-constant spheres contribute to the product defining $\rho_0$ so we can argue as in the first case. The remaining cases follow by the exact same arguments.
	
	Now that we have established Equation (\ref{eq:refsur}), we can proceed by induction, as in the proof of Proposition \ref{prop:surj}, to complete the proof.	
\end{proof}

\begin{lemma}\label{lem:G}
	Let $P \in\Lambda\langle\langle x,y,z\rangle\rangle$ with $val(P- W_{lead})\geq0$ and define
	$$G(s,x,y,z,\tau):= W_\tau(x,y,z)+ s(P(x,y,z)-W_{lead})\in \Lambda\langle\langle s, x,y,z, \tau\rangle\rangle.$$
	There exist $c_i \in \Lambda\langle\langle s,\tau\rangle\rangle$, and $t_1, t_2, t_3 \in \Lambda\langle\langle s, x,y,z, \tau\rangle\rangle$ with $val(c_i)$, $val(t_j)>0$ such that 
	$$\frac{\partial G}{\partial s}= \sum_i c_i \frac{\partial G}{\partial \tau_i} + t_1 \frac{\partial G}{\partial x}+ t_2 \frac{\partial G}{\partial y}+ t_3 \frac{\partial G}{\partial z}.$$
\end{lemma}
\begin{proof}
    We will prove the following more general statement: given $Q \in \Lambda\langle\langle x,y,z, \tau,s\rangle\rangle$ with $val(Q)\geq 0$ there are $c_i$ and $t_j$ as in the statement with valuation greater or equal than zero such that 
    \begin{equation}\label{eq:Q}
    Q= \sum_i c_i \frac{\partial G}{\partial \tau_i} + t_1 \frac{\partial G}{\partial x}+ t_2 \frac{\partial G}{\partial y}+ t_3 \frac{\partial G}{\partial z}.	
    \end{equation}
    Which easily implies the lemma. It is enough to consider the case of $Q\in \Lambda\langle\langle x,y,z, \tau\rangle\rangle$. We proceed in two steps:
    
    \noindent Step 1: At $s=0$, $\frac{\partial G}{\partial \tau_i}= \KS_\tau(e_i)$ and $\partial_x G=\partial_x W_\tau$, $\partial_y G=\partial_y W_\tau$, $\partial_z G=\partial_z W_\tau$. Applying Lemma \ref{lem:refsurj} to $Q$ we obtain $t_j$ and $\rho= \sum_i c_i e_i$ satisfying Equation (\ref{eq:Q}).
    
    \noindent Step 2: (Similar to Proposition \ref{prop:genAs}) By Step 1 ,there are $c_i^{(0)}$ and $t_j^{(0)}$ satisfying Equation (\ref{eq:Q}). By definition, $\partial_x G=\partial_x W_\tau + s T^{\alpha}f_1$ for some $\alpha>0$ and $f_1$ with valuation greater or equal than zero. Similarly for $y$ and $z$. Hence 
    $$Q= \sum_i c_i^{(0)} \frac{\partial G}{\partial \tau_i} + t_1^{(0)} \frac{\partial G}{\partial x}+ t_2^{(0)} \frac{\partial G}{\partial y}+ t_3^{(0)} \frac{\partial G}{\partial z} - s T^{\alpha} \sum_j t_j^{(0)} f_j.$$
    Next, we apply Step 1 to $Q^{(1)}:=\sum_j t_j^{(0)} f_j$ gives $c_i^{(1)}$ and $t_j^{(1)}$ which allows us to rewrite $Q$ as
    $$ \sum_i (c_i^{(0)} -sT^{\alpha} c_i^{(1)} )\frac{\partial G}{\partial \tau_i} + (t_1^{(0)} -sT^{\alpha} t_1^{(1)} ) \frac{\partial G}{\partial x}+ (t_2^{(0)} -sT^{\alpha} t_2^{(1)} ) \frac{\partial G}{\partial y}+ (t_3^{(0)} -sT^{\alpha} t_3^{(1)} ) \frac{\partial G}{\partial z} + s^2 T^{2\alpha} Q^{(2)}$$
    By induction, we construct $c_i^{(n)}$ and $t_j^{(n)}$ and define $c_i :=\sum_n c_i^{(n)} s^n T^{n \alpha}$ and $t_j:= \sum_n t_j^{(n)}s^n T^{n \alpha}$. From the construction, it is clear these satisfy Equation (\ref{eq:Q}).
    \end{proof}

The next lemma is a general result about the existence of coordinate changes by integrating a vector field in our non-archimedean setting. It should be well known to experts.  We include a proof for completeness. We will use the short-hand notation $\Lambda\langle\langle x, \tau \rangle\rangle:= \Lambda\langle\langle x_1,\ldots ,x_n, \tau_1, \ldots \tau_m\rangle\rangle$.

\begin{lemma}\label{lem:flow}
	Consider $A_j \in\Lambda\langle\langle s, x, \tau\rangle\rangle$ and $B_i \in\Lambda\langle\langle s, \tau\rangle\rangle$ with valuations $\geq\epsilon>0$ and let $X$ be the vector field
	$$X:=\sum_j A_j \frac{\partial}{\partial x_j} + \sum_i B_i \frac{\partial}{\partial \tau_i}.$$
	Then there exists a coordinate change $\Phi(s,x,\tau)=(s, \psi(s,x,\tau), \varphi(s,\tau))$, with $\Phi(0,x,\tau)=(0,x,\tau)$ and 
	$$\frac{d \Phi}{d s}(s,x,\tau)= X(\Phi(s,x,\tau)).$$ 
	\end{lemma}
\begin{proof}
    	Simplifying the notation, we have to show that there is $(\psi_s(x,\tau),\varphi_s(\tau))$ such that $\frac{d}{d s}((\psi_s(x,\tau),\varphi_s(\tau)))=(A(s, \psi_s(x,\tau),\varphi_s(\tau)), B(s, \varphi_s(\tau)))$ and $(\psi_0(x,\tau),\varphi_0(\tau))=(x,\tau)$. This is equivalent to
    	\begin{equation}\label{eq:integral}
    	(\psi_s(x,\tau),\varphi_s(\tau)) - (x,\tau) = \left(\int_0^{s} A(u,\psi_u(x,\tau),\varphi_u(\tau))\,du , \int_0^{s} B(u,\varphi_u(\tau))\,du )\right). 
    	\end{equation}
    	We define a sequence $\Phi^k:=(\psi^k, \varphi^k)$ inductively as $(\psi^0, \varphi^0)=(x,\tau)$ and $$(\psi^{k+1}_s(x,\tau), \varphi^{k+1}_s(\tau))=\left(\int_0^{s} A(u,\psi^k_u(x,\tau),\varphi^k_u(\tau))\,du , \int_0^{s} B(u,\varphi^k_u(\tau))\,du )\right).$$
    	
    	By assumption there is $\epsilon>0$ such that $val(A),val(B)\geq \epsilon$. We claim that $val(\Phi^k-\Phi^{k-1})\geq k \epsilon$. We prove it by induction on $k$. First note that 
    	$$val(\Phi^k-\Phi^{k-1})=val( \int_0^{s} F(u,\Phi^{k-1})- F(u,\Phi^{k-2})\,du),$$
      where $F=(A,B)$. Then by induction $\Phi^{k-1}=\Phi^{k-2}+T^{(k-1)\epsilon}\rho$ for some $\rho$ of non-negative valuation. Hence $F(u,\Phi^{k-1})-F(u,\Phi^{k-2})=T^\epsilon T^{(k-1)\epsilon}\tilde{\rho}$, for some $\tilde{\rho}$ of non-negative valuation, which conclude the induction step.
      
      This argument shows that $(\psi, \varphi):= (\psi^0, \varphi^0)+ \sum_{k\geq 1} (\psi^k, \varphi^k)-(\psi^{k-1}, \varphi^{k-1})$ converges. By construction, it is a coordinate change. Obviously $(\psi, \varphi)=\lim_k (\psi^k, \varphi^k)$ and therefore it solves (\ref{eq:integral}).
\end{proof}

\begin{proof}[Proof of Theorem \ref{thm:versality}]
Given $P$, define $G$ as in Lemma \ref{lem:G} and let $t_j$, $c_i$ be the series provided by that lemma. Let $X$ be the vector field
	$$X:=\sum_j -t_j \frac{\partial}{\partial x_j} + \sum_i -c_i \frac{\partial}{\partial \tau_i},$$
and let $\Phi(s,x,\tau)$ be the coordinate change provided by Lemma \ref{lem:flow}. 
By construction $\displaystyle X\cdot G = - \frac{\partial G}{\partial s}$ which implies $\displaystyle \frac{d}{ds}\left(G(\Phi(s,x,y,z,\tau))\right)=0$. Hence $G(0,x,y,z,\tau)=G(\Phi(1,x,y,z,\tau))$. Using the notation $\Phi(s,x,y,z,\tau)=(s, \psi_s(x,y,z,\tau), \varphi_s(\tau))$, this is equivalent to 
$$G(0, \psi_1^{-1}(x, y, z, \tau), \varphi_1^{-1}(\tau))=G(1,x,\tau).$$

Evaluating at $\tau=0$ we obtain $W_{\varphi^{-1}(0)}(\psi^{-1}(x, y, z, 0)) = W_{lead}(x,y,z) + P(x,y,z) - W_{lead}(x,y,z)$. Denoting $\varphi^{-1}(0)=\tau'$ and $(x',y',z')=\psi(x,y,z,0)$ we get the desired equality
$W_{\tau'}(x,y,z)=P(x',y',z').$
\end{proof}

\subsection{Valuation of critical points}\label{subsec:valcritpt1}
Recall that instead of working with geometric variables for the immersed sectors $(\tilde{x},\tilde{y},\tilde{z})$,
we  switched to new variables  $(x,y,z)$ which were defined by $x=T^{3} \tilde{x},y=T^{3} \tilde{y},z=T^{3} \tilde{z}$.
The Jacobian ring for $W_\tau(x,y,z)$ and $W_\tau(\tilde{x},\tilde{y},\tilde{z})$ are a priori different.  We show that indeed, there is an example that the Kodaira-Spencer map is {\em not} an isomorphism
if we consider the map to the Jacobian ring of $W_\tau(\tilde{x},\tilde{y},\tilde{z})$.

This can happen because of the following. Before bulk deformation for the critical points of $W_\tau(\tilde{x},\tilde{y},\tilde{z})$, we have
$val(\tilde{x}), val(\tilde{y}), val(\tilde{z}) \geq 0$.  But after bulk deformations by twisted sectors $\tau=\tau_{tw}$ with small valuations, critical points of $W_\tau (\tilde{x},\tilde{y},\tilde{z})$ may begin to have negative valuations, hence moving away from the disc of Novikov convergence $\Lambda^3_{+,(\tilde{x},\tilde{y},\tilde{z})}$. We illustrate such a phenomenon in the example $(a,b,c)=(2,2,r)$. In fact,  in this example all the critical points will escape.

\begin{prop}\label{prop:critescape22r}
Consider the case $X=\mathbb{P}^1_{2,2,r}$. There exists $\tau \in H^*_{orb}(X,\Lambda_+)$ such that any critical point of the bulk-deformed potential $W_\tau (x,y,z)$ has at least one coordinate whose valuation is smaller than $3$.
\end{prop}

Proposition \ref{prop:critescape22r} implies that every critical point of $W_\tau (\tilde{x},\tilde{y},\tilde{z})$ have at least one coordinate with a negative valuation, and hence the conventional boundary deformation of $\mathbb{L}$ would not capture these points.

\begin{proof}
Let us take $\tau= T^\lambda \floor*{\frac{1}{2}}_a + T^{\lambda} \floor*{\frac{1}{2}}_b$ supported on the two orbi-points with $\Z/2$-singularity. We further assume that $\lambda < \min \left\{3, \frac{3r-5}{2} \right\}$.
The associated bulk-deformation only adds two terms $T^{\lambda} x + T^{\lambda} y $ to the non-bulk-deformed potential. To see this, recall that if the term $x^{n_1} y^{n_2} z^{n_3}$ appears in the potential, the contributing polygon satisfies the inequality
$$ \frac{n_1}{2}+\frac{n_2}{2}+\frac{n_3}{r}+ \sum_j (1-\iota_j) \leq 1.$$
by the area formula together with Lemma \ref{lem:geq0}. In our case $\iota_j=\frac{1}{2}$, and hence the polygon can have at most one orbi-insertion. Any such polygon lifts to a holomorphic disc in the universal cover of $\mathbb{P}^1_{2,2,r}$ by Corollary \ref{coro:conelifts}, and the lift must be invariant under the group action corresponding to either $\floor*{\frac{1}{2}}_a$ or $\floor*{\frac{1}{2}}_b$. It is easy to see that the 2-gons responsible for $x^2$ and $y^2$ are only such (see diagrams in \cite[Section 12]{CHKL14}), and their halves are orbi-discs producing $T^{\lambda} x + T^{\lambda} y $ in $W_\tau$.

More concretely,{
	and arguing as in the proof of Proposition \ref{coro:w-8}, we conclude} the resulting bulk-deformed potential is given as
 $$W_\tau (x,y,z)=-T^{-8} xyz + x^2 + y^2 + z^r+ f(z) + T^{\lambda} x + T^{\lambda} y $$
where $f(z)$ is a polynomial in $z$ of the form
$$ f(z) = c_1 T^{16} z^{r-2} + c_2 T^{32} z^{r-4} + \cdots  = \sum_{k=1}^{\lfloor \frac{r}{2} \rfloor} c_k T^{16k} z^{r-2k}$$
for some combinatorially defined integers $c_k$. The precise value of $c_k$, which can be found in \cite[Theorem 12.2]{CHKL14}, is not important to us, but we will use the fact that $val(c_k)=0$ in the argument below.

Given the formula, critical points of $W_\tau$ satisfy
\begin{equation}\label{eqn:crit22rval}
\begin{array}{l}
 -T^{-8} yz + 2x + T^\lambda =0 \\
 -T^{-8} xz + 2y + T^\lambda = 0 \\
 -T^{-8} xy + rz^{r-1} + f'(z)=0.
 \end{array}
\end{equation}
Subtracting the second equation from the first gives
$$ (x-y) (T^{-8} z +2 ) =0,$$
and hence, either $z = -2 T^{8}$ or $x=y$.
\begin{enumerate}
\item[(i)] If $z= -2T^8$, then $2(x+y) = - T^{\lambda}$ and $T^{-8} xy= C T^{r-1}$ for some constant $C$ with $val(C)=0$. Thus $x$ and $y$ are solutions of the quadratic equation (in $s$)
$$ s^2 + \frac{1}{2} T^{\lambda} s + C T^{8 + (r-1)}=0,$$
which are
$$ -\frac{1}{4}T^{\lambda} \pm \frac{1}{4} T^{\lambda} \sqrt{1- 2 C T^{8+ (r-1) - 2\lambda}}. $$
In particular, one of $x$ and $y$ always has valuation $\lambda$, which is smaller than $3$.
\item[(ii)] Consider the case of $x=y$. The second equation in \eqref{eqn:crit22rval} reads
$$x (-T^{-8} z + 2) = -T^{\lambda}.$$
If $val(z) \geq 8$, then $val(x) = \lambda <3$, we are done. 

Now suppose $val(z) < 8$, which implies $ val(x) + val(z) = \lambda +8$. Therefore
\begin{equation}\label{eqn:valx+zlambda8}
val(x) = \lambda+8 - val(z)
\end{equation}
On the other hand, the third equation of \eqref{eqn:crit22rval} gives
\begin{equation}\label{eqn:xyrzf'z}
T^{-8} xy = r z^{r-1} + f'(z) = rz^{r-1} + \sum_{k=1}^{\lfloor \frac{r-1}{2} \rfloor} (r-2k) c_k T^{16k} z^{r-2k-1},
\end{equation}
and the valuation of monomials in $f'(z)$ can be estimated as
\begin{equation*}
\begin{array}{lcl}
val(T^{16k} z^{r-2k-1}) &=& 16k + (r-2k-1) val(z)  \\
&=& (r-1) val(z) + 2k (8 - val(z)) > (r-1) val(z).
\end{array}
\end{equation*}
Therefore the right hand side of \eqref{eqn:xyrzf'z} has valuation $(r-1) val(z)$, and we obtain
\begin{equation}\label{eqn:valxvalzr-1}
 -8 + 2 val(x) = (r-1) val(z)
\end{equation}
Combining \eqref{eqn:valx+zlambda8} and \eqref{eqn:valxvalzr-1}, we conclude that $ val(z) = \frac{ 2 \lambda +8}{r+1}$, which is again smaller than $3$.
\end{enumerate}
Therefore, at least one coordinate of any critical point $(x,y,z)$ of $W_\tau$ has valuation smaller than $3$, and this proves the claim. 
\end{proof}

In $(x,y,z)$ coordinates, these critical points still have $val(x),val(y),val(z)\geq 0$ for $val(\tau)\geq 0$, so it does not violate the isomorphism $\KS_\tau : \QH^*_{orb}(X,\tau) \cong \Jac(W_\tau (x,y,z))$. 

\begin{remark}
Although geometric variables  $\tilde{x},\tilde{y},\tilde{z}$ with negative valuations are not usually considered in Floer theory, we speculate that such a deformation can be replaced by another Lagrangian using the gluing procedure explained in \cite{CHL-glue}. 
\end{remark}

\subsection{\texorpdfstring{Explicit computation of $\KS_\tau$ for $\mathbb{P}^1_{3,3,3}$ without bulk-parameters}{Explicit computation of KS for P1,3,3,3 without bulk-parameters}}

In this section, we give an explicit computation of the Kodaira-Spencer map for $ \bP^1_{3,3,3}$ without bulk-insertions. Namely, we set $\tau=0$ throughout the section. For notational simplicity, let us write $X$ for $ \bP^1_{3,3,3}$ from now on.

We use the following notations for generators of $\QH^*_{orb}(X):=\QH^*_{orb}(X,0)$.
We set $\one_X$ to be the unit class, and denote twisted sectors
by $\Delta^{1/3}_i$ and $\Delta^{2/3}_i$ for $i=1,2,3$ where $i$ indicates orbifold points. We introduce these new notations to avoid the potential confusion due to the coincidence $a=b=c=3$ in this case. In the previous terminology, they all happen to be $\left\lfloor \frac{i}{3} \right\rfloor$. $\Delta^{1/3}_i$ has degree shifting number $\iota(\Delta^{1/3}_i) = 1/3$, and $\Delta^{2/3}_i$ has degree shifting number $\iota(\Delta^{2/3}_i) = 2/3$.

For the point class, we take 8-generic points $\pt_1, \cdots, \pt_8$ on $\mathbb{P}^1_{3,3,3}$ which are not orbifold singular points, as shown in Figure \ref{fig:QHJAC_pt}. Then we define the point class $\pt$ to be the average of these 8 points. i.e. $\pt := \frac{1}{8} \sum_i \pt_i$.
This is to make the calculation of $\KS ([\pt])$ easier. Notice that such a choice makes the number of $[\pt]$-insertions proportional to the symplectic area of the contributing polygons.
Then $\one_X, [\pt], \Delta^{1/3}_i,\Delta^{2/3}_i$ for $i=1,2,3$ form a basis of  $QH_{orb}^*(X)$.  

\begin{figure}
\begin{center}
\includegraphics[height=2.3in]{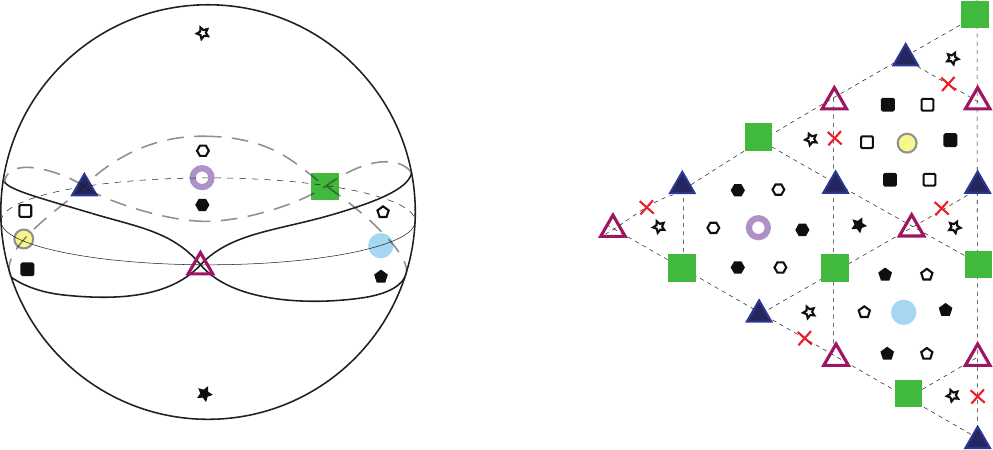}
\caption{$\mathbb{P}^1_{3,3,3}$ and its cover}\label{fig:QHJAC_pt}
\end{center}
\end{figure}

On the other hand, in \cite{CHL17}, a Morse model was adopted for $CF(\mathbb{L},\mathbb{L})$ together with the combinatorial sign rule in \cite{Se}, which produces the explicit formula for the potential $W $ given as
\begin{equation}\label{eqn:w333explicit}
W=  \tilde{\phi}(T) (\tilde{x}^3 + \tilde{y}^3 + \tilde{z}^3) - \tilde{\psi}(T) \tilde{x}\tilde{y} \tilde{z}
\end{equation}
where $T$ is the (exponentiated) area of the minimal triangle as before and 
\begin{equation}\label{eqn:phipsi333}
\begin{array}{l}
\displaystyle
 \tilde{\phi} (T) = \sum_{k=0}^{\infty} (-1)^{k} (2k+1) T^{(6k+3)^2} \\
\displaystyle
 \tilde{\psi} (T) = T + \sum_{k=1}^{\infty} (-1)^{k} \left((6k+1) T^{(6k+1)^2} - (6k-1) T^{(6k-1)^2} \right)
\end{array}
\end{equation}
See \cite{CHL17} for detailed computation\footnote{To be more precise, \eqref{eqn:w333explicit} is obtained by changing $\tilde{x}$ and $\tilde{z}$ to $-\tilde{x}$ and $-\tilde{z}$ in the formula therein. $\tilde{x}$, $\tilde{y}$, $\tilde{z}$ did not appear to be symmetric  in \cite{CHL17}, due to some asymmetric choice of the input data for the combinatorial sign rule.}.
We then make the change of coordinate \eqref{eqn:chvar} to obtain
$$W (x,y,z) =   \phi (T) (x^3+ y^3 + z^3) - \psi (T) xyz$$
with
\begin{equation*}
\begin{array}{lcl}
\phi(T) &=&\sum_{k=0}^{\infty} (-1)^k (2k+1) T^{36k^2 + 36k} \\ 
\psi(T) 
&=&T^{-8} \left(1 + \sum_{k=1}^{\infty} (-1)^{k} \left((6k+1) T^{36k^2 + 12k} - (6k-1) T^{36k^2 -12k} \right) \right).
\end{array}
\end{equation*}

Recall that the map $\KS: QH^*_{orb} (X) \to \Jac (W)$ is defined by sending a cycle $C$ to the polynomial class in $\Jac (W)$ represented by
\begin{equation}\label{QH->Jac}
\sum_{\beta} \sum_{k\geq 0} q^\beta n_{1;k}(\beta;C;b,\ldots,b)
\end{equation}
where $b = xX + yY + zZ (=T^3 \tilde{x} X + T^3 \tilde{y} Y + T^3 \tilde{z} Z)$.  \eqref{QH->Jac} is a series in $x,y,z$ in general, but we will see from explicit calculations that it is just a polynomial (over the Novikov field) for $C = \one_X, \pt, \Delta^{1/3}_i,\Delta^{2/3}_i$, $i=1,2,3$.
By dimension counting \eqref{subsec:generalorbmod} restricted to this case, $n_{1;k}(\beta; \Delta^{1/3}_i;b,\ldots,b)$ (resp. $n_{1;k}(\beta;\Delta^{2/3}_i;b,\ldots,b)$) is non-zero only when
$\mu_{\rm CW} = 2/3$ (resp. $\mu_{\rm CW} = 4/3$).
%

On the other hand, $X$ has a manifold cover $\tilde{X}$, (which is the unique elliptic curve $E$ that admits $\Z/3$-action), and the preimage $\WT{\BL}$ of $\BL$ in $\tilde{X}$ is a $\Z /3$-orbit of a circle in $\WT{\BL}$, which are represented as dotted lines along three different directions in Figure \ref{fig:QHJAC_pt}. Recall that any orbi-disc in our interest should lift to $\tilde{X}$ by Corollary \ref{coro:conelifts}.
We will need the following classification of orbi-discs for explicit calculations:
 
\begin{prop}\label{prop:orbi-polyclass}
Let $u$ be a holomorphic orbi-disc in $X$ bounded by $\BL$ of Maslov index $2/3$ with one interior orbi-marked point passing through the twisted sector $\Delta^{1/3}_1$ (or $\Delta^{1/3}_2$, $\Delta^{1/3}_3$), and suppose that $u$ only passes through the immersed sectors $X,Y,Z$ (but not $\bar{X},\bar{Y},\bar{Z}$).  Then, it can be lifted to a holomorphic disc in $\tilde{X}$ bounded by $L$ of Maslov index two whose boundary passes through the (preimage of) immersed sector $X$ (or $Y$, $Z$ respectively) three times.

If $u$ is a holomorphic orbi-disc in $X$ bounded by $\BL$ of Maslov index $4/3$ with one interior orbi-marked point passing through the twisted sector $\Delta^{2/3}_1$ (or $\Delta^{2/3}_2$, $\Delta^{2/3}_3$), it can be lifted to a holomorphic disc in $\tilde{X}$ bounded by $\WT{\BL}$ of Maslov index four whose boundary either passes through $X$ (or $Y$, $Z$ respectively) six times or passes through both $(Y,Z)$ for three times (or $(X,Z)$, $(X,Y)$ respectively).
\end{prop}

\begin{proof}
Let $u$ be a holomorphic orbi-disc in $X$ bounded by $\BL$ of Maslov index $2/3$ with one interior orbi-marked point passing through the twisted sector $\Delta^{1/3}_1$.  The orbi-marked point in the domain of $u$ must have isotropy group $\Z / 3$ because of the injectivity between local groups.  Thus $u$ can be lifted to a $\Z/3$-equivariant holomorphic map $\tilde{u}:(\Delta,\partial \Delta) \to (\tilde{X},\WT{\BL})$ with the domain of $\tilde{u}$ covers the domain of $u$ by the map $\zeta = \tilde{\zeta}^3$.  Since $u$ has Maslov index $2/3$, $\tilde{u}$ has Maslov index two.  Moreover $\tilde{u}$ only passes through the immersed sectors $X,Y,Z$.  Each of these immersed sectors contribute $2/3$ to the Maslov index.  Thus $\tilde{u}$ can only pass through three of them.  By the $\Z/3$-equivariance these three immersed sectors are the same.  This forces the immersed sectors that $\tilde{u}$ passes through to be all $X$.

The proof for the second statement is similar.  Consider the uniformizing disc $\tilde{u}$ of the orbi-disc of Maslov index $4/3$. $\tilde{u}$ can only pass through six immersed sectors by the constraint on Maslov index.  Unless $\tilde{u}$ is constant, $\Z/3$ acts non-trivially on $\tilde{u}$, and hence the six corners of $\tilde{u}$ can only pass through at most two distinct immersed sectors.  Thus the immersed sectors $\tilde{u}$ passes through are either all $X$ (or all $Y$ or all $Z$) or three copies of $Y$ and $Z$ (or three copies of $X$ and $Y$, or three copies of $Z$ and $X$).
(See Figure \ref{fig:twothird} for the shape of these orbi-discs.)

To see this, we first fix the twisted sector   $\Delta^{2/3}_1$, and choose one of its pre-images in $\C$, say $p_a$.
Pick any point $\WT{Y} \in \C$ corresponding to the immersed sector $Y$. Then, if $\tilde{u}$
pass through $p_a$ (at the orbifold point) and has an immersed boundary at  $\WT{Y}$, then we know that  $\tilde{u}$ also passes through $\Z/3$-rotation images
of $\WT{Y}$ with respect to $p_a$, say  $\WT{P}, \WT{Q}$. These two points still corresponds to immersed sector $Y$. 
Then,  remaining immersed sectors should connect two points out of $\WT{Y}, \WT{P}, \WT{Q}$ along the lifts of the Lagrangian.
It is not hard to check that such immersed sector correspond to $Z$.  The remaining cases are similar and we omit the details.
\end{proof}

By calculating contribution from each orbi-disc in the above classification, we can explicitly calculate the map $\KS : QH^*_{orb} (X) \to \Jac(W)$ as follows.

\begin{prop} \label{QH->Jac explicit}
The map $\KS$ from the orbifold quantum cohomology of $X$ to the Jacobian ring of $W$ defined in \eqref{QH->Jac} is given by
$$\one_X \mapsto  1 , \qquad [\pt] \mapsto  \frac{1}{8} T \dd{T} W(\tilde{x},\tilde{y},\tilde{z}),$$
\begin{equation*}
\left\{
\begin{array}{l}
\Delta^{1/3}_1 \mapsto P(T) x \\
\Delta^{1/3}_2 \mapsto P(T) y \\
\Delta^{1/3}_1 \mapsto P(T) z
\end{array}\right. , \qquad \qquad
\left\{
\begin{array}{l}
\Delta^{2/3}_1 \mapsto Q(T) x^2 + R(T) yz \\
\Delta^{2/3}_2 \mapsto  Q(T) y^2 + R(T) zx \\ 
\Delta^{1/3}_3 \mapsto  Q(T) z^2 + R(T) xy
\end{array}\right. .
\end{equation*}
where 
\begin{align*}
 P(T) &\!=\! \sum_{k=0}^\infty (-1)^{k} (2k+1) T^{12k^2 + 12k}, \\ 
 Q(T) &\!=\! \sum_{k=0}^\infty (2k+1) T^{24k^2+24k}+ \sum_{k=1}^\infty \sum_{i=0}^{k-1} (-1)^{3k-i} (6k-2i+2) T^{36k^2 + 36k - 12i^2 - 12i}, \\
 R(T) &\!=\! \sum_{k=1}^\infty \sum_{i=0}^{k-1} (-1)^{3k-i} T^{-8} \left(  (6k-2i) T^{36k^2 + 12k - 12i^2 - 12i  } -(6k-2i-2) T^{36k^2 - 12k - 12i^2 - 12i  }  \right) . 
\end{align*}
\end{prop}
The proof will be given in Appendix \ref{sec:pf333orb}.

%
\begin{remark}
 Satake--Takahashi \cite{ST} provided an explicit description of the genus zero 
 Gromov--Witten potential of $\mathbb{P}^1_{3,3,3}$, which, in particular, determines  the structure constants for the product structure on $QH^\ast_{orb} (\mathbb{P}^1_{3,3,3})$. 
For instance, 
one of the interesting quantum products is given by
 $$\Delta_1^{1/3} \bullet_0 \Delta_1^{1/3} = f_1(q)\Delta_1^{1/3},$$
 where $f_1(q)$ given in \cite{ST} is an expression involving Dedekind eta function. This gives a complicated looking identity on $\Jac(W)$-side through our explicit map, which is a priori highly nontrivial. 
\end{remark}

\appendix

\section{Proof of Proposition \ref{prop:divisor}}\label{sec:divisor}

Proposition \ref{prop:divisor} is equivalent to the following equalities in the Jacobian ring
from the Definition \ref{def:KS} of Kodaira-Spencer map and the potential. 
\begin{equation}\label{eq:unital}
 \sum_{\beta, k, l} \exp( \tau^2 \mathbf{p} \cap \beta)\frac{T^{\beta\cap \omega}}{l!}\mathfrak{q}_{l+1,k,\beta}(\one_{X}, \tau_{tw}^l; b^k)= \one_\BL 
 \end{equation}
 	\begin{equation}\label{eq:divisor}
 \sum_{\beta, k, l} \exp( \tau^2 \mathbf{p} \cap \beta) \frac{T^{\beta\cap \omega}}{l!}\mathfrak{q}_{l+1,k,\beta}(\mathbf{Q}, \tau_{tw}^l; b^k)= \sum_{\beta, k, l} \exp( \tau^2 \mathbf{p} \cap \beta) \frac{T^{\beta\cap \omega}}{l!}( \mathbf{Q} \cap \beta)\mathfrak{q}_{l,k,\beta}(\tau_{tw}^l; b^k)
\end{equation}

The proof of both statement is similar, we will start with the second one.
The main technical issue is that the Kuranishi structure which is used to define bulk deformation as well as Kodaira-Spencer map may not be compatible with the operation of forgetting interior marked points - see Remark 3.1 in \cite{F10} for a discussion of this point.  To overcome this problem,  we will construct Kuranishi structures and CF-perturbations on the spaces  $\mathcal{M}^{para}_{k+1,l+1}(\beta, \mathbf{Q}, \tau_{tw},\gamma):=\mathcal{M}_{k+1,l+1}(\beta, \mathbf{Q}, \tau_{tw},\gamma)\times [0,1]$. First note these have the following boundary decomposition
\begin{align}\label{eq:boundpara}
\partial \mathcal{M}^{para}_{k+1,l+1}(\beta, \mathbf{Q}, \tau_{tw}^l,\gamma)\times [0,1] & =  \mathcal{M}_{k+1,l+1}(\beta, \mathbf{Q}, \tau_{tw}^l,\gamma)\times \{0\} \bigcup \mathcal{M}_{k+1,l+1}(\beta, \mathbf{Q}, \tau_{tw},\gamma)\times \{1\}\nonumber \\
      \bigcup_{\substack{ \beta_1+\beta_2=\beta \\ k_1+k_2=k, 1\leq i\leq k_2+1}}  &  \mathcal{M}^{para}_{k_1+1, l_1+1}(\beta_1;\mathbf{Q}, \tau_{tw}^{l_1},\gamma)\;_{ev_0}\times_{ev_i} \mathcal{M}_{k_2+2, l_2}(\beta_2, \tau_{tw}^{l_2}, \gamma) \\ 
      \bigcup_{\substack{ \beta_1+\beta_2=\beta \\ k_1+k_2=k, 1\leq i\leq k_2+1}} &  \mathcal{M}_{k_1+1, l_1}(\beta_1, \tau_{tw}^{l_1},\gamma) \;_{ev_0}\times_{ev_i} \mathcal{M}^{para}_{k_2+2, l_2+1}(\beta_1;\mathbf{Q}, \tau_{tw}^{l_2},\gamma)\nonumber  
\end{align}  

On the factors with no $\mathbf{Q}$ insertion in the second and third line in Equation \ref{eq:boundpara} the Kuranishi structure and CF-perturbations coincide with the ones used to define the $\m^\tau_k$ operations. On the component $\mathcal{M}_{k+1,l+1}(\beta, \mathbf{Q}, \tau_{tw},\gamma)\times \{0\}$ the Kuranishi structure and CF-perturbations coincide with ones used to define the maps $\mathfrak{q}_{l+1,k,\beta}$, which are used in Section \ref{subsec:ringhomkura} and crucially respect the decomposition in (\ref{eqglue}). Finally, on the component $\mathcal{M}_{k+1,l+1}(\beta, \mathbf{Q}, \tau_{tw}^l,\gamma)\times \{1\}$ we require \emph{compatibility} with the map that forgets the interior marked point where we insert the divisor $\mathbf{Q}$
$$\pi: \mathcal{M}_{k+1,l+1}(\beta, \mathbf{Q}, \tau_{tw},\gamma) \rightarrow \mathcal{M}_{k+1,l}(\beta, \tau_{tw}^l,\gamma).$$
The notion of compatibility we use here is a variation of the notions considered in \cite[Def. 3.1]{Amo} and \cite[Sec 7.3.2]{FOOO}.

\begin{defn}
We say the Kuranishi structures are \emph{compatible} with respect to $\pi$ if for every $u\in \mathcal{M}_{k+1,l+1}(\beta, \mathbf{Q}, \tau_{tw}^l,\gamma)$ and $v=\pi(u)$, there is a map between the Kuranishi neighborhoods $(V_u,E_u,\Gamma_u,s_u,\psi_u)$ and $(V_v,E_v,\Gamma_v,s_v,\psi_v)$ satisfying the following
\begin{itemize}
	\setlength{\parsep}{0pt}
	\setlength{\itemsep}{0pt}
	\item $h_{uv}:\Gamma_u \rightarrow\Gamma_v$ is an injective homomorphism;
	\item $V_u\cong V_v \times B$ where $B$ is a ball in $\mathbb{C}$ and $\varphi_{uv}:V_u \rightarrow V_v$ is $h_{uv}$-equivariant continuous map, strata-wise smooth;
	\item an isomorphism $E_u\simeq \varphi_{uv}^*E_v \oplus N$ where $N$ is a rank two bundle;
	\item the $\varphi_{uv}^*E_v$ component of $s_u$ equals $\varphi_{uv}^*s_v$;
	\item $\varphi\circ\psi_u=\psi_v\circ\varphi_{uv}$ on $s^{-1}_u(0)/\Gamma_u$.
\end{itemize}
\end{defn} 
\begin{lemma}
	There are Kuranishi structures on the moduli spaces $\mathcal{M}^{para}_{k+1,l+1}(\beta, \mathbf{Q}, \tau_{tw}^l,\gamma)$ which respect the boundary decomposition (\ref{eq:boundpara}) and have the compatibilities described above.
\end{lemma}
\begin{proof}
	With the exception of the compatibility with the forgetful map $\pi$, the construction of such Kuranishi structure is standard by now, see \cite{F10} for example. To ensure compatibility with $\pi$ we proceed as follows. Given the Kuranishi neighborhood $(V_v,E_v,\Gamma_v,s_v,\psi_v)$ we take $V_u\cong V_v \times B$ where $B$ parameterizes the position of the additional marked point $z_1^+$. Then the map $\varphi_{uv}$ is locally modeled on a forgetful map $\Pi: \mathcal{M}_{k+1, l'+1} \to \mathcal{M}_{k+1, l'}$, see the proof of Proposition 4.2 in \cite{Amo} for an analogous argument. Then taking the obstruction bundle $\varphi_{uv}^*E_v $ would give a Kuranishi neighborhood in $\mathcal{M}_{k+1,l+1}(\beta, \tau_{tw}^l,\gamma)$, that is, without incidence relation with $\mathbf{Q}$. We include this restriction by identifying a neighborhood of $ev_{z_1^+}(u)$ (which includes no other component of $\mathbf{Q}$) in $\Pabc$ with a ball in $\mathbb{R}^2\cong N$. Then the $N$ component of  the obstruction map $s_u(x)$ is $ev_{z_1^+}(x) - ev_{z_1^+}(u)$. It is not hard to see this satisfies all the properties.
\end{proof}

\begin{remark}
	As explained in Appendix A.1.4 in \cite{FOOO}, when constructing Kuranishi structures on moduli spaces of discs one has to take a special smooth structure near nodal discs. Due to this choice, forgetful maps are continuous but smooth only when we restrict to a stratum of the stratification according to combinatorial type of the underlying disc. This is the reason $\varphi_{uv}$ is only strata-wise smooth.
\end{remark}

Now consider  CF-perturbations $(\mathcal{ U}_\alpha, W_\alpha, S_\alpha)_{\alpha\in I}$ on $\mathcal{M}^{para}_{k+1,l+1}(\beta, \mathbf{Q}, \tau_{tw}^l,\gamma)$. We say it is compatible with $\pi$ if its restriction to $\mathcal{M}_{k+1,l+1}(\beta, \mathbf{Q}, \tau_{tw}^l,\gamma)\times \{1\}$ is the {\it pull-back} of CF-perturbations $(\mathcal{ V}_\beta, W_\beta, S_\beta)_{\beta\in J}$ on $\mathcal{M}_{k+1,l}(\beta, \tau_{tw}^l,\gamma)$. By pull-back we mean there are maps of Kuranishi neighborhoods from $\mathcal{ U}_\alpha$ to $\mathcal{ V}_{\beta(\alpha)}$, $W_\alpha=W_{\beta(\alpha)}$, $\theta_\alpha=\theta_{\beta(\alpha)}$ and 
$\varphi_{\alpha\beta}$ induces a covering map $S^{-1}_{\alpha,i,j}(0)\rightarrow S^{-1}_{\beta,i,j}(0)$.

\begin{lemma}\label{lem:multicomp}
	There are systems of CF-perturbations $(\mathcal{ U}_\alpha, W_\alpha, S_\alpha)_{\alpha\in I}$ on the moduli spaces $\mathcal{M}^{para}_{k+1,l+1}(\beta, \mathbf{Q}, \tau_{tw}^l,\gamma)$, whose 
	restriction to the boundary agrees with the fiber product of CF-perturbations on the right-hand side of (\refeq{eq:boundpara}). Moreover they are compatible with $\pi$ and the evaluation at the 0-th boundary marked point maps $(ev_0)_\alpha\vert_{S^{-1}_\alpha(0)}$ are submersions.
\end{lemma}
\begin{proof}
	Once again the proof follows the strategy in \cite{F10} and \cite[Proposition 4.4]{Amo}. We take the CF-perturbations on $\mathcal{M}_{k+1,l}(\beta, \tau_{tw}^l,\gamma)$  and define $W_\alpha=W_{\beta(\alpha)}$, $\theta_\alpha=\theta_{\beta(\alpha)}$. On the $\varphi_{\alpha \beta}^*E_\beta $ component we take the map $S_\alpha= S_\beta \circ \varphi_{\alpha \beta}$. On the $N$ component we take a generic perturbation of $\mathbf{Q}$ so that it becomes transversal to the image of $ev_{z_1^+}$. With this choice, $\varphi_{\alpha\beta}$ induces a natural covering map $S^{-1}_{\alpha,i,j}(0)\rightarrow S^{-1}_{\beta,i,j}(0)$. Please note that if consider the union of all $\alpha$ over a fixed $\beta$ the resulting covering map has exactly $\mathbf{Q}\cap \beta$ sheets.
	The remainder of the proof follows the usual induction argument on energy, see \cite{F10}.
\end{proof}

Equipped with these CF-perturbations on $\mathcal{M}^{para}_{k+1,l+1}(\beta, \mathbf{Q}, \tau_{tw}^l,\gamma)$ we can use the evaluation maps at the boundary marked points to define new operations on the Fukaya algebra and set 
$$F_{b,\tau} = \sum_{k,\beta, l} \frac{T^{\omega \cap \beta}}{l!} (ev_0)_* 
(ev_1^*b \wedge \cdots \wedge ev_k^*b).$$

Now we apply the Stokes theorem \cite[Lemma 12.13]{FOOO2} to $\mathcal{M}^{para}_{k+1,l+1}(\beta, \mathbf{Q}, \tau_{tw}^l,\gamma)$. Please note that the terms coming from the second line in (\ref{eq:boundpara}) contribute with zero since $\m_{1}^{\tau, b}$ is unital. Also the terms in the first line of (\ref{eq:boundpara}) give respectively the left and right-hand side of Equation (\ref{eq:divisor}), by the previous proof. Therefore we conclude that $\m_1^{\tau,b}(F_{b,\tau} )$ is exactly the difference between the two side of (\ref{eq:divisor}). Combining this with Proposition \ref{prop:Imdiff} proves the desired statement

The proof of (\ref{eq:unital}) is entirely analogous. The main difference is that in this case when describing compatibility with $\pi$ there is no extra component $N$ in the obstruction bundle. Therefore, with the exception of the case $\beta=0$ and $k=l=0$, the induced map $S^{-1}_{\alpha,i,j}(0)\rightarrow S^{-1}_{\beta,i,j}(0)$ is a submersion of positive codimension and the corresponding operation gives zero (see \cite[Proposition 3.7]{Amo}). When $\beta=0$ and $k=l=0$, it is easy to see that we obtain $\one_\BL$. This proves that the two sides of (\ref{eq:unital}) differ by a $\m_1^{\tau,b}$ coboundary which proves the result.

\section{Proof of Theorem \ref{as:def} (2)}\label{sec:thmasdefpfpf}

We give a proof of Theorem \ref{as:def} (2), here. 
First observe that we have the following basic relations in $\mathcal{A}_0$ from the ideal $\langle g_1, g_2, g_3 \rangle$:
$$a T^8x^{a-1} = yz, \quad b T^8y^{b-1} = zx,  \quad c T^8z^{c-1} = xy.$$
We will refer to these as the Jacobi relations.
\begin{defn}
Given a monomial in $\Lambda \langle\langle x,y,z\rangle\rangle$, the operation replacing $yz, zx, xy$ in the monomial by $T^8x^{a-1}, T^8y^{b-1}, T^8z^{c-1}$ will be referred to as {\em type I replacement}, 
and  replacing  $x^{a-1}, y^{b-1}, z^{c-1}$ by $T^{-8}yz, T^{-8}zx, T^{-8}xy$  will be referred to as {\em type II replacement}.
Each of individual replacements (as well as their corresponding relations in $\langle g_1, g_2, g_2 \rangle$) will be called by $\I_x, \I_y, \I_z, \II_x, \II_y, \II_z$, respectively. 
\end{defn}
Hence, if we perform type $\I$ replacement $a$-times and type $\II$ replacement $b$ times, then
the the original exponent of $T$ is increased by $8a - 8b$.
We will use the following properties of $\langle g_1, g_2,g_3 \rangle$ frequently.
\begin{lemma}\label{jacl1}
If an expression $$x^{p}y^qz^r  - T^m x^{p+i}y^{q+j}z^{r+k}$$ lies in the ideal $\langle g_1, g_2,g_3 \rangle$
with $p,q,r,i,j,k \geq 0$ and $m >0$, then
so does $x^{p}y^qz^r$ itself.
\end{lemma}
\begin{proof}
We have $$x^{p}y^qz^r  - T^m x^{p+i}y^{q+j}z^{r+k} = x^{p}y^qz^r (1 - T^m x^i y^j z^k)$$
and $(1 - T^m x^i y^j z^k)$ is invertible in $\Lambda \langle\langle x,y,z \rangle\rangle$. Hence the lemma follows.
\end{proof}

\begin{lemma}\label{jacl2}
For $p,q,r, i,j,k \geq 0$, if an expression $x^{p}y^qz^r$ transforms to another expression $x^iy^jz^k$ by
 performing type $\I$ or $\II$ replacements
 then their difference lies in the ideal $\langle g_1, g_2, g_3 \rangle$. i.e. $$x^{p}y^qz^r - x^iy^jz^k = \sum_{j=1}^3 t_jg_j$$
for some $t_j$ with $val(t_j) \geq s$, where $s$ is the minimum valuation of the intermediate expressions (including both ends of the operation sequence).
\end{lemma}

\begin{proof}
It directly follows from the fact that both of replacements are trivial modulo relations in the ideal $\langle g_1, g_2, g_3 \rangle$.
\end{proof}

Let us now begin the proof of Theorem \ref{as:def} (2). 
We divide the proof into a few different cases (Lemma \ref{lem:(2)abc3}, \ref{lem:(2)abc22}, \ref{lem:(2)abc23} and \ref{lem:(2)abc24g} below) depending on the type of $(a,b,c)$.

\begin{lemma}\label{lem:(2)abc3}
	If $a,b,c \geq 3$, then Theorem \ref{as:def} (2) holds.
\end{lemma}
\begin{proof}
	Consider a monomial $x^{i'}y^{j'}z^{k'}$ (with $i',j',k' \geq 0$ ) which does not appear in the basis.
	By symmetry, we may assume that $i' \geq j' \geq k'$.
	First we consider the case that $k' \neq 0$. Then $i' \geq 2$ since otherwise $x^{i'}y^{j'}z^{k'}= xyz$ is an element of the basis.
	Thus we can write $i' = i+2, j' = j+1, k'=k$ (with $i,j \geq 0, k \geq 1$).
	By using Jacobian relation $\I_z, \I_y, \I_x$ successively, we have
	\begin{equation*}
	\begin{array}{ccccc}
	
	x^{2+i}y^{1+j}z^k &\equiv& c \, x^i y^j z^k ( x (T^{8}z^{c-1})) &\equiv& c \, x^i y^j z^{k+c-2} T^8 (zx) 
	  \\
		&\equiv& (bc) \, x^i y^j z^{k+c-2} T^{16} y^{b-1}  &\equiv & (bc) x^i y^{j+b-2} z^{k+c-3} T^{16} yz \\
		&\equiv & (abc) x^{i+a-1} y^{j+b-2} z^{k+c-3} T^{24} &&  
	\end{array}
	\end{equation*}
	Hence the difference
	$$x^{2+i}y^{1+j}z^k - (abc) x^{i+a-1} y^{j+b-2} z^{k+c-3} T^{24}  = x^{2+i}y^{1+j}z^k( 1 - (abc) x^{a-3}y^{b-3}z^{c-3}T^{24})$$
	lies in the ideal. Lemma \ref{jacl2} tells us that the term in the right hand side  lies in $\langle g_1, g_2, g_3 \rangle$. Therefore $x^{i'} y^{j'} z^{k'}(=x^{2+i}y^{1+j}z^k)$ with $k' \neq 0$ belongs to $\langle g_1, g_2, g_3 \rangle$ by Lemma \ref{jacl1}.
	
	Now, let us consider  the case for $k'=0$.
	If $i' \geq 2, j' \geq 1$, then we can apply exactly the same argument as above to prove that $x^{i'} y^{j'}$ lies in the ideal.
	If $i'=1,j'=1$, then we have $x^{i'} y^{j'}= xy \equiv c T^{8}z^{c-1} \subset c T^{8} \{\gamma_1, \ldots, \gamma_N\}$ and thus the claim holds.
	
	We are left with the case when $j'=k'=0$ and $i' \geq a$.
	If $i'=a$, then $ x^a = \frac{1}{a} T^{-8} xyz$ and hence the claim still holds.
	If $i' = a+i$ with $i \geq 1$, then
	$$x^{i'} = x^{a+i} \equiv (1/a) T^{-8}xyz \cdot x^{i} = (1/a) T^{-8}x^{i+1}yz$$
	We have already shown that $x^{i+1}yz$ is an element in the ideal which can be written as $\sum_{j=1}^n t_j g_j$
	with $val(t_j) \geq 0$. 
	Therefore, 
	$$x^{i'} = (1/a) T^{-8} x^{i+1} ( \overbrace{a T^{8} x^{a-1} -  yz}^{g_1} ) + T^{-8} \sum_{j=1}^3 t_j g_j =  \sum_{j=1}^3 t_j' g_j$$
	with $val(t_j') \geq -8$.
	\end{proof}
\begin{lemma}\label{lem:(2)abc22}
	If $(a,b,c) = (2,2,c)$ $(c \geq 2)$, then Theorem \ref{as:def} (2) holds.
\end{lemma}
\begin{proof}
 For simplicity, we represent the monomial $x^iy^jz^k$ by its exponent vector $(i,j,k)$ in what follows.
We separate them into the following cases depending on which entries of the vector vanish. Below, $i,j,k$ are all assumed to be positive integers.
  \begin{itemize}
 \item[$(i,0,k)$:] By first applying $\I_y$ and later $\I_x$ (or $\I_z$), we can make it into $(i,0,k-2)$: (or $(i-2,0,k)$)
Repeating the procedure, we can reduce it to one of the basis element (by type $\I$ replacements only).
 \item[$(0,j,k)$:] This follows from the previous case by symmetry of $(2,2,n)$. Again, we only need type $\I$ replacements.
 \item[$(0,j,0)$:]  For the case of $j=2$, $(0,2,0) \sim^{\II_y} (1,1,1)$ which becomes a basis element.
For the case of $j \geq 3$, we first apply $\II_y$ to get $(1,j-1,1)$, followed by $\I_z$ to get $(0,j-2,c)$. Since we have applied
 each of type $\I$ and $\II$ exactly once, the exponent of $T$ remains zero, and we can now apply the previous case of $(0,j,k)$. The same argument can be used for $(i,0,0)$.
 \item[$(0,0,k)$:] For the case of $k=c$,  $(0,0,c) \sim^{\II_z} (1,1,1)$ which becomes a basis element.
 For the case of $k \geq c+1$, we may take $$(0,0,k) \sim^{\II_z} (1,1,k-c+1) \sim^{\I_y} (0,2,k-c)$$
to go back to one of the previous cases.
\item[$(i,j,0)$:] For the case that $i=1$, we have $(1,j,0) \sim^{\I_z} (0,j-1,c-1)$, and the latter has been already covered.
Let us consider the case that $i \geq 2$. We have $(i,j,0) \sim^{\I_z} (i-1,j-1,c-1) \sim^{\I_y} (i-2,j,c-2)$.
We can then apply $I_z$ as many times as needed to turn them into  $(*,0,*)$ or $(0,*,*)$, which is already covered.
\item[$(i,j,k)$:] We use induction on $i+j+k$ and 
$(i,j,k) \sim^{\I_x} (i+1,j-1,k-1)$. We can apply either induction hypothesis to $(i+1,j-1,k-1)$ if all entries are non-zero or one of the above steps otherwise.
 \end{itemize}
 
%
%
%
%
%
%
\end{proof}
\begin{lemma}\label{lem:(2)abc23}
	If $(a,b,c) =(2,3,c)$, then Theorem \ref{as:def} (2) holds. 
	\end{lemma}

\begin{proof}
This is the most elaborate case. We claim that a given type of monomial is either equivalent to a basis element or to zero element modulo $\langle g_1,g_2,g_3 \rangle$ by applying type $\I$ and $\II$ replacements. Since we also need to control the valuation of $t_j$ in \eqref{eqn:tjtjtj}, the type $\II$ replacement should be applied carefully. It will be always coupled with the type $\I$ to compensate the energy. 
Here, we only consider $c\geq 3$ since the case with $c=2$ has already been covered by Lemma \ref{lem:(2)abc22}. Again, $i,j,k$ are all assumed to be positive integers.
	
	\begin{enumerate}
		\item[$(0,j,k)$:] We further divide the case into the following.
		\begin{enumerate}
		\item[(i)] $j \leq 2$:  The lowest possibly non-basis element is $(0,1,2)$, and since $(0,1,2) \sim^{\I_x} (1,0,1)$, it is equivalent to a basis element. 
		Now for $k \geq 3$, observe that $(0,1,k) \sim^{\I_x} (1,0,k-1) \sim^{\I_y} (0,2,k-2)$. Thus it suffices to consider
	   $(0,2,k)$ for $k \geq 1$, for which we have
		 $$(0,2,k) \sim^{\I_x} (1,1,k-1) \sim^{\I_x} (2,0,k-2) \sim^{\I_y} (1,2,k-3) $$
		 $$  \sim^{\I_z} ( 0,1,k+c-4) \sim^{\I_x} (1,0,k+c-5) \sim^{\I_y} (0,2,k+c-6).$$
		(For $k=1,2$ case, we stop at 2nd and 3rd equality.)
		If $c \geq 6$, by applying Lemma \ref{jacl1} to the first and the last term, we obtain the claim.		
		If $c=3,4$ or $5$, $ (0,2,k) \sim (0,2,k-3)$, $(0,2,k) \sim (0,2, k-2)$ or $(0,2,k) \sim (0,2, k-1)$. In any case, we can reduce it to either of $(0,2,0), (0,2,1), (0,2,2)$ which were covered in the first step.
		Note that we only use type $\I$ in this case.
		
		\item[(ii)] $j \geq 3$:
		 Consider the case that $j=3$. First, $(0,3,0) \sim^{II_y} (1,1,1)$ which is one of the basis element.
		 For $(0,3,k)$ with $k >0$, we apply $\I_x$ and $\I_z$ to obtain
 $$ (0,3,k) \sim^{\I_x} (1,2,k-1) \sim^{\I_z} (0,1,k+c-2)$$		 
		The last one belongs to the previous case (of $j=1$), and we are done.
		
		 Consider $j \geq 4$. The same argument as above shows that 
		$(0,j,k) \sim (0,j,k+c-6)$ and for $c \geq 6$, this shows the vanishing of the monomial modulo the relations by Lemma \ref{jacl1}.
		Thus it is enough to consider the case that $3 \leq c \leq 5$.
		For $c=3$, we run an induction on $j$ to get $(0,j,0), (0,j,1), (0,j,2)$ as in (i). Finally,
		$$(0,j,0) \sim^{\II_y} (1,j-2,1) \sim^{\I_z} (0,j-3,c),$$
		$$(0,j,1) \sim^{\II_y} (1,j-2,2) \sim^{\I_z} (0,j-3,c+1),$$
		$$(0,j,2) \sim^{\II_y} (1,j-2,3) \sim^{\I_z} (0,j-3,c+2),$$
and inductively we can go back to the case  of $j \leq 3$. The other cases $c=4,5$ are similar.
		Note that we sometimes used type $\II$ exactly once, but immediately followed by type $\I$ in this case.
		\end{enumerate}

		\item[$(0,j,0)$:] It can be done as in the last paragraph.
		$ (0,j,0) \sim^{\II_y} (0,j-3,c)$ uses type $\II$, but the latter can be reduced without further energy loss. So this proves the claim.

		\item[$(0,0,k)$:] We can transform it to the first case since
		$$(0,0,k) \sim^{\II_z} (1,1,k-c+1) \sim^{\I_y} (0,3,k-c).$$
		\item[$(i,0,k)$:]
		If $i \geq 3$,
		$$(i,0,k) \sim^{\I_y} (i-1,2,k-1) \sim^{\I_z, \I_z} (i-3,0,a-1+2c-2),$$
		so we can run induction on $i$ to make $i=1$ or $i=2$.
		For $i=1,2$, we have
		$$ (1,0,k) \sim^{\I_y} (0,2, k-1), \quad  (2,0,k) \sim^{\I_y} (1,2, k-1) \sim^{\I_z} (0,1, c+k-2).$$
		\item[$(i,0,0)$:] The claim follows from
		$$(i,0,0) \sim^{\II_x} (i-1,1,1) \sim^{\I_z} (i-2,0,c)$$
		where the last term was covered in the previous step.
		\item[$(i,j,0)$:] Observe that for $c \geq 6$
		$$ (i,j,0) \sim^{\I_z} (i-1,j-1,c-1) \sim^{\I_x} (i,j-2,c-2) \sim^{\I_y} (i-1,j,c-3)$$
		$$ \sim^{\I_x} (i,j-1,c-4) \sim^{\I_x} (i+1,j-2,c-5) \sim^{\I_y} (i,j,c-6).$$
		Thus, for $c \geq 6$, we have the vanishing of the monomial modulo $\langle g_1,g_2,g_3 \rangle$ by comparing two ends.
		If $c=3$, then $(i,j,0) \sim^{\I_z,\I_x,\I_y} (i-1,j,c-3)= (i-1,j,0)$.
		Thus we run induction on $i$.
		For $c=4,5$, we similarly run induction on $j$.
		
		\item[$(i,j,k)$:]
		We run induction on $(i+j+k)$ for $(i,j,k)$. Since $(i,j,k) \sim^{\I_x} (i+1,j-1,k-1)$, we can make either $j$ or $k$ vanish by applying this operation repeatedly.
	\end{enumerate}
\end{proof}

\begin{lemma}\label{lem:(2)abc24g}
	If $(a,b,c) =(2,b,c)$ with $b,c \geq 4$, then Theorem \ref{as:def} (2) holds.
\end{lemma}
\begin{proof}
We separate them into the following cases.
	\begin{enumerate}
		\item[$(0,j,k)$:] 
		Note that 
		$$(0,j,k) \sim^{\I_z} (1,j-1,k-1) \sim^{\I_z} (0,j+b-2,k-2)$$
		$$ \sim^{\I_x} (1,j+b-3,k-3) \sim^{\I_z} (0,j+ b-4 ,k+c-4),$$
		and since $b, c \geq 4$, this shows that the monomial $(0,j,k)$ is trivial modulo $\langle g_1,g_2,g_3 \rangle$ if $i \geq 1, j \geq 3$.
		The remaining case can be handled by
		$$(0,2,1) \sim^{\I_x} (1,1,0), \quad (0,j,1) \sim^{\I_x} (1,j-1,0) \sim^{\I_z} (0,j-2,c-1),$$ 
		$$ (0,j,2) \sim^{\I_x} (1,j-1,1) \sim^{\I_z} (0,j-2,c).$$
		\item[$(0,j,0)$:] If $j \leq b$, then it is a basis element, so we only consider $j\geq b+1$. In this case,  we have
		$(0,j,0) \sim^{\II_y} (1,j-b+1,1) \sim^{\I_z} (0,j-b,c)$.
		
		\item[$(0,0,k)$:] We only need to consider $k \geq c+1$ for which
		$(0,0,k) \sim^{\II_z} (1,1,k-c+1) \sim^{\I_y} (0,b,k-c)$.

		\item[$(i,0,k)$:] 
		We run induction on $i$. Observe that
		$(i,0,k)  \sim^{\I_y} (i-1,b-1,k-1)$.
		We repeatedly apply $\I_z$ (which adds $(-1,-1,c)$) until either the first or the second entry become $0$, depending on the relative sizes of $i$ and $b$.
		In the former, we have $(0,*,*)$ which was already covered. In the latter case, we obtain  $(i-b,0,k-1+(b-1)(c-1))$, and hence  we can apply the induction.
		
		\item[$(i,0,0)$:]
We only need to consider $i \geq 3$, in which case $(i,0,0) \sim^{\II_x} (i-1,1,1) \sim^{\I_z} (i-2,0,c)$.
		\item[$(i,j,0)$:] We proceed as 
		$$(i,j,0) \sim^{\I_z} (i-1,j-1,c-1) \sim^{\I_x} (i,j-2,c-2) $$
		$$\sim^{\I_y} (i-1,j+b-3,c-3) \sim^{\I_x} (i,j+b-4,c-4).$$
		Thus we can say that $(i,j,0)$ is equivalent to $0$ if $i\geq1, j \geq 2$.
		Also note that $(2,1,0) \sim^{\I_z} (1,0,c-1)$.
		\item[$(i,j,k)$:] It can be handled by using the induction on $(i+j+k)$ based on the relation $(i,j,k) \sim^{\I_x} (i+1,j-1,k-1)$.
		
	\end{enumerate}
	
\end{proof}

\section{Proof of Proposition \ref{QH->Jac explicit}}\label{sec:pf333orb}
 Proposition \ref{QH->Jac explicit} can be shown by directly counting the contributing orbi-discs, which is tedious, but elementary. Below is a reformulation of  Proposition \ref{QH->Jac explicit} in $\tilde{x},\tilde{y},\tilde{z}$-variables which are more accessible in actual disc counting. In addition, we choose 
 \begin{equation}\label{eqn:unusualbbb}
 b= -\tilde{x} X + \tilde{y} Y - \tilde{z} Z
 \end{equation}
 in order to make the signs in the formula more symmetric.
%
It is not difficult to check that Proposition \ref{QH->Jac explicitre} is equivalent to the original statement in Proposition \ref{QH->Jac explicit}.

\begin{prop} \label{QH->Jac explicitre}
The map $\KS$ from the orbifold quantum cohomology of $X$ to the Jacobian ring of $W(\tilde{x},\tilde{y},\tilde{z})$ defined in \eqref{QH->Jac} is given by
\begin{align*}
\one_X \mapsto & 1; \\
[\pt] \mapsto & \frac{1}{8A} T \dd{T} W (\tilde{x},\tilde{y},\tilde{z});
\end{align*}
\begin{align*}
\Delta^{1/3}_1 \mapsto & \tilde{x} \sum_{k=0}^\infty (-1)^{k} (2k+1) \phi_k(T); \\
\Delta^{1/3}_2 \mapsto & \tilde{y} \sum_{k=0}^\infty (-1)^{k} (2k+1) \phi_k(T); \\
\Delta^{1/3}_3 \mapsto & \tilde{z} \sum_{k=0}^\infty (-1)^{k} (2k+1) \phi_k(T); \\
\end{align*}

\begin{align*}
\Delta^{2/3}_1 \mapsto &
\tilde{x}^2 \sum_{k=0}^\infty (2k+1) \phi_k(T^2) 
+ \tilde{x}^2 \sum_{k=1}^\infty \sum_{i=0}^{k-1} (-1)^{3k-i} (6k-2i+2)  \frac{\phi_k(T^3) }{\phi_i (T)} \\
& + \tilde{y}\tilde{z} \sum_{k=1}^\infty \sum_{i=0}^{k-1} \left( (-1)^{3k-i} (6k-2i) \frac{\psi_k^{+} (T)}{\phi_i (T)} + (-1)^{3k-i-1} (6k-2i-2) \frac{\psi_k^{-} (T)}{\phi_i (T)} \right); \\
\Delta^{2/3}_2 \mapsto &
\tilde{y}^2 \sum_{k=0}^\infty (2k+1) \phi_k(T^2) 
+ \tilde{y}^2 \sum_{k=1}^\infty \sum_{i=0}^{k-1} (-1)^{3k-i} (6k-2i+2)  \frac{\phi_k(T^3) }{\phi_i (T)} \\
& + \tilde{z}\tilde{x} \sum_{k=1}^\infty \sum_{i=0}^{k-1} \left( (-1)^{3k-i} (6k-2i) \frac{\psi_k^{+} (T)}{\phi_i (T)} + (-1)^{3k-i-1} (6k-2i-2) \frac{\psi_k^{-} (T)}{\phi_i (T)} \right); \\
\Delta^{2/3}_3 \mapsto &
\tilde{z}^2 \sum_{k=0}^\infty (2k+1) \phi_k(T^2) 
+ \tilde{z}^2 \sum_{k=1}^\infty \sum_{i=0}^{k-1} (-1)^{3k-i} (6k-2i+2)  \frac{\phi_k(T^3) }{\phi_i (T)} \\
& + \tilde{x}\tilde{y} \sum_{k=1}^\infty \sum_{i=0}^{k-1} \left( (-1)^{3k-i} (6k-2i) \frac{\psi_k^{+} (T)}{\phi_i (T)} + (-1)^{3k-i-1} (6k-2i-2) \frac{\psi_k^{-} (T)}{\phi_i (T)} \right)
\end{align*}
where $\phi_k(T) = T^{12k^2+12k +3}$, $\psi_k^{+} (T) = T^{(6k+1)^2}$, $\psi_k^{-} (T) = T^{(6k-1)^2}$.
\end{prop}
 
\begin{proof} 

$\KS(\one_X) = 1$ by the unital property of $\KS$, and $\KS( [\pt] )$ was already computed in the proof of Theorem \ref{thm:KSofdiv}. 
It only remains to compute the image of $\Delta_1^{i/3}$ for $i=1,2$ (other cases can be calculated in a similar way). 
In the computation below, we will use the Morse model with the combinatorial sign rule following \cite{Se}. For this reason we choose a perfect Morse function on $\mathbb{L}$ with the minimum $e$ which serves as the unit class in $CF(\mathbb{L},\mathbb{L})$. In addition, we choose a generic point which is close to $e$ that represents a nontrivial structure put on $\mathbb{L}$.
Readers may consult \cite[3.4]{CHL17} for more details on the disc counting in this setting.

\noindent(1)  $\KS(\Delta^{1/3}_1)$ : From our earlier lifting argument, the holomorphic triangles counted for the potential $W$ can be
regarded as uniformizing covers of [1/3] orbi-discs which contribute to $\KS(\Delta^{1/3}_1)$, as shown in Figure \ref{fig:onethird}. Comparing with \eqref{eqn:phipsi333}, there are sequences $\Delta_{x,k}$ and $\Delta_{x,k}^{op}$ of such orbi-discs with sizes $\phi_k (T)$, $k=0,1,2,\cdots$. Here, we set $\Delta_{x,k}$ to be a positive triangle, and $\Delta_{x,k}^{op}$ a negative one. 

We also need to count the number of times in which the discs meet the minimum $e$. By direct counting, we see that
for $\Delta_{x,k}$ and $\Delta_{x,k}^{op}$ of size $\phi_k (T)$, there are $k+1$ and $k$ many $e$'s on their boundaries, respectively. Taking signs into account ($s(\Delta_{x,k}) = (-1)^{k+1}$, $s(\Delta_{x,k}^{op}) = (-1)^{|X|} (-1)^k = (-1)^{k+1}$), the element $\Delta^{1/3}_1$ of $QH^\ast_{orb} (X)$ maps to
$$\tilde{x} \sum_{k=0}^\infty (-1)^{k} (2k+1) \phi_k(T) = \phi (T) \tilde{x}$$
as desired.\footnote{Notice that $(-1)^{k+1}$ in $s(\Delta_{x,k}) = s(\Delta_{x,k}^{op})=  (-1)^{k+1}$ turns into $(-1)^k$ due to \eqref{eqn:unusualbbb}.} \\
\begin{figure}
\begin{center}
\includegraphics[height=2.2in]{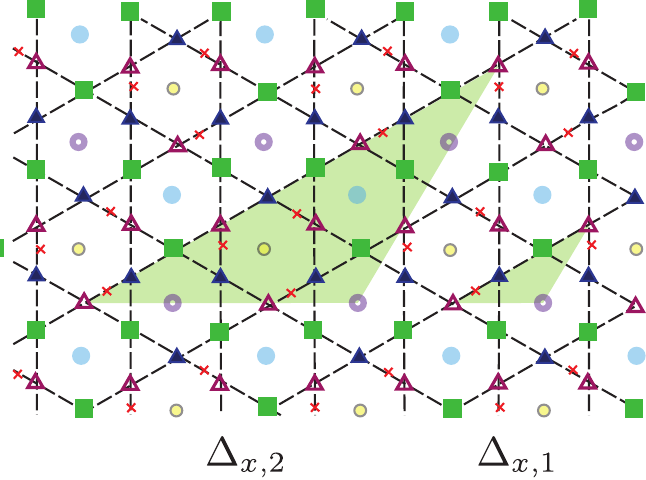}
\caption{$[1/3]$ orbi-discs}\label{fig:onethird}
\end{center}
\end{figure}
\begin{figure}
\begin{center}
\includegraphics[height=2.2in]{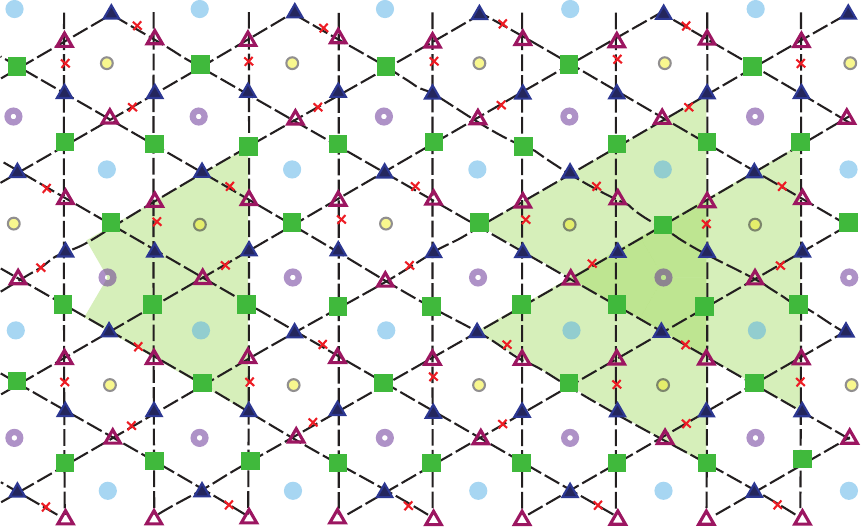}
\caption{$[2/3]$ orbi-discs $\Delta_{xyz,1,+} \setminus \Delta_{x, 1}$ and its tripled image (lifting)}\label{fig:twothird}
\end{center}
\end{figure}

\noindent(2) $\KS(\Delta^{2/3}_1)$ :
From Proposition \ref{prop:orbi-polyclass}, there are two types of such orbi-discs, corresponding to either $\tilde{x}^2$ or $\tilde{y}\tilde{z}$. 
The images of liftings of orbi-discs can be triangles or immersed hexagons as depicted in Figure \ref{fig:twothird}. 
We first consider  the case when the images are triangles, which occurs only for $\tilde{x}^2$-type orbi-discs. Similarly to (1), we have two sequences $\Delta_{x^2,k}$ and $\Delta_{x^2,k}^{op}$ for such discs.
%
Namely, we can take two third of such triangles to get desired orbi-discs, and their count is given by
$$\tilde{x}^2 \sum_{k=0}^\infty (2k+1) \phi_k(T^2).$$
Here, the two triangles $\Delta_{x^2,k}$ and $\Delta_{x^2,k}^{op}$ have the common size $\phi_k (T^2)$. Also, they have $2k+2$ and $2k$ many $e$'s on their boundaries respectively, but because of the rotation symmetry (which gives an automorphism on the moduli) we should count them as $k+1$ and $k$. Signs of contribution are given by
$$s(\Delta_{x^2,k}) = (-1)^{2k+2}=1, \quad s(\Delta_{x^2,k}^{op}) = (-1)^{|X| + |X|} (-1)^{2k} = (-1)^{2k+2} =1.$$

We next consider orbi-discs whose liftings become immersed hexagons.
Again, there are two types of such orbi-discs corresponding to either $\tilde{x}^2$ or $\tilde{y}\tilde{z}$.

\begin{itemize}
\item[(i) $\tilde{x}^2$:]
In this case, we count the orbi-discs $\Delta_{x^3,k} \setminus \Delta_{x, i}$ ($i=0, \cdots, k-1$) of size $\frac{\phi_k (T^3)}{ \phi_i (T)}$, which has $3k+3 - (i+1) = 3k -i +2$  many $e$'s on its boundary,
and $s(\Delta_{x^3,k} \setminus \Delta_{x, i}) = (-1)^{3k - i+2}$. Its reflection image $\left( \Delta_{x^3,k} \setminus \Delta_{x, i} \right)^{op}$ has $3k - i$ many $e$'s on the boundary, and $s \left( \left( \Delta_{x^3,k} \setminus \Delta_{x, i} \right)^{op} \right) =(-1)^{|X| + |X|} (-1)^{3k-i} = (-1)^{3k-i}$. 
In total, they produce 
\begin{equation*}
\begin{array}{l}
\tilde{x}^2\displaystyle\sum_{i=0}^{k-1} \left( (-1)^{3k-i+2} (3k-i+2) + (-1)^{3k-i} (3k-i) \right)  \frac{\phi_k(T^3) }{\phi_i (T)}  \\
=\tilde{x}^2\displaystyle\sum_{i=0}^{k-1} (-1)^{3k-i} (6k-2i+2)  \frac{\phi_k(T^3) }{\phi_i (T)}.
\end{array}
\end{equation*}
%

\item[(ii) $\tilde{y}\tilde{z}$:]
Denote the two positive triangles contributing to the $k$-th terms in \ref{eqn:phipsi333} by $\Delta_{xyz,k,\pm}$. The only contribution to $\tilde{y}\tilde{z}$ is from the count of 
%
$\Delta_{xyz,k,\pm} \setminus \Delta_{x, i}$ ($i=0,1, \cdots, k-1$) and their reflection image, both of which have with size $\psi_k^{\pm} (T)$. 

 $\Delta_{xyz,k,+} \setminus \Delta_{x, i}$ has $3k+1 - (i+1) = 3k-i$ many $e$'s along the boundary, and
$s(\Delta_{xyz,k,+} \setminus \Delta_{x, i}) = (-1)^{3k-i}$. For $\left(\Delta_{xyz,k,+} \setminus \Delta_{x, i} \right)^{op}$, we have $3k-i$ many $e$'s, and $s\left( \left(\Delta_{xyz,k,+} \setminus \Delta_{x, i} \right)^{op} \right) = (-1)^{|Y| + |Z|} (-1)^{3k-i} = (-1)^{3k-i}$. So, these two discs give
$$\tilde{y}\tilde{z}\left( (-1)^{3k-i} (3k-i) + (-1)^{3k-i} (3k-i) \right) \frac{\psi_k^+ (T)}{\phi_i (T)} $$
 Similarly, $\Delta_{xyz,k,-} \setminus \Delta_{x, i}$ and its reflection image contribute
$$\tilde{y}\tilde{z}\left( (-1)^{3k-i-1} (3k-i-1) + (-1)^{3k-i-1} (3k-1-i) \right) \frac{\psi_k^+ (T)}{\phi_i (T)}.$$
 
\end{itemize}

\end{proof}

\bibliographystyle{amsalpha}
\bibliography{geometry}

\end{document}